\documentclass [twoside,reqno,12pt] {amsart}
\usepackage[left=1in,right=1in,top=1in,bottom=1in]{geometry}

\usepackage[hidelinks]{hyperref}
\usepackage{amsfonts}
\usepackage{amssymb}
\usepackage{color}
\usepackage{graphics}
\usepackage{comment}

\newtheorem{thm}{Theorem}[section]

\newtheorem{lem}[thm]{Lemma}
\newtheorem{prop}[thm]{Proposition}

\newtheorem{defn}[thm]{Definition}
\newtheorem{rem}[thm]{Remark}

\theoremstyle{definition}

\numberwithin{equation}{section}

\renewcommand{\Re}{\mathrm{Re}}
\renewcommand{\Im}{\mathrm{Im}}

\newcommand{\C}{\mathbb{C}}

\renewcommand{\div}{\operatorname{div}}

\newcommand{\n}[1]{\langle #1 \rangle}
\newcommand{\N}{\mathbb{N}}

\newcommand{\R}{\mathbb{R}}

\newcommand{\pM}{{\partial M}}
\newcommand{\tr}{\hbox{tr}\,}
\newcommand{\scl}{\mathrm{scl}}

\def\hat{\widehat}
\def\tilde{\widetilde}
\def \bfo {\begin {eqnarray*} }
\def \efo {\end {eqnarray*} }
\def \ba {\begin {eqnarray*} }
\def \ea {\end {eqnarray*} }
\def \beq {\begin {eqnarray}}
\def \eeq {\end {eqnarray}}
\def \supp {\hbox{supp }}

\def \det {\hbox{det}}

\def \p {\partial}

\newcommand{\confm}{c^{-\frac{n-2}{2}}}

\newcommand{\openbdy}{(0, T)\times \p M}

\newcommand{\tT}{\tilde{t}}
\newcommand{\conf}{c^{-\frac{n-2}{4}}}

\def\hat{\widehat}
\def\tilde{\widetilde}
\def \bfo {\begin {eqnarray*} }
\def \efo {\end {eqnarray*} }
\def \ba {\begin {eqnarray*} }
\def \ea {\end {eqnarray*} }
\def \beq {\begin {eqnarray}}
\def \eeq {\end {eqnarray}}
\def \supp {\hbox{supp }}

\def \det {\hbox{det}}

\def \p {\partial}

\usepackage{color,comment}
\newcommand{\cF}{\mathcal{F}}
\newcommand{\cO}{\mathcal{O}}

\begin{document}
\title[Partial data inverse problem for hyperbolic equation]
{Partial data inverse problem for hyperbolic equation with time-dependent damping coefficient and potential}

\author[Liu]{Boya Liu}
\address{B. Liu, Department of Mathematics\\
North Carolina State University, Raleigh\\ 
NC 27695, USA}
\email{bliu35@ncsu.edu}

\author[Saksala]{Teemu Saksala}
\address{T. Saksala, Department of Mathematics\\
North Carolina State University, Raleigh\\ 
NC 27695, USA}
\email{tssaksal@ncsu.edu}

\author[Yan]{Lili Yan}
\address{L. Yan, School of Mathematics, University of Minnesota, Minneapolis, MN 55455, USA}
\email{lyan@umn.edu}

\begin{abstract}
We study an inverse problem of determining a time-dependent damping coefficient and potential appearing in the wave equation in a compact Riemannian manifold of dimension three or higher. More specifically, we are concerned with the case of conformally transversally anisotropic manifolds, or in other words, compact Riemannian manifolds with boundary conformally embedded in a product of the Euclidean line and a transversal manifold. With an additional assumption of the attenuated geodesic ray transform being injective on the transversal manifold, we prove that the knowledge of a certain partial Cauchy data set determines the time-dependent damping coefficient and potential uniquely.
\end{abstract}

\maketitle

\section{Introduction and Statement of Results}
\label{sec:intro}



This paper is devoted to an inverse problem of a hyperbolic initial boundary value problem, with the aim of determining lower order time-dependent perturbations, namely, a scalar-valued damping coefficient and potential of a Riemannian wave operator, from a set of partial Cauchy data. As introduced in \cite{Kian_damping}, from the physical point of view, this inverse problem is concerned with determining properties such as the time-evolving damping force and the density of an inhomogeneous medium by probing the medium with disturbances generated on the lateral boundary and at the initial time, and by measuring the response at the end of the experiment as well as on some part of the lateral boundary. 

To state the inverse problem considered in this paper, let $(M, g)$ be a smooth, compact, oriented Riemannian manifold of dimension $n \ge 3$ with smooth boundary $\p M$. We denote 
the spacetime $Q =(0,T)\times M^{\mathrm{int}}$ with $0<T<\infty$, $\overline{Q}$ the closure of $Q$, and $\Sigma = \openbdy$ the lateral boundary of $Q$. Recall that the Laplace-Beltrami operator $\Delta_g$ of the metric $g$ acts on  $C^2$-smooth functions according to the following expression in local coordinates $x_1,\ldots,x_n$ of the manifold $M$:
\[
\Delta_g v(x)=|g|^{-1/2}(x)\p_{x^j}\left(g^{jk}(x)|g(x)|^{1/2}\p_{x^k}v(x)\right), \quad x \in M.
\]
Here $|g|$ and $g^{jk}$ denote the absolute value of the determinant and the inverse of $g_{jk}$, respectively. 

For a given smooth and strictly positive function $c(x)$ on $M$, we consider the wave operator
\begin{equation}
\label{eq:wave_conformal}
\Box_{c,g}=c(x)^{-1}\p_t^2-\Delta_g,
\end{equation}
whose coefficients are time-independent. In this paper we study an inverse problem for the following linear hyperbolic partial differential operator
\begin{equation}
\label{eq:damping_conformal}
\mathcal{L}_{c,g,a,q}=\Box_{c,g}+a(t,x)\p_t+q(t, x), \quad (t,x)\in Q,
\end{equation}
with time-dependent lower order coefficients $a\in W^{1,\infty}(Q)$ (the \textit{damping coefficient}) and $q\in C(\overline{Q})$ (the \textit{potential}).
%
%
Our first geometric assumption is the following.
\begin{defn}
\label{def:CTA_manifolds}
A Riemannian manifold $(M,g)$ of dimension $n\ge 3$ with boundary $\p M$ is called conformally transversally anisotropic (CTA) if $M$ is a compact subset of a manifold $\R\times M_0^{\text{int}}$ 
with smooth boundary and nonempty interior, and $g= c(e \oplus g_0)$. Here $(\R,e)$ is the real line, $(M_0,g_0)$ is a smooth compact $(n-1)$-dimensional Riemannian manifold with smooth boundary, called the transversal manifold, and $c\in C^\infty(\R\times M_0)$ is a strictly positive function. 
\end{defn}
Examples of CTA manifolds include precompact smooth proper subsets of Euclidean, spherical, and hyperbolic spaces. We refer readers to \cite{Ferreira_Kur_Las_Salo} for more examples. 
Since the manifold $M$ is embedded into the product manifold $\R \times M_{0}^{\mathrm{int}}$, we can write every point $x\in M$ in the form $x=(x_1,x')$, where $x_1 \in \R$ and $x'\in M_0$. 
In particular, the projection $\varphi(x) = x_1$ is a \textit{limiting Carleman weight}. It was established in \cite[Theorem 1.2]{Ferreira_Kenig_Salo_Uhlmann} that the existence of a limiting Carleman weight implies that a conformal multiple of the metric $g$ admits a parallel unit vector field, and the converse holds for simply connected manifolds. Locally, the latter condition is equivalent to the fact that the manifold $(M,g)$ is conformal to the product of an interval and some Riemannian manifold $(M_0,g_0)$ of one dimension less. 

The limiting Carleman weight $\varphi$ gives us a canonical way to define the front and back faces of $\p M$ and $\p Q$. Let $\nu$ be the outward unit normal vector to $\pM$ with respect to the metric $g$. We denote $\p M_\pm = \{x \in \p M: \pm \p_\nu\varphi(x) \ge 0\}$ and $\Sigma_\pm=(0, T)\times \p M_\pm^{\text{int}}$. 
Then we define  $U=(0, T)\times U'$ and $V=(0,T)\times V'$, where $U',V'\subset \p M$ are open neighborhoods of $\p M_+$, $\p M_-$, respectively. 

The goal of this paper is to show that the time-dependent damping coefficient $a(t,x)$ and potential $q(t,x)$, appearing in \eqref{eq:damping_conformal}, can be uniquely determined from the following set of partial Cauchy data:
\begin{equation}
\label{eq:Cauchy_data}
\mathcal{C}_{g, a,q} = \{(u|_\Sigma, u|_{t=0}, u|_{t=T}, \p_tu|_{t=0},\p_\nu u|_V): \:u\in H^1(0,T;L^2(M)), \: \mathcal{L}_{c,g, a,q}u=0\}.
\end{equation}
We will define these data carefully in Section \ref{sec:direct_problem}. 


Notice that in addition to the data measured on the lateral boundary, the set of Cauchy data $\mathcal{C}_{g, a,q}$ also includes measurements made at the initial time $t=0$ and the end time $t=T$. Indeed, it was established in \cite{Isakov_91} that the full lateral boundary data with vanishing initial conditions,
\begin{equation}
\label{eq:Lat_bound_data}
\mathcal{C}_{g,a,q}^{\mathrm{lat}}=\{(u|_\Sigma, \p_\nu u|_\Sigma): u\in H^1(0, T; L^2(M)), \: \mathcal{L}_{g,a,q}u=0, \: u|_{t=0}=\p_tu|_{t=0}=0\}, 
\end{equation}
determines time-independent damping coefficients and potentials uniquely for $T>\mathrm{diam}(M)$, where $M$ is a bounded domain in $\R^n$.  However, due to domain of dependence arguments, as explained for instance in \cite[Subsection 1.1]{Kian_partial_data}, it is only possible to recover a general time-dependent coefficient in the optimal set
\[
\mathcal{D}=\{(t, x)\in Q:  \mathrm{dist}(x, \p M)<t<T- \mathrm{dist}(x, \p M)\}
\]
from $\mathcal{C}_{g,a,q}^{\mathrm{lat}}$. Thus, even for large measurement time $T>0$, a global recovery of time-dependent lower order coefficients of the hyperbolic operator \eqref{eq:damping_conformal} needs some additional information 
at the beginning $\{t=0\}$ and at the end $\{t=T\}$ of the measurement.

Unfortunately, the product structure of the ambient space $\R \times M_0$ of the manifold $(M,g)$ is not quite sufficient for the recovery method presented in this paper. We need to also assume that certain geodesic ray transforms on the transversal manifold $(M_0,g_0)$ are injective. Such assumptions have been successfully implemented to solve many important inverse problems on CTA manifolds, see for instance \cite{Cekic,Ferreira_Kur_Las_Salo,Krupchyk_Uhlmann_magschr,Yan} and the references therein.

Let us now recall some definitions related to geodesic ray transforms on Riemannian manifolds with boundary.
Geodesics of $(M_0,g_0)$ can be parametrized (non-uniquely) by points on the unit sphere bundle $SM_0 = \{(x, \xi) \in TM_0: |\xi|=1\}$. Moreover, we use the notation
\[
\p_\pm SM_0 = \{(x, \xi) \in SM_0: x \in \p M_0, \pm \langle \xi, \nu(x) \rangle > 0\}
\]
for the incoming (--) and outgoing (+) boundaries of $SM_0$. These sets correspond to the geodesics touching the boundary, and $\langle \cdot, \cdot \rangle$ is the Riemannian inner product of $(M_0,g_0)$.

Let $(x, \xi) \in \p_-SM_0$, and let $\gamma = \gamma_{x, \xi}$ be a geodesic of $M_0$ with initial conditions $\gamma(0) = x$ and $\dot{\gamma} (0) = \xi$. Then $\tau_{\mathrm{exit}}(x, \xi)>0$ stands for the first time when $\gamma$ meets $\p M_0$ with the convention that $\tau_{\mathrm{exit}}(x, \xi) = +\infty$ if $\gamma$ stays in the interior of $M_0$. We say that a unit speed geodesic segment $\gamma: [0, \tau_{\mathrm{exit}}(x, \xi)] \to M_0$, $0<\tau_{\mathrm{exit}}(x, \xi)<\infty$, is \textit{nontangential} if $\gamma(0)$,  $\gamma(\tau_{\mathrm{exit}}(x, \xi)) \in \pM_0$, $\dot{\gamma}(0)$ and $\dot{\gamma}(\tau_{\mathrm{exit}}(x, \xi))$ are nontangential vectors to $\p M_0$, and $\gamma(\tau)\in M_0^\mathrm{int}$ for all $0<\tau<\tau_{\mathrm{exit}}(x, \xi)$.

In this paper we shall reduce the determination of unknown time-dependent coefficients $a(x,t)$ and $q(x,t)$ from the set of partial Cauchy data \eqref{eq:Cauchy_data} to the invertibility of the \textit{attenuated geodesic ray transform} on the transversal manifold $(M_0,g_0)$. Given a smooth function $\alpha$ on $M_0$, the attenuated geodesic ray transform of a function $f\colon M_0 \to \R$ is given by
\begin{equation}
\label{eq:geo_trans_def}
I^\alpha(f)(x, \xi) = \int_0^{\tau_{\mathrm{exit}}(x, \xi)}  \exp\bigg[\int_{0}^{t}\alpha(\gamma_{x, \xi}(s))ds\bigg] f(\gamma_{x, \xi}(t))dt, \quad (x, \xi) \in \p_-SM_0 \setminus \Gamma_-, 
\end{equation}
where $\Gamma_- = \{(x, \xi) \in \p_-SM_0: \tau_{\mathrm{exit}}(x, \xi)= +\infty\}$. Our second geometric assumption is the following.

\textbf{Assumption 1.} 
\label{asu:inj}
There exists $\varepsilon>0$ such that for each smooth function $\alpha$ on $M_0$ with $\|\alpha\|_{L^\infty(M_0)}<\varepsilon$, the respective attenuated geodesic ray transform $I^\alpha$ on $(M_0, g_0)$ is injective over continuous functions $f$ in the sense that if $I^\alpha(f)(x, \xi)=0$ for all $(x, \xi) \in \p_-SM_0 \setminus \Gamma_-$ such that $\gamma_{x, \xi}$ is a nontangential geodesic, then $f=0$ in $M_0$.  

It was verified in \cite[Theorem 7.1]{Ferreira_Kenig_Salo_Uhlmann} that simple manifolds always satisfy Assumption 1.
Traditionally, a compact, simply connected Riemannian manifold with smooth boundary is called \textit{simple} if its boundary is strictly convex, and no geodesic has conjugate points.
Also, the injectivity of the geodesic ray transform ($\alpha=0$) on simple manifolds is well-known, see  \cite{Mukhometov,Sharafutdinov}.

In addition to simple manifolds, there are some other geometric conditions under which the attenuated geodesic ray transform $I^\alpha$ is known to be injective.
For instance, if the manifold is radially symmetric and satisfies the Herglotz condition, then $I^\alpha$ is injective whenever the attenuation $\alpha$ is also radially symmetric and Lipschitz continuous \cite[Theorem 29]{deHoop_Ilmavirta}. For the purposes of the current paper it suffices to only consider constant attenuations. The Herglotz condition is a special case of a manifold satisfying a convex foliation condition, and in \cite{paternain2019geodesic} the injectivity of $I^\alpha$ is verified on this type of manifold of dimension $n\ge 3$. Some examples of manifolds satisfying the global foliation condition are the punctured Euclidean space $\R^n \setminus \{0\}$ and the torus $\mathbb{T}^n$. We refer readers to \cite[Section 2]{paternain2019geodesic} for more examples. Finally, we would like to recall that a convex foliation condition does not forbid the existence of conjugate points. Hence, there are many nonsimple Riemannian manifolds with an invertible attenuated geodesic ray transform.

The main result of this paper is the following.
\begin{thm}
\label{thm:main_result_damping}
Let $T>0$. Suppose that $(M, g)$ is a CTA manifold of dimension $n \ge 3$ and that Assumption \ref{asu:inj} holds for the transversal manifold $(M_0,g_0)$. Let $a_i \in W^{1,\infty}(Q)$ and $q_i \in C(\overline{Q})$, $i=1, 2$. If $a_1=a_2$ and $q_1=q_2$ on $\p Q$,
then $\mathcal{C}_{g, a_1, q_1} = \mathcal{C}_{g, a_2, q_2}$ implies that $a_1=a_2$ and $q_1=q_2$ in $Q$.
\end{thm}

Theorem \ref{thm:main_result_damping} can be viewed as an extension of \cite[Theorem 1.4]{Kian_Oksanen}, where only the potential was considered, to the case of recovering both damping coefficient and potential from the set of partial Cauchy data $\mathcal{C}_{g, a,q}$. From the perspective of a geometric setting, 
this paper extends \cite{Kian_damping} from the Euclidean space, as well as \cite{Kian_Oksanen} from CTA manifolds with a simple transversal manifold, to a larger class of CTA manifolds. 
As in \cite{Kian_Oksanen}, we attack the problem by utilizing tools from the theory of inverse problems for elliptic operators. However, in comparison with the earlier works, we would not be able to relax the simplicity assumption on the transversal manifold without significant modifications to the construction of complex geometric optics (CGO) solutions.

Assumption \ref{asu:inj} of this paper is different from the literature concerning inverse problems for elliptic operators on CTA manifolds; see, for instance, \cite{Ferreira_Kur_Las_Salo,Kenig_Salo, Krupchyk_Uhlmann_magschr}. These works assume the invertibility of the geodesic ray transform $I^\alpha$ for $\alpha=0$. In the case of elliptic operators, where there is only one Euclidean direction $x_1$, the authors reduced the problem first to the attenuated geodesic ray transform. Then the authors recovered the geodesic ray transform for each coefficient in the Taylor expansions of the unknown functions by differentiating an expression similar to (5.21) in our manuscript with respect to the variable $\lambda$ at zero. 
Unfortunately, this approach is not applicable in our case, as the mapping $(\lambda, \beta)\mapsto -\lambda(\beta, 1)$, appearing in (5.21), is a diffeomorphism only if $\lambda \ne 0$. Thus, computing $\lambda$ and $\beta$-derivatives of (5.21) at $\lambda=0$ will not give us the geodesic ray transform of Taylor coefficients of the unknown functions at the origin in the Fourier variables.

\subsection{Previous literature}
\label{ssec:literature} 

The recovery of coefficients appearing in hyperbolic equations from boundary measurements has attracted lots of attention in recent years. Results in this direction are generally divided into two categories with respect to time-independent and time-dependent coefficients. 

Starting with seminal works \cite{Belishev,Belishev_Kurylev}, there has been extensive literature related to the recovery of time-independent coefficients appearing in hyperbolic equations. We refer readers to \cite{Anikonov_Cheng_yama,Bel_Yama,HELIN2018132,Hussein_Lesnic_Yama,LaOk} and references therein for some works in this direction. A powerful tool to prove uniqueness results for time-independent coefficients of hyperbolic equations, including the leading order coefficient, is the boundary control method, which was developed in \cite{Belishev,Belishev_Kurylev}, as well as a time-sharp unique continuation theorem proved in \cite{Tataru}. We refer readers to \cite{Kian_Oksanen_Morancey} for an introduction to the method and \cite{Belishev_07,KKL_book} for reviews. However, it was discovered in \cite{Alinhac,Alinhac_Baoendi} that the unique continuation theorem analogous to \cite{Tataru} may fail when the dependence of coefficients on time is not analytic, which means that the boundary control method is not well-suited to recover time-dependent coefficients in general. 

Aside from the boundary control method, the approach of geometric optic (GO) solutions is also widely utilized to recover time-independent coefficients of hyperbolic equations. Using this approach, the unique recovery of time-independent potential $q$ (with $a=0$) from full lateral boundary Dirichlet-to-Neumann map was established in \cite{Rakesh_Symes}, and \cite{Isakov_91} extended this result to recovery of time-independent damping coefficients using the same boundary data. A uniqueness result from partial boundary measurements was considered in \cite{Eskin_06}. 
The GO solution approach has also been used to obtain stronger stability results \cite{Bel_DDS,Bel_Jel_Yama,Kian_stability_14,Stefanov_Uhlmann} than the boundary control method \cite{And_Kat_Kur_Las_Tay}, but it gives less sharp uniqueness results from the perspective of geometric assumptions than the latter.

Turning the attention to the time-dependent category, most of the results in this direction rely on the use of GO solutions. This approach was first implemented in the context of determining time-dependent coefficients of hyperbolic equations from the knowledge of scattering data by using properties of the light-ray transform \cite{Stefanov}. Recovery of time-dependent potential $q$ from the full lateral boundary data $\mathcal{C}_q^{\mathrm{lat}}$, given by \eqref{eq:Lat_bound_data}, on the infinite cylinder $\R\times \Omega$, where $\Omega$ is a domain in $\R^n$, was established in \cite{Ramm_Sjo}. On a finite cylinder $(0, T)\times \Omega$ with $T>\mathrm{diam}(\Omega)$, it was proved in \cite{Rakesh_Ramm} that $\mathcal{C}_q^{\mathrm{lat}}$ determines $q$ uniquely in the optimal set $\mathcal{D}$ of $(0, T)\times \Omega$.
Uniqueness and stability results for determining a general time-dependent potential $q$ from partial data were established in \cite{Kian_partial_data} and \cite{Bellassoued_Rassas,Kian_stability}, respectively. 

Going beyond the Euclidean space, uniqueness results for time-dependent potential $q$ from both full and partial boundary measurements were established in \cite{Kian_Oksanen} on a CTA manifold $(M, g)$, with a simple transversal manifold $M_0$, by using the GO solution approach. For more general manifolds, recently it was proved in \cite{Alexakis_Feiz_Oksanen_22} that the set of full Cauchy data uniquely determines the potential $q$ in Lorentzian manifolds satisfying certain two-sided curvature bounds and some other geometric assumptions, and this curvature bound was weakened in \cite{Alexakis_Feiz_Oksanen_23} near Minkowski geometry. 
In particular, the proof of \cite{Alexakis_Feiz_Oksanen_22} is based on a new optimal unique continuation theorem and can be viewed as a generalization of the boundary control method to the cases without real analyticity assumptions. 

There is also some literature related to determining time-dependent first order perturbations appearing in hyperbolic equations from boundary data analogous to \eqref{eq:Cauchy_data}. In the Euclidean setting, \cite{Kian_damping} extended the result of \cite{Kian_partial_data} to a unique determination of time-dependent damping coefficients and potentials from $\mathcal{C}_{g,a,q}$. When the vector field perturbation appears in the wave equation, similar to elliptic operators such as the magnetic Schr\"odinger operator, one can only recover the first order perturbation up to a gauge invariance, i.e., the differential of a test function in $Q$, see \cite{Eskin} for a uniqueness result when the dependence of coefficients on time is real-analytic, and this analyticity assumption was removed later in \cite{Salazar}. Logarithmic type stability estimates for the vector field perturbation as well as the potential were proved in \cite{Bellassoued_Aicha}. A uniqueness result analogous to \cite{Eskin, Salazar} from a partial Dirichlet-to-Neumann map was obtained in \cite{Krishnan_Vashisth}.  In the non-Euclidean setting, it is established in \cite{Feizmohammadi_et_all_2019} that the hyperbolic Dirichlet-to-Neumann map determines the first order and the zeroth order perturbations up a gauge invariance on a certain non-optimal subset of $Q$ by inverting the light-ray transform of the Lorentzian metric $-dt^2+g(x)$ for one forms and functions. 
%
%

To summarize, there are only two known methods to recover the coefficients appearing in hyperbolic equations from boundary measurements. Since in the current paper the unknown lower order coefficients of the hyperbolic operator $\mathcal{L}_{c,g,a,q}$, as in \eqref{eq:damping_conformal}, are time-dependent, we cannot apply the boundary control method. Thus, our proofs are based on GO solutions. We believe that an introduction of any new method, which can be used to attack the hyperbolic inverse problems, would be a major breakthrough. Obviously, this is not in the scope of the current paper.

Finally, we would like to emphasize that to the best of our knowledge, the global recovery of a full first order time-dependent perturbation (a one-form and potential function in $Q$) of the Riemannian wave operator from a set of partial Cauchy data, and the optimal recovery of these coefficients from the respective hyperbolic Dirichlet-to-Neumann map, remain important open problems. 

\subsection{Outline for the proof of Theorem \ref{thm:main_result_damping}}
The first main ingredient of the proof is the construction of exponentially growing and decaying CGO solutions to the equation $\mathcal{L}_{c ,g, a, q}u=0$ of the form 
\[
u(t, x)=e^{\pm s(\beta t+\varphi(x))}(v_s(t,x)+r_s(t,x)), \quad (t, x)\in Q.
\]
Here $s=\frac{1}{h}+i\lambda$ is a complex number, $h\in (0,1)$ is a semiclassical parameter, $\lambda\in \R$ and $\beta\in (\frac{1}{\sqrt{3}}, 1)$ are some fixed numbers, $v_s$ is a Gaussian beam quasimode, and $r_s$ is a correction term that decays with respect to the parameter $h$. The function $\varphi(x)=x_1$ is a limiting Carleman weight on $M$. 

We exploit the existence of the limiting Carleman weight $\beta t+x_1$ in $Q$ and derive necessary boundary and interior Carleman estimates, see Proposition \ref{thm:boundary_carleman} and Lemma \ref{prop:shift_index}, respectively. 
The boundary Carleman estimates are used to control the solutions on the inaccessible part of the boundary, while the interior Carleman estimates are needed to verify the existence of the correction term $r_s$ in Proposition \ref{prop:solvability}.
In the current paper the Dirichlet boundary values are given on the full lateral boundary, which is in line with many earlier works involving a first order term, see for instance  \cite{Kian_damping,Krishnan_Vashisth,mishra_determining_2021,senapati_stability_2021}. Meanwhile, in the absence of the damping term, but with Dirichlet data measured only on a part of the lateral boundary, the authors of \cite{Kian_partial_data,Kian_Oksanen} constructed GO solutions to the respective hyperbolic equation that vanish initially and on part of the lateral boundary. In this way the authors were able to utilize the boundary Carleman estimate to control their GO solutions on the inaccessible part of the boundary. Unfortunately, this method only provides an $L^2$-estimate for the correction term $r_s$, and due to the existence of the damping coefficient, we need an $H^1$-estimate for $r_s$. This is provided in Proposition \ref{prop:solvability}.

Since the transversal manifold $(M_0,g_0)$ is not necessarily simple in this paper, the approach based on global GO solutions is not applicable to us. To medicate this, in Theorem \ref{prop:Gaussian_beam} we construct Gaussian beam quasimodes for every nontangential geodesic in the transversal manifold $M_0$ by using techniques developed in solving inverse problems for elliptic operators, see for instance \cite{Cekic,Ferreira_Kur_Las_Salo,Krupchyk_Uhlmann_magschr}, followed by a concentration property for the quasimodes given in 
Theorem \ref{prop:limit_behavior}.  The construction of CGO solutions is finalized in Theorem \ref{prop:CGO_solution}. In this part we need the regularity conditions imposed on the unknown time-dependent coefficients $a$ and $q$.

The second main component in the proof is the integral identity \eqref{eq:int_id_2}, whose derivation needs the equivalence of the partial Cauchy data. When the obtained CGO solutions are inserted in the integral identity, the boundary Carleman estimate of Proposition \ref{prop:Car_est} forces the right-hand side of \eqref{eq:int_id_2}
to vanish in the limit $h\to 0$. On the other hand, the concentration property given in Theorem \ref{prop:limit_behavior} implies that the left-hand side of \eqref{eq:int_id_2} converges to the attenuated geodesic ray transform of the Fourier transform (in the two Euclidean variables $(t,x_1)$) of the unknown coefficients in the transversal manifold $(M_0,g_0)$. To carry on this reduction step, we need the regularity and boundary conditions imposed on the unknown time-dependent coefficients $a$ and $q$. We need Assumption \ref{asu:inj} to invert the attenuated geodesic ray transform. We first provide a proof for the uniqueness result for the damping coefficient $a(t,x)$, followed by verifying the uniqueness for the potential $q(t,x)$.

This paper is organized as follows. We begin by carefully defining the set of partial Cauchy data \eqref{eq:Cauchy_data} in Section \ref{sec:direct_problem}.  In Section \ref{sec:Car_est} we derive the boundary and interior Carleman estimates.
In Section \ref{sec:CGO_solution} we construct the CGO solutions to the hyperbolic equation $\mathcal{L}_{c,g,a,q}u=0$ based on Gaussian beam quasimodes in $Q$. Finally, the proof of Theorem \ref{thm:main_result_damping} is presented in Section \ref{sec:proof_of_theorem}.

\subsection*{Acknowledgments}  
We would like to express our gratitude to Katya Krupchyk and Lauri Oksanen for their valuable discussions and suggestions. T.S. is partially supported by the National Science Foundation (DMS 2204997). L.Y. is partially supported by the National Science Foundation (DMS 2109199). We are also very grateful to the anonymous referees for their invaluable feedback, which led to significant improvements of the paper.

\section{Definition of the Partial Cauchy Data 
}
\label{sec:direct_problem}
The goal of this short section is to recall some properties of the weak solutions to the initial boundary value problem  
\begin{equation}
\label{eq:ibvp}
\begin{cases}
\mathcal{L}_{c ,g, a, q}u(t, x)=0 \quad \text{in} \quad Q,
\\
u(0, x)=h_0(x), \: \p_tu(0, x)=h_1(x) \quad \text{in}\quad M,
\\
u(t, x)=f(t, x) \quad \text{on} \quad \Sigma,
\end{cases}
\end{equation}
as introduced in \cite[Section 2]{Kian_damping}. 

We define the space
\[
H_{\Box_{c,g}}(Q)=\{u\in H^1(0,T; L^2(M)): \Box_{c,g} u=(c^{-1}\p_t^2-\Delta_g)u\in L^2(Q)\},
\]
equipped with the norm
\[
\|u\|^2_{H_{\Box_{c,g}}(Q)}=\|u\|^2_{H^1(0,T; L^2(M))}+\| \Box_{c,g}u\|^2_{L^2(Q)}.
\] 
Our starting point is the following result, originally presented in \cite[Theorem A.1]{Kian_partial_data}. 
\begin{lem}
\label{lem:desity}
The space $ H_{\Box_{c,g}}$ is continuously embedded into the closure of $C^\infty(\overline{Q})$ in the space
\[
K_{\Box_{c,g}}(Q)=\{ u\in H^{-1}(0,T; L^2(M)): \Box_{c,g} u\in L^2(Q)\}
\]
equipped with the norm
\[
\|u\|^2_{K_{\Box_{c,g}}(Q)}=\|u\|^2_{H^{-1}(0,T; L^2(M))}+\| \Box_{c,g}u\|^2_{L^2(Q)}.
\]
\end{lem}

For each $w\in C^\infty(\overline{Q})$, we introduce two linear maps
\[
\iota_0w=(\iota_{0,1}w,\iota_{0,2}w,\iota_{0,3}w)=(w|_\Sigma, w|_{t=0}, \p_t w|_{t=0}) 
\]
and
\[
\iota_1w=(\iota_{1,1}w,\iota_{1,2}w,\iota_{1,3}w)=(\p_\nu w|_\Sigma, w|_{t=T}, \p_t w|_{t=T}).
\]
\begin{lem}
\label{lem:cont_trace}
The maps $\iota_0$ and $\iota_1$ defined above can be extended continuously to 
\[
\iota_0: H_{\Box_{c,g}}(Q)\to H^{-3}(0,T; H^{-1/2}(\p M))\times H^{-2}(M)\times H^{-4}(M)
\]
and 
\[
\iota_1: H_{\Box_{c,g}}(Q)\to H^{-3}(0,T; H^{-3/2}(\p M))\times H^{-2}(M)\times H^{-4}(M),
\]
respectively.
\end{lem}
\begin{proof}
Since the conformal factor $c$ in $\Box_{c,g}$ is time-independent, the proof is a straightforward modification of the proof for  \cite[Proposition A.1]{Kian_partial_data}.
\end{proof}

We note that by the same argument as in \cite[Section 2]{Kian_damping}, the set 
\[
\mathcal{J}=\{u\in H^1(0,T; L^2(M)): \Box_{c,g} u=0\}
\]
is a closed vector subspace of $H^1(0,T; L^2(M))$, contained in $H_{\Box_{c,g}}(Q)$. Finally, we record the range of the map $\iota_0$
\[
\mathcal{K}:=\{\iota_0w: u\in H_{\Box_{c,g}}(Q)\}\subset H^{-3}(0,T; H^{-1/2}(\p M))\times H^{-2}(M)\times H^{-4}(M).
\]
By an analogous argument to the proof for \cite[Proposition 2.1]{Kian_partial_data}, we get the following result.
\begin{lem}
\label{lem:iota_0_bijection}
The linear map $\iota_0\colon \mathcal{J} \to \mathcal{K}$ is a bijection.
\end{lem}
%
By Lemma \ref{lem:iota_0_bijection}, the inverse function $\iota_0^{-1}: \mathcal{K}\to  \mathcal{J}$ exists, and we can use it to define a norm in $\mathcal{K}$ via the formula
\[
\|(f, h_0, h_1)\|_\mathcal{K}=\|\iota_0^{-1}(f,h_0, h_1)\|_{H^1(0,T; L^2(M))}, \quad (f, h_0, h_1)\in \mathcal{K}.
\]
We would like to recall that we have defined $\p M_\pm = \{x \in \p M: \pm \p_\nu\varphi(x) \ge 0\}$ and $V=(0,T)\times V'$,  
where $V'\subset \p M$ is an open neighborhood of $\p M_-$. We are now ready to state and prove the existence and uniqueness of solutions to the initial boundary value problem \eqref{eq:ibvp} with the datum $(f,h_0,h_1)\in \mathcal{K}$.
\begin{prop}
\label{prop:well_posedness}
Let $a\in W^{1,\infty}(Q)$ and $q\in C(\overline{Q})$. For each datum $(f,h_0,h_1)\in \mathcal{K}$, the initial boundary value problem \eqref{eq:ibvp} has a unique weak solution $u\in H^1(0,T; L^2(M))$ that satisfies
\begin{equation}
\label{eq:bound_ibvp_solution}
\|u\|_{H^1(0,T;L^2(M))}\le C\|(f,h_0,h_1)\|_\mathcal{K}.
\end{equation}
Furthermore, the boundary operator 
\begin{equation}
\label{eq:boundary_operator}
\mathcal{B}_{a,q}:\mathcal{K}\to H^{-3}(0,T;H^{-3/2}(V'))\times H^{-2}(M), \quad \mathcal{B}_{a,q}(f,h_0,h_1)=(\iota_{1,1}u|_V, \iota_{1,2}u)
\end{equation}
is bounded, and the partial Cauchy data set $\mathcal{C}_{g,a,q}$, as in \eqref{eq:Cauchy_data}, is the graph of the map $\mathcal{B}_{a,q}$.
\end{prop}
\begin{proof}
The proof is a straightforward modification of the proof of \cite[Proposition 2.1]{Kian_damping}. 
\end{proof}

\section{Carleman Estimates}
\label{sec:Car_est}
Our goal of this section is to prove a boundary Carleman estimate as well as an interior Carleman estimate for the operator $\mathcal{L}_{c,g, a,q}$ conjugated by an exponential weight corresponding to a linear function $\beta t+ x_1$, where $0<\beta<1$ is a constant. We shall utilize the boundary Carleman estimate to control boundary terms over subsets of the boundary $\p Q$ where measurements are not accessible, and the interior Carleman estimates will be used in Section \ref{sec:CGO_solution} to construct the remainder term for both exponentially decaying and growing CGO solutions. 

Let $(M, g)$ be a CTA manifold as defined in Definition \ref{def:CTA_manifolds}, and let $\tilde{g}=e\oplus g_0$. By the conformal properties of the Laplace-Beltrami operator, we have
\begin{equation}
\label{eq:conformal_equivalence}
c^{\frac{n+2}{4}} (-\Delta_g)  (\conf u)= -\Delta_{\tilde{g}}u- \big(c^{\frac{n+2}{4}}\Delta_g(\conf)\big)u,
\end{equation}
see \cite[Section 2]{Ferreira_Kur_Las_Salo}. Also, since $c$ is independent of $t$, we get
\begin{equation}
\label{eq:conf_damp}
c^{\frac{n+2}{4}} a\p_t (\conf u)=ca \p_t u, \quad \text{and} \quad  c^{\frac{n+2}{4}} \p_t^2 (\conf u)=c \p_t^2 u.
\end{equation}
Thus, it follows from \eqref{eq:conformal_equivalence} and \eqref{eq:conf_damp} that for the hyperbolic operator $\mathcal{L}_{c,g,a,q}$, we have
\begin{equation}
\label{eq:equivalence_operator}
c^{\frac{n+2}{4}}\circ \mathcal{L}_{c,g,a,q} \circ \conf = \mathcal{L}_{\tilde{g}, \tilde{a}, \tilde{q}},
\end{equation}
where
\begin{equation}
\label{eq:equiv_coeff}
\tilde{a}=ca, \quad \tilde{q}=c(q-c^{\frac{n-2}{4}}\Delta_g(\conf)).
\end{equation}
Hence, by replacing the metric $g$ and coefficients $a, q$ with $\tilde{g},\tilde{a},\tilde{q}$, respectively, we can assume that the conformal factor $c=1$. In this section we shall make use of this assumption and consider the leading order wave operator $\Box_{e\oplus g_0}=\p_t^2-\Delta_{e\oplus g_0}$. Let us denote $\mathcal{L}_{g,a,q}$ the hyperbolic partial differential operator $\mathcal{L}_{c,g,a,q}$ when $c=1$.

\subsection{Boundary Carleman estimate}

Due to the damping coefficient, we need to use a convexification argument similar to \cite{Kian_damping,Krishnan_Vashisth} to establish the needed boundary Carleman estimate. To elaborate, let us first introduce a new parameter $\varepsilon>0$, which is independent of $h$ and to be determined later. For $0<h< \varepsilon < 1$, we consider the perturbed weight
\begin{equation}
\label{eq:perturbed_weight}
\varphi_{\pm h, \varepsilon}(t,x)= \pm \frac{1}{h}(\beta t+x_1)-\frac{t^2}{2\varepsilon}.
\end{equation}

Our first result in this section can be viewed as an extension of \cite[Theorem 3.1]{Kian_damping} from the Euclidean setting to that of Riemannian manifolds with dependence on a parameter $\beta$. Note that \cite[Theorem 3.1]{Kian_damping} is not directly applicable in our case since the parameter $\beta$ is strictly less than $1$.

\begin{prop}
\label{thm:boundary_carleman}
Let $a, q\in L^\infty(Q, \C)$ and $u\in C^2(\overline{Q})$. 
If $u$ satisfies the conditions
\begin{equation}
\label{eq:condition_u}
u|_\Sigma=u|_{t=0}=\p_tu|_{t=0}=0,
\end{equation}
then for all $0< h \ll \varepsilon\ll 1$ we have 
\begin{equation}
\label{eq:Carleman_est}
\begin{aligned}
&\|e^{-\varphi_{ h, \varepsilon}}h^2\mathcal{L}_{g, a, q} (e^{\varphi_{ h, \varepsilon}}u)\|^2_{L^2(Q)} +(\frac{4}{\beta}-\frac{\beta}{2})h^3\|\nabla_gu(T,\cdot)\|^2_{L^2(M)}+ 3\beta h\|u(T,\cdot)\|^2_{L^2(M)}\\
\ge &\frac{(3\beta^2-1) h^2}{4\varepsilon}\|u\|^2_{L^2(Q)} + \frac{\beta h^3}{4} \|\p_tu(T,\cdot)\|^2_{L^2(M)}+\frac{h^4}{2\varepsilon}(\|\p_tu\|^2_{L^2(Q)}+\|\nabla_gu\|^2_{L^2(Q)})
\\
&+h^3\int_\Sigma \textcolor{black}{\nu_1}|\p_\nu u|^2dS_gdt
\end{aligned}
\end{equation}
and
\begin{equation}
\label{eq:Carleman_est_minus}
\begin{aligned}
&\|e^{-\varphi_{ -h, \varepsilon}}h^2\mathcal{L}_{g, a, q} e^{\varphi_{-h, \varepsilon}}u\|^2_{L^2(Q)} 
+ 2(\beta+1)h^3(\|\nabla_gu(T,\cdot)\|^2_{L^2(M)}
+  \|\p_t u(T,\cdot)\|^2_{L^2(M)})
\\
&\ge \frac{(3\beta^2-1) h^2}{4\varepsilon}\|u\|^2_{L^2(Q)} 
+\frac{h^4}{2\varepsilon}(\|\p_tu\|_{L^2(Q)}
+\|\nabla_gu\|^2_{L^2(Q)})
-h^3\int_\Sigma \nu_1|\p_\nu u|^2dS_gdt,
\end{aligned}
\end{equation}
where $\varphi_{\pm h, \varepsilon}$ is given by \eqref{eq:perturbed_weight}, $\frac{1}{\sqrt{3}}\le\beta < 1$, and $\nu_1:=\n{\nu, \p_{x_1}}_g$. 
\end{prop}
\begin{proof}
We shall only provide a detailed proof for estimate \eqref{eq:Carleman_est}. The derivation of  \eqref{eq:Carleman_est_minus} is analogous and therefore omitted. 
To proceed, we omit the subscripts $h, \varepsilon$ in $\varphi_{h, \varepsilon}$ to simplify the notation. 

\textit{Step 1: The conjugated operator}  $e^{-\varphi}h^2\mathcal{L}_{g, a, q}e^{\varphi}u$. By direct computations, we have
\begin{equation}
\label{eq:conju_op_Car}
\begin{aligned}
e^{-\varphi}h^2\mathcal{L}_{g, a, q}e^{\varphi}u = 
&h^2\left[\p_t^2u+2\p_t\varphi\p_tu +u\p_t^2 \varphi+u(\p_t\varphi)^2 - (\Delta_g u+ 2\n{\nabla_g \varphi, \nabla_gu}_g +u\Delta_g\varphi\right.
\\
&\left. + u|\nabla_g\varphi|^2)+ a\p_t u+ au\p_t\varphi+qu\right]
\\
:=&P_1u+P_2u+P_3u,
\end{aligned}
\end{equation}
where 
\begin{equation}
\label{eq:P_operators}
\begin{aligned}
&P_1u=h^2(\Box_gu+(\p_t\varphi)^2u - |\nabla_g\varphi|^2u+(\Box_g \varphi)u), \quad P_2u=h^2(2\p_t\varphi \p_tu  -2\n{\nabla_g \varphi, \nabla_g u}_g),
\\
&P_3u=h^2(a\p_tu +(a\p_t\varphi)u +qu).
\end{aligned}
\end{equation}
Thus, $P_1u$ includes the even order derivatives of $u$, $P_2u$ contains the odd order derivatives of $u$, and $P_3u$ has all the terms involving $a$ or
$q$. Due to \eqref{eq:conju_op_Car}, \eqref{eq:P_operators}, and the triangle inequality, we have
\begin{equation}
\label{eq:norm_triangle}
\begin{aligned}
\|e^{-\varphi}h^2\mathcal{L}_{g, a,q}e^{\varphi }u\|_{L^2(Q)}^2 
\ge \frac{1}{2}\|P_1u+P_2u\|^2_{L^2(Q)}-\|P_3u\|^2_{L^2(Q)}.
\end{aligned}
\end{equation}
In the last two steps of this proof we shall bound $\|P_1u+P_2u\|^2_{L^2(Q)}$ from below and $\|P_3u\|^2_{L^2(Q)}$ from above. We first use  Cauchy's inequality to convert the first term on the right-hand side of \eqref{eq:norm_triangle} into a product. The choice of the operators $P_1$ and $P_2$ simplifies the subsequent computations, which involve several integration by parts. The estimate for the second term in the right-hand side of \eqref{eq:norm_triangle} is short and is based on the boundedness of $a$ and $q$. Both estimations rely on the choice of the perturbed weight given in \eqref{eq:perturbed_weight}.
Finally, the inequality \eqref{eq:Carleman_est} is obtained by combining these two estimates. 

\textit{Step 2: A lower bound of} $\|P_1u+P_2u\|^2_{L^2(Q)}$. To start, we have by Cauchy's inequality that
\begin{align*}
\frac{1}{2}\|P_1u+P_2u\|^2_{L^2(Q)}\ge \int_Q \Re (P_1u\overline{P_2u})dV_gdt.
\end{align*}
Since $g(x_1,x')=(dx_1)^2+g_0(x')$, we get from direct computations that
\begin{equation}
\label{eq:derivatives_of_varphi}
\p_t \varphi=\frac{1}{h}\beta -\frac{1}{\varepsilon} t, \quad \p_t^2 \varphi= -\frac{1}{\varepsilon}, \quad \p_{x_1} \varphi= \frac{1}{h}, \quad \n{\nabla_g \varphi, \nabla_g u}_g=\frac{1}{h} \p_{x_1}u,
\end{equation}
which yield
\begin{equation}
\label{eq: Car_est_comp}
\begin{aligned}
&\int_Q \Re (P_1u\overline{P_2u})dV_gdt=
\Re \int_Q 
2h^4 \left[\p_t^2u(\frac{1}{h}\beta-\frac{1}{\varepsilon} t)\overline{\p_t u}- \frac{1}{h} \p_t^2u\overline{ \p_{x_1}u}- \Delta_g u(\frac{1}{h}\beta-\frac{1}{\varepsilon} t)\overline{\p_t u}\right.
\\
& 
\left. + \frac{1}{h} \Delta_g u\overline{\p_{x_1}u}-\big(\frac{1}{h^2}(1-\beta^2) +\frac{1}{\varepsilon} +\frac{2\beta}{h\varepsilon}t-\frac{1}{\varepsilon^2}t^2\big)u \big((\frac{1}{h}\beta-\frac{1}{\varepsilon}t)\overline{\p_t u}  -\frac{1}{h}\overline{\p_{x_1}u}\big)\right]
dV_gdt.
\end{aligned}
\end{equation}

Let us proceed to estimate each term on the right-hand side of \eqref{eq: Car_est_comp}. For the first term, we integrate by parts and use the assumption $\p_t u|_{t=0}=0$ to deduce
\begin{equation}
\label{eq:Car_est_term1}
\begin{aligned}
2h^4\Re\int_Q (\frac{1}{h}\beta-\frac{1}{\varepsilon} t ) \p_t^2u \overline{\p_t u} dV_gdt 
&= 
h^4(\frac{1}{h}\beta-\frac{1}{\varepsilon} T)\|\p_tu(T,\cdot)\|^2_{L^2(M)}+\frac{h^4}{\varepsilon}\|\p_tu\|^2_{L^2(Q)}.
\end{aligned}
\end{equation}

Turning attention to the second term, we note that the Lie bracket $[\p_t,\p_{x_1}]$ vanishes. Thus, we integrate by parts and apply $\p_t u|_{t=0}=0$ to obtain
\[
2h^4\Re\int_Q(-\frac{1}{h}) \p_t^2u \overline{\p_{x_1}u} dV_gdt = -2h^3\Re\int_M \p_tu(T,x)\overline{\p_{x_1}u(T,x)}dV_g+h^3 \int_Q \p_{x_1}|\p_tu|^2dV_gdt.
\]
Since the vector field $\p_{x_1}$ is divergence free, we get from the assumption $u|_\Sigma=0$ and integration by parts that the last term in the equation above vanishes. Hence, we have the following equality for the second term:
\begin{equation}
\label{eq:Car_est_term2}
2h^4\Re\int_Q(-\frac{1}{h}) \p_t^2u \overline{\p_{x_1}u} dV_gdt = -2h^3\Re\int_M \p_tu(T,x)\overline{\p_{x_1}u(T,x)}dV_g.
\end{equation}

Before estimating the third term, we recall that in local coordinates $(t,(x_j)_{j=1}^n)$ of $Q$ we have $[\p_t,\p_{x_j}]=0$ for every $j=1,\ldots,n$, and $\nabla_g u(t,x)=g^{ik}(x)\p_{x^k}u(t,x)$. Furthermore, since the metric $g$ is time-independent, we have 
\begin{align*}
\p_t|\nabla_g u|^2
= 2\langle \nabla_g \p_t u, \nabla_g u \rangle_g.
\end{align*}
Thus, by Green's identities and $u|_\Sigma=0$, we obtain
\[
-2h^4\Re\int_Q  \Delta_gu(\frac{1}{h}\beta-\frac{1}{\varepsilon} t ) \overline{\p_t u} dV_gdt=h^4\int_Q (\frac{1}{h}\beta-\frac{1}{\varepsilon} t ) \p_t|\nabla_g u|^2dV_gdt.
\]
Since $u|_{t=0}=0$, it follows immediately that $\nabla_g u(0,\cdot)=0$. Then we integrate by parts to get
\[
\int_Q (\frac{1}{h}\beta-\frac{1}{\varepsilon} t ) \p_t|\nabla_g u|^2dV_gdt=\int_M (\frac{1}{h}\beta-\frac{1}{\varepsilon} T ) |\nabla_g u(T, \cdot)|^2dV_g+\int_Q \frac{1}{\varepsilon}|\nabla_gu|^2dV_gdt.
\]
Therefore, we have verified that the third term on the right-hand side of \eqref{eq: Car_est_comp} satisfies
\begin{equation}
\label{eq:Car_est_term3}
-2h^4\Re\int_Q \Delta_gu(\frac{1}{h}\beta-\frac{1}{\varepsilon} t )\p_t u dV_gdt = h^4 (\frac{1}{h}\beta-\frac{1}{\varepsilon} T)\| \nabla_gu(T,\cdot) \|^2_{L^2(M)}+ \frac{h^4}{\varepsilon}\|\nabla_gu\|^2_{L^2(Q)}.
\end{equation}

We next follow the proof of  \cite[Lemma 4.2]{Kian_Oksanen} to estimate the fourth term. Since the metric $g$ is $x_1$-independent, it follows from the Leibniz rule and the local representation of the divergence operator \cite[Proposition 2.46]{lee2018introduction} that
\begin{align*}
2\overline{ \p_{x_1}u}\Delta_g u
&=2\div_g(\overline{ \p_{x_1}u}\nabla_g u)- \div_g(|\nabla_g u|^2_g \p_{x_1}).
\end{align*}
Thus, an application of the divergence theorem yields
\begin{align*}
2h^4\Re\int_Q \frac{1}{h} \overline{\p_{x_1}u}  \Delta_gu dV_gdt 
&= h^3\Re\int_\Sigma 2\p_\nu u\overline{ \p_{x_1}u}- |\nabla_g u|^2\n{\nu, \p_{x_1}}_gdS_gdt.
\end{align*}
Since $u|_\Sigma=0$, we see that $\nabla_g u|_\Sigma=(\p_\nu u)\nu$ and $\p_{x_1}u|_\Sigma=\p_\nu u\n{\nu, \p_{x_1}}_g:=\p_\nu u \nu_1$. Therefore, we have 
\begin{equation}
\label{eq:Car_est_term4}
2h^4\Re\int_Q \frac{1}{h} \Delta_gu \overline{\p_{x_1}u}  dV_gdt = h^3 \int_\Sigma {\nu_1} |\p_\nu u|^2 dS_gdt.
\end{equation}
%

We now turn our attention to the last term on the right-hand side of \eqref{eq: Car_est_comp}. To that end, we 
integrate by parts, and use $\div_g(\p_{x_1})=0$ and $u|_\Sigma = 0$ to write
\[
2\Re\int_M u(t,\cdot)\overline{\p_{x_1}u(t,\cdot)} dV_g=\int_M \p_{x_1}|u(t,\cdot)|^2 dV_g=\int_{\p M}  |u(t,\cdot)|^2 \nu_1 dS_g=0.
\]
Hence, by utilizing the condition $u|_{t=0}=0$, the last term on the right-hand side of  \eqref{eq: Car_est_comp} can be written as 
\begin{align*}
&2h^4\Re\int_Q -\big(\frac{1}{h^2}(1-\beta^2) +\frac{1}{\varepsilon} +\frac{2\beta}{h\varepsilon }t-\frac{1}{\varepsilon^2}t^2\big) u \big((\frac{1}{h}\beta-\frac{1}{\varepsilon} t )\overline{\p_t u}  -\frac{1}{h}\overline{\p_{x_1}u}\big) dV_gdt
\\
=&-h^4\int_Q  \big(\frac{1}{h^2}(1-\beta^2) +\frac{1}{\varepsilon} +\frac{2\beta}{h\varepsilon }t-\frac{1}{\varepsilon^2}t^2\big) (\frac{1}{h}\beta-\frac{1}{\varepsilon} t) \p_t|u|^2 dV_gdt
\\
=&-h^4 \big(\frac{1-\beta^2}{h^2} + \frac{1}{\varepsilon} + \frac{2\beta}{\varepsilon h}T-\frac{1}{\varepsilon^2}T^2\big)  \big(\frac{1}{h}\beta-\frac{1}{\varepsilon} T\big) \|u(T,\cdot)\|_{L^2(M)}^2
\\
&+h^4 \int_Q \big(\frac{3\beta^2-1}{\varepsilon h^2} -\frac{1}{\varepsilon^2} -\frac{6\beta}{\varepsilon^2 h}t+\frac{3}{\varepsilon^3}t^2\big) |u|^2dV_gdt.
\end{align*}

We now choose the numbers $\varepsilon,h>0$ such that
\begin{equation}
\label{eq:restriction_epsilon}
0<\varepsilon < 3T^2,
\quad 
\frac{1}{\sqrt{3}}<\beta<1,
\quad \text{ and } \quad 
\frac{1}{h}>\max\bigg\{\frac{2T}{\varepsilon \beta}, \frac{12\beta T}{\varepsilon(3\beta^2-1)}, \frac{2\beta T}{\varepsilon}, \frac{1}{\varepsilon }\bigg\}.
\end{equation}
%
These choices yield $h<\varepsilon$,
\[
\frac{3\beta^2-1}{\varepsilon h^2} -\frac{1}{\varepsilon^2} -\frac{6\beta}{\varepsilon^2h}t+\frac{3}{\varepsilon^3}t^2 \ge \frac{3\beta^2-1}{2\varepsilon h^2},
\]
and 
\[
0<\left(\frac{1-\beta^2}{h^2} + \frac{1}{\varepsilon} + \frac{2\beta}{\varepsilon h}T-\frac{1}{\varepsilon^2}T^2\right) \left(\frac{1}{h}\beta-\frac{1}{\varepsilon} T \right) 
\le 
\frac{3\beta}{h^3}.
\]
The choices of $h$ and $\varepsilon$ in \eqref{eq:restriction_epsilon} allow
the term $\frac{1}{h^2}$ to absorb the lower order terms when $0<h\ll \varepsilon \ll 1$.
Therefore, we get from these choices of $\varepsilon$ and $h$ that
\begin{equation}
\label{eq:Car_est_term5}
\begin{aligned}
&2h^4\Re\int_Q -\big(\frac{1}{h^2}(1-\beta^2) +\frac{1}{\varepsilon} +\frac{2\beta}{h\varepsilon }t-\frac{1}{\varepsilon^2}t^2\big) u ((\frac{1}{h}\beta-\frac{1}{\varepsilon} t )\p_t u  -\frac{1}{h}\p_{x_1}u) dV_gdt
\\
\ge& -3\beta h\|u(T,\cdot)\|_{L^2(M)}^2+ \frac{(3\beta^2-1)h^2}{2\varepsilon}\|u\|^2_{L^2(Q)}.
\end{aligned}
\end{equation}

By combining estimates \eqref{eq:Car_est_term1}--\eqref{eq:Car_est_term4} and \eqref{eq:Car_est_term5}, we obtain
\begin{align*}
\frac{1}{2}\|P_1u+P_2u\|^2
\ge &h^4\big(\frac{1}{h}\beta-\frac{1}{\varepsilon} T\big)\|\p_tu(T,\cdot)\|^2_{L^2(M)}+\frac{h^4}{\varepsilon}(\|\p_tu\|^2_{L^2(Q)}+\|\nabla_gu\|^2_{L^2(Q)})
\\
&
-2h^3\Re\int_M \p_tu(T,x)\overline{\p_{x_1}u}(T,x)dV_g+h^4 \big(\frac{1}{h}\beta-\frac{1}{\varepsilon} T \big)\| \nabla_gu(T,\cdot) \|^2_{L^2(M)}
\\
&
+h^3 \int_\Sigma \nu_1|\p_\nu u|^2 dS_gdt + \frac{(3\beta^2-1) h^2}{2\varepsilon}\|u\|^2_{L^2(Q)}-3\beta h\|u(T,\cdot)\|_{L^2(M)}^2.
\end{align*}
We sharpen the estimate above
%
by implementing $\frac{1}{h}> \frac{2T}{\varepsilon\beta}$ from \eqref{eq:restriction_epsilon} and utilizing the following inequality,
\[
\Re\int_M \p_tu(T,x)\overline{\p_{x_1}u}(T,x)dV_g \le \frac{\beta}{8}\|\p_tu(T,\cdot)\|_{L^2(M)}^2 + \frac{8}{\beta}\|\nabla_g u(T,\cdot)\|_{L^2(M)}^2,
\]
to obtain 
\begin{equation}
\label{eq:Car_est_L2}
\begin{aligned}
\frac{1}{2}\|P_1u+P_2u\|^2 \ge &\frac{1}{4} \beta h^3\|\p_tu(T,\cdot)\|^2_{L^2(M)}+\frac{h^4}{\varepsilon}(\|\p_tu\|^2_{L^2(Q)}+\|\nabla_gu\|^2_{L^2(Q)})
\\
&-h^3(\frac{16}{\beta}-\frac{\beta}{2})\|\nabla_g u(T,\cdot)\|_{L^2(M)}^2+ \frac{(3\beta^2-1) h^2}{2\varepsilon}\|u\|^2_{L^2(Q)}
\\
&+h^3 \int_\Sigma \nu_1|\p_\nu u|^2 dS_gdt-3\beta h\|u(T,\cdot)\|_{L^2(M)}^2.
\end{aligned}
\end{equation}

\textit{Step 3: An upper bound of} $\|P_3u\|_{L^2(Q)}$.
We deduce from \eqref{eq:P_operators}, \eqref{eq:derivatives_of_varphi}, as well as the triangle inequality that
\begin{equation}
\label{eq:Car_P3}
\begin{aligned}
\|P_3u\|^2_{L^2(Q)}
\le  3h^4\big( \|a\|^2_{L^\infty(Q)}\|\p_t u\|^2_{L^2(Q)} +\big( \frac{\beta^2}{h^2}\|a\|^2_{L^\infty(Q)} +\|q\|^2_{L^\infty(Q)} \big)\|u\|^2_{L^2(Q)} \big).
\end{aligned}
\end{equation}
Here we have used the inequality $(x+y+z)^2\le 3(x^2+y^2+z^2)$ for $x,y,z \in \R$.

In addition to the choice $0<\varepsilon<3T^2$ made in \eqref{eq:restriction_epsilon}, we will further require that 
\[
\frac{1}{2\varepsilon} \ge 3\|a\|^2_{L^\infty(Q)} \quad \text{ and } \quad	\frac{3\beta^2-1}{4\varepsilon}\ge 3( \beta^2\|a\|^2_{L^\infty(Q)}+\|q\|^2_{L^\infty(Q)}).
\]

After combining estimates \eqref{eq:norm_triangle}, \eqref{eq:Car_est_L2}, and \eqref{eq:Car_P3}, we obtain
the claimed estimate \eqref{eq:Carleman_est}.
This completes the proof of Proposition \ref{thm:boundary_carleman}.
\end{proof}


We are now ready to state and prove the boundary Carleman estimate.
\begin{prop}
\label{prop:Car_est}
Let $a, q\in L^\infty(Q, \C)$ and $v\in C^2(\overline{Q})$. If $v$ satisfies
\begin{equation}
\label{eq:condition_v}
v|_\Sigma=v|_{t=0}=\p_tv|_{t=0}=0,
\end{equation}
then for all $0<h \ll \varepsilon \ll 1$, we have
\begin{equation}
\label{eq:Car_est_conjugated}
\begin{aligned}
&\|e^{-\frac{1}{h}(\beta t +x_1)}h^2\mathcal{L}_{g, a,q}v\|_{L^2(Q)} +\cO(h^{3/2})\|e^{-\frac{1}{h}(\beta T+x_1)}\nabla_gv(T,\cdot)\|_{L^2(M)}\\
&+\cO(h^{1/2})\|e^{-\frac{1}{h}(\beta T+x_1)}v(T,\cdot)\|_{L^2(M)}+\cO(h^{3/2})\bigg(\int_{\Sigma_{-}} |\p_\nu \varphi||e^{-\frac{1}{h}(\beta t +x_1)}\p_\nu v|^2dS_gdt\bigg)^{1/2}
\\
&\ge \cO(h)\|e^{-\frac{1}{h}(\beta t +x_1)}v\|_{L^2(Q)} + \cO(h^{3/2})\|e^{-\frac{1}{h}(\beta T+x_1)}\p_tv(T,\cdot)\|_{L^2(M)}+\cO(h^2)(\|e^{-\frac{1}{h}(\beta t +x_1)}\p_tv\|_{L^2(Q)}
\\
&+\|e^{-\frac{1}{h}(\beta t +x_1)}\nabla_gv\|_{L^2(Q)}) +\cO(h^{3/2})\bigg(\int_{\Sigma_{+}}\p_\nu \varphi |e^{-\frac{1}{h}(\beta t +x_1)}\p_\nu v|^2dS_gdt\bigg)^{1/2}.
\end{aligned}
\end{equation}
Here $\varphi(x) = x_1$, $\Sigma_\pm=(0, T)\times \p M_\pm^{\text{int}}$, and $\p M_\pm = \{x \in \p M: \pm \p_\nu\varphi(x) \ge 0\}$. 
\end{prop}

\begin{proof}
By following the steps in the proof of \cite[Theorem 3.1]{Kian_damping}, we deduce the claimed estimate \eqref{eq:Car_est_conjugated} from estimate \eqref{eq:Carleman_est} by substituting $u = \exp(-\varphi_{h, \varepsilon}) v $, where $\varphi_{h, \varepsilon}$ is given by \eqref{eq:perturbed_weight}, as well as choosing $\varepsilon>0$ small but fixed.
\end{proof}

\begin{rem}\label{rem:density_argument}
The Carleman estimate \eqref{eq:Car_est_conjugated} can be extended to any function
\\
$v\in \mathcal{H}:= C^1([0, T]; L^2(M))\cap C([0, T]; H^1(M))$ satisfying \eqref{eq:condition_v} and $\Box_{e\oplus g_0}v\in L^2(Q)$. Indeed, we may approximate $f:=\Box_{e\oplus g_0}v\in L^2(Q)$ by a sequence $f_j\in C_0^\infty(Q)$ such that $f_j\to f$ in $L^2(Q)$ as $j\to \infty$. If $v_j$ solves $\Box_{e\oplus g_0}v_j=f_j$ and satisfies $v_j|_\Sigma=v_j|_{t=0}=\p_tv_j|_{t=0}=0
$, then $v_j\in C^\infty(\overline{Q})$ by \cite[Remark 2.10]{Lasiecka_Lions_Triggiani}. In particular, the boundary Carleman estimate \eqref{eq:Car_est_conjugated} holds for $v_j$. 

Furthermore, we have 
\[
\|v_j-v\|_{\mathcal{H}}+ \|\p_\nu v_j -\p_\nu v\|_{L^2(\Sigma)}  \le C \|f_j - f\|_{L^2(Q)}\to 0, \quad j\to 0,
\]
by the energy estimate \cite[Theorem 2.1]{Lasiecka_Lions_Triggiani} together with \cite[Remark 2.2]{Lasiecka_Lions_Triggiani}. Thus, estimate \eqref{eq:Car_est_conjugated} extends to $v$.
\end{rem}

\subsection{Semiclassical pseudodifferential operators}

In this subsection we recall some fundamental concepts of semiclassical pseudodifferential calculus on closed Riemannian manifolds by following the expositions of \cite{Sjo_2009} and \cite[Chapter 14]{Zworski}. Let $(N,\tilde g)$ be a smooth compact $n$-dimensional Riemannian manifold without boundary.
For each $m \in \R$, the Kohn-Nirenberg symbol class $S^m(T^\ast N)$ consists of smooth functions on the cotangent bundle $T^\ast N$, which in local coordinates of $N$ are given by 
\begin{equation}
\label{eq:symbol_class}
S^m_{1,0}(T^\ast N)=S^m(T^\ast N)=\{a(x, \xi)\in C^\infty (T^\ast N): |\p_{x}^\alpha \p_{\xi}^\beta a(x, \xi)|\le C_{\alpha \beta}\n{\xi}^{m-|\beta|}\},
\end{equation}
where $\n{\xi}=(1+ |\xi|^2)^{1/2}$. For a parameter-dependent symbol $a(x,\xi;h)$, we say that $a \in S^m(T^\ast N)$ if the estimate in \eqref{eq:symbol_class} holds uniformly for every $h\in (0,h_0)$ and for some $h_0>0$. A linear operator $B\colon C^\infty(N) \to C^\infty(N)$ is called \textit{negligible} if its Schwartz kernel $K_B\in C^{\infty}(N\times N)$ locally satisfies the estimate
$
\p_{x}^\alpha \p_{y}^\beta  K_B (x,y)=\cO (h^\infty)$ for all  $\alpha,\beta \in \N^n.
$

A linear map $A \colon C^\infty(N) \to C^\infty(N)$ is a semiclassical pseudodifferential operator of order $m \in \R$ if there exists $a \in S^m(T^\ast N)$ such that in local coordinates, the operator $A$ is given by the standard $h$-quantization
\begin{equation}
\label{eq:h_quantization}
Au(x)=\frac{1}{(2\pi h)^{n}}\int \int e^{\frac{i}{h}(x-y)\cdot \xi}a(x,\xi; h)u(y) dyd\xi+Bu(x),
\end{equation}
where the operator $B$ is negligible, and the operator $\psi A \varphi$ is negligible for each $\varphi, \psi \in C^\infty(N)$ with disjoint supports. We denote $\Psi^{m}(N)$ the set of semiclassical pseudodifferential operators of order $m$ on $(N,g)$.

We recall that the correspondence from an operator to a symbol is not globally well-defined, but there exists a bijective map between the following equivalence classes
\[
\Psi^{m}(N)/\Psi^{m-1}(N) \to S^{m}(T^\ast N)/S^{m-1}(T^\ast N).
\]
The image $\sigma_A(x,\xi;h)$ of $A\in \Psi^{m}(N)$ under this map is called the \textit{principal symbol} of $A$. These definitions allow us to compose the operators $A_j \in \Psi^{m_j}(N)$, $j=1,2$, and we have $A_1A_2 \in \Psi^{m_1+m_2}(N)$ with principal symbol $\sigma_{A_1A_2}=\sigma_{A_1}\sigma_{A_2}$.

An operator $A\in \Psi^{m}(N)$ is called \textit{elliptic} if there is a constant $C>0$, independent of $h$, such that the principal symbol satisfies
\[
|\sigma_A(x,\xi;h)|>\frac{1}{C}\langle \xi \rangle^m.
\]
An elliptic semiclassical operator $A\in \Psi^{m}(N)$ has an inverse $R \in \Psi^{-m}(N)$ in the sense that there exists $h_0>0$ such that for all $h\in (0,h_0)$ we have $RA=AR=I$ as linear operators on $C^\infty(M)$ and 
\[
\sigma_R(x,\xi;h)=\sigma_A(x,\xi;h)^{-1}\in S^{-m}(T^\ast N)/S^{-m-1}(T^\ast N).
\]
By \cite[Proposition 10.1]{Sjo_2009}, the operator
\begin{equation}
\label{eq:J_operator}
J^s = (1-h^2\Delta_{\tilde g})^{\frac{s}{2}}, \quad s\in \R,
\end{equation}
which is defined by the means of the spectral theorem, is elliptic and belongs to the class $\Psi^{s}(N)$ with principal symbol $\langle \xi \rangle^s$. We note that for all $s_1, s_2\in \R$, we  have 
\begin{equation}\label{eq:J^2_property2}
J^{s_1+s_2} = J^{s_1}J^{s_2},\quad (J^{s_1})^{-1} = J^{-s_1},\quad J^0 = I.
\end{equation}

We now define the semiclassical inner product of order $s\in \R$:
\[
(u,v)_{H^s_{\scl}(N)} := (J^su,J^sv)_{L^2(N)}, \quad u,v \in C^\infty(N).
\]
Then the semiclassical Sobolev space $H^s_{\scl}(N)$ is defined as the completion of $C^\infty(N)$ with respect to a related norm. Furthermore, every operator $A\in \Psi^{m}(N)$ 
yields a bounded map $A\colon H^s_{\scl}(N) \to H^{s-m}_{\scl}(N)$. Also, we recall that if $A$ is negligible, then the operator norm satisfies \begin{equation}
\label{eq:negligible_op_norm}
\|A\|_{H^{-s}_{\scl}(N) \to H^{s}_{\scl}(N)}=\cO(h^\infty) \quad \text{for all } s \in \R.
\end{equation}

Finally, we discuss the definition of semiclassical Sobolev spaces on an open subset $U \subset N$. We recall that for $u \in C^\infty(N)$, the norms
$
\|u\|_{H^{1}_{\scl}(N)}
$
and
\[
\|u\|^2_{h}:=\|u\|^2_{L^2(N)}+\|h\nabla_{\tilde g}u\|^2_{L^2(N)}
\]
are equivalent. We use the latter to define the semiclassical Sobolev space
$
H^{1}_{\scl}(U)
$
as a completion of $C^\infty(U)$ with respect to the norm $\|\cdot\|_h$ restricted on $U$, and $H^{-1}_{\scl}(U)$ as the topological dual of $H^1_{\scl}(U)$. We would like to recall the following characterization \cite[Section 3.13]{adams2003sobolev} of $H^{-1}_{\scl}(N)$ via the completion of $L^2(N)$ with respect to the norm
\begin{equation}
\label{eq:H-1norm}
\|v\|_{H^{-1}_{\scl}(N)} = \sup_{0\neq \psi \in H^{1}_{\scl}(N)}\frac{|\langle v,\psi \rangle_{L^2(N)}|}{\|\psi\|_{H^{1}_{\scl}(N)}}.
\end{equation}

\subsection{Interior Carleman estimate}
In this subsection we assume that $(\overline{Q},e\oplus g)$ is isometrically embedded into a closed Riemannian manifold  $(N,g')$ without boundary, where $g' =e\oplus g$ in some open neighborhood $U \subset N$ of $\overline{Q}$. Here $U=(a,b)\times \hat M$, where $[0,T]\subset (a,b)$ and $\hat{M}$ is an open manifold such that $M \subset \hat M$.
To prove the existence of suitable solutions to $\mathcal{L}_{g,a,q}u=0$ in $Q$, we need the following interior Calerman estimate for negative order Sobolev spaces.


\begin{lem}
\label{prop:shift_index}
Let $a\in W^{1,\infty}(Q)$ and $q\in L^\infty(Q, \C)$. Then for all $0<h \ll \varepsilon \ll 1$ and $w\in C^\infty_0(Q)$, there exists a constant $C>0$ such that
\begin{equation}
\label{eq:shift_index}
h\|w\|_{L^2(N)} \le C\|e^{\mp\frac{1}{h}(\beta t +x_1)}h^2\mathcal{L}_{g,a,q}e^{\pm\frac{1}{h}(\beta t+x_1)}w\|_{H^{-1}_{\scl}(N)}.
\end{equation}
\end{lem}


\begin{proof}
We shall follow the arguments from \cite[Section 2]{Krupchyk_Uhlmann_magschr}, see also \cite{DDS_Kenig_Sjo_Uhl,Krishnan_Vashisth}. 
We start by extending $a$ and $q$ from $W^{1,\infty}(Q)$ and $L^{\infty}(Q)$ to $W^{1,\infty}(N)$ and $L^{\infty}(N)$, respectively, and denote the extensions by the same letters.
Throughout the proof of this lemma, we will use the following shorthand notations for the convexified operators:
\[
{\Box_{\varphi_{\pm h, \varepsilon}}} = e^{-\varphi_{\pm h, \varepsilon}}h^2{\Box_g}e^{\varphi_{\pm h, \varepsilon}}
\quad \text{and} \quad 
\mathcal{P}_{\varphi_{\pm h, \varepsilon}} = e^{-\varphi_{\pm h, \varepsilon}}h^2\mathcal{L}_{g,a,q}e^{\varphi_{\pm h, \varepsilon}}.
\]
Here the convexified weight $\varphi_{\pm h,\varepsilon}$ is defined as in \eqref{eq:perturbed_weight}. Let us also define $u:=e^{\frac{t^2}{2\varepsilon}}w\in C^\infty_0(Q)$. 

We get from the triangle inequality that 
\begin{align}\label{eq:proof_lemma35_1}
\|\mathcal{P}_{\varphi_{\pm h,\varepsilon}}u\|_{H^{-1}_{\scl}(N)} \ge \|\Box_{\varphi_{\pm h,\varepsilon}}u\|_{H^{-1}_{\scl}(N)} -\|e^{-\varphi_{\pm h, \varepsilon}}h^2(a\p_t+q)e^{\varphi_{\pm h, \varepsilon}}u\|_{H^{-1}_{\scl}(N)},
\end{align}
and begin by estimating the first term on the right-hand side of inequality \eqref{eq:proof_lemma35_1}. This will be followed by a perturbation of the first term with the second term, which contains the lower order terms. 

Let the open set $U \subset N$ be the same as at the beginning of this subsection.
We note that by taking $a = q = 0$ and $v\in C^\infty_0(U)$, Proposition \ref{thm:boundary_carleman} yields
\begin{equation}
\label{eq:bound_conjugated}
\frac{h}{\sqrt{\varepsilon}}\|v\|_{H^1_{\scl}(N)} 
\le 
C\|\Box_{\varphi_{\pm h,\varepsilon}}v\|_{L^2(N)},
\end{equation}
where $C = \frac{2}{\sqrt{3\beta^2 - 1}}$.

Let $\chi\in C^\infty_0(U)$ be equal to $1$ in a neighborhood of $\overline{Q}$.
Since the operator $J^{-1}$, as defined in \eqref{eq:J_operator}, is in $\Psi^{-1}(N)$, and $\supp(1-\chi)\cap\supp u = \emptyset$, we have by \eqref{eq:negligible_op_norm}
that
\begin{equation}\label{eq:pseudolocal}
\|(1-\chi)J^{-1}u\|_{H^1_{\scl}(N)} = \mathcal{O}(h^{\infty}) \|u\|_{H^{-1}_{\scl(N)}} \le \mathcal{O}(h^{\infty})\|u\|_{L^2(N)}.
\end{equation}
Therefore, for all $0<h \ll \varepsilon \ll 1$ we have
\begin{equation}
\label{eq:proof_lemma35_2}
\begin{aligned}
\|u\|_{L^2(N)} 
& \le
\|\chi J^{-1} u \|_{H^1_{\scl}(N)} +  \|(1-\chi)J^{-1} u \|_{H^1_{\scl}(N)}
\le \|\chi J^{-1} u \|_{H^1_{\scl}(N)}+\cO(h^\infty)\|u\|_{L^2(N)}.
\end{aligned}
\end{equation}

To estimate the first term on the right-hand side of \eqref{eq:proof_lemma35_2}, we apply \eqref{eq:bound_conjugated} 
with $v = \chi J^{-1}u\in C^\infty_0(U)$ to get from the triangle inequality and the commutator $[A,B]=AB-BA$ that
\begin{equation}
\label{eq:est_v_H^1}
\begin{aligned}
&\|\chi J^{-1}u\|_{H^1_{\scl}(N)} 
\\&
\le \mathcal{O}\left(\frac{\sqrt{\varepsilon}}{h}\right) \left(\|\chi J^{-1} \Box_{\varphi_{\pm h,\varepsilon}} u\|_{L^2(N)} 
+ \|\chi [\Box_{\varphi_{\pm h,\varepsilon}} ,J^{-1}]u\|_{L^2(N)} 
+ \|[\Box_{\varphi_{\pm h,\varepsilon}},\chi ]J^{-1}u\|_{L^2(N)}\right).
\end{aligned}
\end{equation}

We now estimate the two terms containing the commutators on the right-hand side of \eqref{eq:est_v_H^1}. Noting that $\Box_{\varphi_{\pm h, \varepsilon}}\in \Psi^{2}(N)$, an application of \cite[Theorem 9.5(iii)]{Zworski} yields that the commutator satisfies the equation $[\Box_{\varphi_{\pm h, \varepsilon}},J^{-1}]=hR_1$, where $R_1 \in \Psi^0(N)$. Thus, for the second term on the right-hand side of \eqref{eq:est_v_H^1} we have
\begin{equation}
\label{eq:lemma35_comm_1}
\|\chi [\Box_{\varphi_{\pm h,\varepsilon}} ,J^{-1}]u\|_{L^2(N)} \le h\|u\|_{L^2(N)}.
\end{equation}

For the third term on the right-hand side of \eqref{eq:est_v_H^1}, we observe that
\[
[\Box_{\varphi_{\pm h,\varepsilon}},\chi ]=h^2\Box\chi +2h^2\left(\p_t\chi\p_t \varphi_{\pm h,\varepsilon}-\langle \nabla_g\chi,\nabla_g\varphi_{\pm h,\varepsilon} \rangle\right) + 2h^2(\p_t\chi\p_t- \langle \nabla_g\chi,\nabla_g\cdot \rangle).
\]
Thus, $\supp([\Box_{\varphi_{\pm h,\varepsilon}},\chi ])\subset \supp (\p_t\chi,\nabla_g\chi)$ and $\supp (\p_t\chi,\nabla_g\chi)\cap \supp u = \emptyset$. Therefore, we get from \eqref{eq:negligible_op_norm} that
\begin{equation}\label{eq:lemma35_comm_2}
\|[\Box_{\varphi_{\pm h,\varepsilon}},\chi ]J^{-1}u\|_{L^2(N)}\le \mathcal{O}(h^\infty)\|u\|_{L^2(N)}.
\end{equation}

Since $\chi J^{-1}\in \Psi^{-1}(N)$, together with \eqref{eq:est_v_H^1}, \eqref{eq:lemma35_comm_1}, and \eqref{eq:lemma35_comm_2}, we deduce from \eqref{eq:proof_lemma35_2} that 
\begin{align*}
\|u\|_{L^2(N)}&
\le \mathcal{O}\left(\frac{\sqrt{\varepsilon}}{h}\right) \|\Box_{\varphi_{\pm h,\varepsilon}} u\|_{H^{-1}_{\scl}(N)} + \mathcal{O}(\sqrt{\varepsilon})\|u\|_{L^2(N)}+\mathcal{O}(h^\infty)\|u\|_{L^2(N)}.
\end{align*}
For later purposes, let us rewrite this inequality as
\begin{equation}
\label{eq:proof_lemma35_3}
\frac{h}{\sqrt{\varepsilon}}\|u\|_{L^2(N)}\le \mathcal{O}(1)\|\Box_{\varphi_{\pm h,\varepsilon}} u\|_{H^{-1}_{\scl}(N)} + \mathcal{O}(h)\|u\|_{L^2(N)}.
\end{equation}

To estimate the lower order terms in \eqref{eq:proof_lemma35_1}, 
taking $0\neq\psi \in H^{1}_{\scl}(N)$ with $\|\psi\|_{H^{1}_{\scl}(N)}=1$, we integrate by parts to obtain
\[
\begin{aligned}
&\langle e^{-\varphi_{\pm h, \varepsilon}}h^2(a\p_t+q)e^{\varphi_{\pm h, \varepsilon}}u,\psi \rangle_{L^2(N)}
\\
&=-\langle u,  h^2\left((\p_t a)\psi+a\p_t\psi-(\p_t\varphi_{\pm h,\varepsilon})a\psi \right)\rangle_{L^2(N)} + \langle h^2qu,\psi \rangle_{L^2(N)}.
\end{aligned}
\]
By recalling that $\p_t\varphi_{\pm h,\varepsilon} =\pm \frac{\beta}{h}-\frac{t}{\varepsilon}$, $\|\psi\|_{H^1_{\scl}(N)}=1$,  $h<\varepsilon$, and $\beta<1$, we get from the Cauchy–Schwartz inequality that 
\begin{equation}
\label{eq:proof_lemma35_4}
\begin{aligned}
|\langle e^{-\varphi_{\pm h, \varepsilon}}h^2(a\p_t+q)e^{\varphi_{\pm h, \varepsilon}}u,\psi \rangle_{L^2(N)}|
\le h(3+T)(\|a\|_{W^{1,\infty}(N)}+\|q\|_{L^\infty(N)})\|u\|_{L^2(N)}.
\end{aligned}
\end{equation}

Therefore, the characterization \eqref{eq:H-1norm} of the semiclassical $H^{-1}$-norm and  \eqref{eq:proof_lemma35_4} imply that
\begin{align}
\label{eq:proof_lemma35_5}
\|e^{-\varphi_{\pm h, \varepsilon}}h^2(a\p_t+q)e^{\varphi_{\pm h, \varepsilon}}u\|_{H^{-1}_{\scl}(N)} \le \mathcal{O}(h) \|u\|_{L^2(N)}.
\end{align}
Using \eqref{eq:proof_lemma35_3} and \eqref{eq:proof_lemma35_5}, we derive from \eqref{eq:proof_lemma35_1} that
\begin{align*}\label{}
\mathcal{O}(1)\|\mathcal{P}_{\varphi_{\pm h,\varepsilon}}u\|_{H^{-1}_{\scl}(N)} + \mathcal{O}(h)\|u\|_{L^2(N)} \ge \frac{h}{\sqrt{\varepsilon}}\|u\|_{L^2(N)},
\end{align*}
which can be rewritten as 
\begin{align*}
\mathcal{O}(\frac{\sqrt{\varepsilon}}{h})\|\mathcal{P}_{\varphi_{\pm h,\varepsilon}}u\|_{H^{-1}_{\scl}(N)} + \mathcal{O}(\sqrt{\varepsilon})\|u\|_{L^2(N)} \ge \|u\|_{L^2(N)}.
\end{align*}
We now take $\varepsilon$ small enough but fixed to absorb the second term on the left-hand side and get 
\begin{align}\label{eq:appriori_est}
\mathcal{O}(\frac{1}{h})\|\mathcal{P}_{\varphi_{\pm h,\varepsilon}}u\|_{H^{-1}_{\scl}(N)} \ge \|u\|_{L^2(N)}.
\end{align}

Finally, we use $w = e^{-\frac{t^2}{2\varepsilon}}u$ for $0<t<T$ and apply \eqref{eq:appriori_est}
to obtain
\begin{align*}
\|w\|_{L^2(N)}
\le \cO(\frac{1}{h}) e^{\frac{T^2}{2\varepsilon}} \|e^{\mp\frac{1}{h}(\beta t +x_1)}h^2\mathcal{L}_{g,a,q}e^{\pm\frac{1}{h}(\beta t+x_1)}w\|_{H^{-1}_{\scl}(N)}.
\end{align*}
Here the inequality relies on the assumption $u\in C_0^\infty(Q)$. This completes the proof of Proposition \ref{prop:shift_index}. 
\end{proof}

The following solvability result will be implemented in the next section to construct CGO solutions for the operator $\mathcal{L}_{g,a,q}$.
\begin{prop}
\label{prop:solvability}
Let $a \in W^{1,\infty}(Q)$, $q \in L^\infty(Q, \C)$, and $s = \frac{1}{h} +i\lambda$ with $\lambda \in \R$ fixed. If $h>0$ is small enough, then for all $v\in L^2(Q)$ there exists a solution $u\in H^1_{\scl}(Q)$ to the equation
\begin{equation}
\label{eq:solvability}
e^{\pm s(\beta t+x_1)}h^2\mathcal{L}_{g, a, q}e^{\mp s(\beta t+x_1)}u = v \textit{ in } Q
\end{equation}
such that 
\begin{equation}
\label{eq:solvability_bound}
\|u\|_{H^1_{\scl}(Q)}\le \cO(h^{-1})\|v\|_{L^2(Q)}.
\end{equation}
\end{prop}
\begin{proof}
The proof uses standard functional analysis arguments, which have been utilized to prove analogous results for elliptic operators, see for instance \cite{Ferreira_Kenig_Salo_Uhlmann,Krupchyk_Uhlmann_magschr_2014}, as well as hyperbolic operators in \cite{Krishnan_Vashisth}. 
\end{proof}

\section{Construction of CGO Solutions Based on Gaussian Beam Quasimodes}
\label{sec:CGO_solution}
Let $(M,g)$ be a CTA manifold, $a \in W^{1,\infty}(Q)$, and $q \in C(\overline{Q})$. Let $0<h <1$, $\lambda \in \R$, and $s=\frac{1}{h}+i\lambda \in \C$. In the first part of this section we shall assume that the conformal factor $c=1$ and write $\mathcal{L}_{g,a,q}$ for $\mathcal{L}_{c,g,a,q}$. The goal of this section is to construct an exponentially decaying, with respect to the real part of $s$, solution to the equation $\mathcal{L}_{g, a,q}^*u_1=0$ in $Q$ of the form
\begin{equation}
\label{eq:form_CGO}
u_1=e^{-s(\beta t+x_1)}(v_s+r_1),
\end{equation}
where $\mathcal{L}_{g, a, q}^\ast =\mathcal{L}_{g, -\overline{a}, \overline{q}-\p_t\overline{a}}$ is the formal $L^2$-adjoint of the operator $\mathcal{L}_{g, a, q}$, as well as an exponentially growing solution to the equation $\mathcal{L}_{g,a,q}u_2=0$ in $Q$ of the form
\begin{equation}
\label{eq:CGO_exp_grow}
u_2=e^{s(\beta t+x_1)}(w_s+r_2).
\end{equation}
Here $v_s$ and $w_s$ are smooth Gaussian beam quasimodes, and $r_1=r_{1,s}, r_2=r_{2,s}$ are correction terms that vanish in the limit $h\to 0$. 
We construct the Gaussian beam quasimodes $v_s$ and $w_s$ for each nontangential geodesic of $(M_0,g_0)$ in Subsection \ref{subsec:Gaussian_beam}. In particular, the functions $v_s$ and $w_s$ satisfy the estimates given in \eqref{eq:estimate_v} and \eqref{eq:estimate_w}. In Subsection \ref{Ssec:concetration} we will establish a concentration property of the quasimodes along the nontangential geodesic in the limit $h\to 0$. Finally, in Subsection \ref{subsec:CGO} we will construct the remainder terms $r_1$ and $r_2$, whose existence and decaying properties follow directly from the estimates for the quasimode and the interior Carleman estimate Proposition \ref{prop:solvability}. We emphasize that the CGO solutions \eqref{eq:form_CGO} and \eqref{eq:CGO_exp_grow} are constructed under the assumption $c=1$. We shall incorporate general conformal factors $c$ and modify our CGO solutions accordingly in Subsection \ref{subsec:CGO}.

Let us write $x=(x_1, x')$ for coordinates in $\R \times M_0$, globally in $\R$ and locally in $M_0$. To justify our construction we note that a function $u_1$ of the form  \eqref{eq:form_CGO} solves the equation $\mathcal{L}^{\ast}_{g, a,q}u_1=0$ if
\[
e^{s(\beta t+x_1)}\mathcal{L}^{\ast}_{g, a,q}e^{-s(\beta t+x_1)}r_1=-e^{s(\beta t+x_1)}\mathcal{L}^{\ast}_{g, a,q}e^{-s(\beta t+x_1)}v_s.
\]

\subsection{Construction of Gaussian beam quasimodes}	
\label{subsec:Gaussian_beam}

In this subsection we focus on constructing Gaussian beam quasimodes. Initially introduced in \cite{Babich_Ulin, Ralston_1982}, the construction of Gaussian beam quasimodes has a very long tradition in spectral theory and in microlocal analysis, see also \cite{Babich_Buldyrev, Ralston_1977}. Gaussian beam quasimodes have also been used extensively to solve inverse problems, starting with \cite{Belishev_Katchalov, Katchalov_Kurylev}. Among the literature in this direction, we refer readers to \cite{Cekic, Ferreira_Kenig_Salo_Uhlmann, Ferreira_Kur_Las_Salo, Krupchyk_Uhlmann_magschr,Yan} for applications to elliptic operators and \cite{Feizmohammadi_et_all_2019,KKL_book} to hyperbolic operators. 

Let $(M, g)$ be a CTA manifold with the conformal factor $c=1$, and let $T>0$. 
Replacing the transversal manifold $(M_0, g_0)$ by a slightly larger manifold if necessary, we may assume without loss of generality that $(M, g)\subset (\R \times M_0^\mathrm{int}, e\oplus g_0)$. 
The Gaussian beam quasimodes will be constructed in $\R^2 \times M_0^{\mathrm{int}}$. To obtain $C^\infty$-smooth Gaussian beam quasimodes, we shall regularize the damping coefficient $a$ and explain the necessity of doing so in the proof of Theorem \ref{prop:Gaussian_beam}. 
We extend $a$ to $W^{1,\infty}(\R^2\times M_0^{\mathrm{int}})$ with compact support.
Using a partition of unity argument combined with a regularization in each coordinate patch, we have the following result, see \cite[Lemma 2.1]{Salo_diss} for details.
\begin{prop}
\label{prop:regularization}
For any $a\in W_0^{1,\infty}(\R^2\times M_0^{int})$, there exists an open  and bounded set $W\subset \R^2\times M_0^{int}$ and a family $a_\zeta\in C_0^\infty (W, \C)$ such that  
\begin{equation}
\label{eq:est_diff_regu}
\begin{aligned}
&\|a-a_\zeta\|_{L^\infty}=o(1), \quad \|a_\zeta\|_{L^\infty}=\cO(1), \quad \|\nabla_g a_\zeta\|_{L^\infty}=o(\zeta^{-1}),
\\
&\|\p_t a_\zeta\|_{L^\infty}=o(\zeta^{-1}), \quad \|\p_t^2 a_\zeta\|_{L^\infty}=o(\zeta^{-2}), \quad \|\Delta_g a_\zeta\|_{L^\infty}=o(\zeta^{-2}), \quad \zeta \to 0.
\end{aligned}
\end{equation}
Here the $L^\infty$-norms are taken over the set $\R^2\times M_0^{int}$.
\end{prop}

We are now ready to state and prove our first main result of this section.

\begin{thm}
\label{prop:Gaussian_beam}
Let $(M, g)$ be a smooth CTA manifold with boundary, $T>0$, $\beta \in (\frac{1}{\sqrt{3}},1)$, and let $s=\frac{1}{h}+i\lambda$, $0<h\ll 1$, $\lambda \in \R$ fixed.
Let $a \in W^{1,\infty}(Q)$ and $q \in C(\overline{Q})$. Then for every unit speed nontangential geodesic $\gamma$ of $(M_0,g_0)$ there exist one parameter families of Gaussian beam quasimodes $v_s, w_s\in C^\infty(\R^2\times M_0)$ such that the estimates
\begin{equation}
\label{eq:estimate_v}
\begin{aligned}
&\|v_s\|_{L^2(Q)}=\cO(1), 
\quad 
\|\p_t v_s\|_{L^2(Q)}=o(h^{-1/2}),
\quad
\|e^{s(\beta t+x_1)}h^2\mathcal{L}^\ast_{g, a, q}e^{-s(\beta t+x_1)}v_s\|_{L^2(Q)} =o(h),
\end{aligned}
\end{equation}
and
\begin{equation}
\label{eq:estimate_w}
\begin{aligned}
&\|w_s\|_{L^2(Q)}=\cO(1), 
\quad 
\|\p_t w_s\|_{L^2(Q)}=o(h^{-1/2}),
\quad
\|e^{-s(\beta t+x_1)}h^2\mathcal{L}_{g, a, q}e^{s(\beta t+x_1)}w_s\|_{L^2(Q)}=o(h)
\end{aligned}
\end{equation}
are valid as $ h\to 0$.
\end{thm}

The proof of Theorem \ref{prop:Gaussian_beam} is very long and given in the following subsections.

\subsubsection{Preparations for the proof of Theorem \ref{prop:Gaussian_beam}}

Let $\gamma=\gamma(\tau)$ be a nontangential geodesic in the transversal manifold $(M_0,g_0)$ of length $L>0$. By following \cite[Example 9.32]{lee2012smooth}, we embed $(M_0, g_0)$ into a closed manifold $(\hat{M}_0, g_0)$ of the same dimension. We also extend $\gamma$ as a unit speed geodesic in $\hat{M}_0$. Since $\gamma$ is nontangential, we can choose $\varepsilon>0$ so that $\gamma(\tau)\in \hat{M}_0\setminus M_0$ and it does not self-intersect for $\tau \in [-2\varepsilon, 0) \cup (L, L+2\varepsilon]$.

We begin with the construction of a Gaussian beam quasimode for the conjugated operator $e^{-s(\beta t+x_1)}\mathcal{L}^\ast_{g,a,q}e^{s(\beta t+x_1)}$. 
We follow the main ideas from \cite{Cekic,DDS_Kenig_Sjo_Uhl,Krupchyk_Uhlmann_magschr} and modify the argument in accordance with the extra time variable $t$. 
Consider the CGO ansatz $v_s$ of the form 
\[
v_s(t,x_1, x';h)=e^{is\Theta(x')}b(t,x_1,x';h), 
\]
where $s=\frac{1}{h}+i\lambda$ is a complex number, and the amplitude $b$ depends implicitly on the semiclassical parameter $h$.

Since the phase function $\Theta$ is independent of $t$, we have
\begin{equation}
\label{eq:conjugated_dt}
e^{-is\Theta}a\p_t(e^{is\Theta} b)=a\p_t b \quad \text{and} \quad e^{-is\Theta}\p_t^2(e^{is\Theta} b)=\p_t^2b. 
\end{equation}
Also, as $\Theta$ is independent of $x_1$ and $g=e\oplus g_0$, we get
\begin{equation}
\label{eq:conjugated_Laplacian}
e^{-is\Theta}(-\Delta_g)e^{is\Theta} b=-\Delta_g b-is[2\langle \nabla_{g_0}\Theta, \nabla_{g_0}b(x_1, \cdot)\rangle_{g_0}+(\Delta_{g_0}\Theta)b]+s^2\langle \nabla_{g_0}\Theta, \nabla_{g_0}\Theta\rangle_{g_0}b.
\end{equation}
Using \eqref{eq:conjugated_dt} and \eqref{eq:conjugated_Laplacian}, we obtain

\begin{equation}
\label{eq:conjugated_L1}
\begin{aligned}
e^{s(\beta t+x_1)} \mathcal{L}_{g, -\overline{a}, \overline{q}-\p_t \overline{a}}e^{-s(\beta t+x_1)}v_s
= &
e^{is\Theta}[s^2(\langle \nabla_{g_0}\Theta, \nabla_{g_0}\Theta\rangle_{g_0} - (1 - \beta^2))b
\\
&+
s(2\p_{x_1}b-2\beta\p_tb - 2i\langle \nabla_{g_0}\Theta, \nabla_{g_0}b(x_1, \cdot)\rangle_{g_0} - i(\Delta_{g_0}\Theta)b+ \beta \overline{a} b)
\\
&+
\mathcal{L}_{g, -\overline{a}, \overline{q}-\p_t \overline{a}}b].
\end{aligned}
\end{equation}
The computations in \eqref{eq:conjugated_L1} suggest that in order to verify the estimates in \eqref{eq:estimate_v}, we should construct the phase function $\Theta$ and the amplitude $b$ such that they approximately solve the eikonal and transport equations appearing on the right-hand side of \eqref{eq:conjugated_L1} as multipliers of the terms $s^2$ and $s$, respectively.

The construction of the Gaussian beam quasimode $v_s$ is divided into several steps, which are addressed in the following subsections.
\begin{enumerate}
\item[(1)] Fermi coordinates near $\gamma$: In these local coordinates $x'=(\tau,y) \in M_0$ is given by its closest point $\gamma(\tau)$ to the fixed geodesic $\gamma$ and by its location $y$ in the respective geodesic plane perpendicular to $\dot\gamma(\tau)$. Due to these choices, the metric tensor $g_0$ on $\gamma$ is Euclidean up to the first order. This simplifies many subsequent computations.
\item[(2)] Eikonal equation and phase function: To satisfy estimate \eqref{eq:estimate_v} we solve the eikonal equation, which is the multiplier of $s^2$ in \eqref{eq:conjugated_L1}, up to a certain order. The solution is called the phase function $\Theta$.
\item[(3)] Transport equation and amplitude: To satisfy estimate \eqref{eq:estimate_v} we solve the transport equation, which is the multiplier of $s$ in \eqref{eq:conjugated_L1}, up to a certain order, and call the obtained solution $b$ the amplitude. Since we want to construct a smooth Gaussian beam quasimode, we shall solve the transport equation with a regularized damping coefficient $a_\zeta$ and relate the regularization parameter $\zeta$ to the semiclassical parameter $h$ with an explicit expression. We solve the transport equation via a change of variables, which converts the transport equation into a better understood $\overline{\p}$-equation. 
\item[(4)] Local verification of estimate \eqref{eq:estimate_v}: In this step we provide a proof for estimate \eqref{eq:estimate_v} in a neighborhood of a fixed point $z_0=\gamma(\tau_0)$. The key of the proof is the fact that map $x\mapsto |x|^ke^{-d|x|^2}$ is in $L^2(\R^n)$ for all $d>0$ and $k\geq 0$.
\item[(5)] Global construction of Gaussian beam quasimodes: We provide the global construction of $v_s$ by gluing together quasimodes defined along the small pieces of the geodesic.
\end{enumerate}

\subsubsection{Fermi coordinates near $\gamma$.}
We fix a point $z_0=\gamma(\tau_0)$ on $\gamma([-\varepsilon, L+\varepsilon])$ and construct the quasimode locally near $z_0$. Let $(\tau, y) \in \Omega:=\{(\tau, y)\in \R \times \R^{n-2}: |\tau-\tau_0|<\delta, |y|<\delta'\}$, $\delta, \delta'>0$, be Fermi coordinates near $z_0$. The detailed construction of these coordinates is given in \cite[Lemma 7.4]{Kenig_Salo}. Heuristically, the idea is to first choose the number $\delta>0$ to be sufficiently small, so that the geodesic segment $\gamma|_{[\tau_0-\delta,\tau_0+\delta]}$ is not self-intersecting. Then one uses the parallel transport to choose some vector fields $E_1(\tau),\ldots,E_{n-2}(\tau)$ along $\gamma$ such that the vector fields $\dot\gamma(\tau),E_1(\tau),\ldots,E_{n-2}(\tau)$ span a parallel orthonormal frame along $\gamma$. Due to the choice of $\delta$, as well as the inverse function theorem, there exists $\delta'>0$ such that the map
$
F(\tau,y)=\exp_{\gamma(\tau)}\left(y^\alpha E_\alpha(\tau) \right)
$
is a diffeomorphism in the set $\Omega$. Here $\exp$ is the exponential map of $(M_0,g_0)$ and $\alpha\in \{1,\ldots,n-2\}$.
%
%
%

We note that near $z_0=\gamma(\tau_0)$ the trace of the geodesic $\gamma$ is given by the set $\Gamma=\{(\tau, 0): |\tau-\tau_0|<\delta\}$. 
Due to this construction, we get
\begin{equation}
\label{eq:g0_in_Fermi_coordinates}
g_0^{jk}(\tau,0)=\delta^{jk} \quad \text{and} \quad  \p_{y_l}g_0^{jk}(\tau, 0)=0.
\end{equation}
Hence, by Taylor's theorem, for small $|y|$ we have
\begin{equation}
\label{eq:g0_near_geo}
g_0^{jk}(\tau, y)=\delta^{jk}+\mathcal{O}(|y|^2).
\end{equation}

In these coordinates the Gaussian beam ansatz takes the form
\begin{equation}
\label{eq:Gaussian_beam_form}
v_s(t, x_1, \tau, y)=e^{is\Theta(\tau,y)} b(t, x_1, \tau, y;h), 
\end{equation}
and our aim is to find the phase function $\Theta \in C^\infty(\Omega, \C)$ that satisfies
\begin{equation}
\label{eq:property_of_phi}
\Im \Theta \ge 0, \quad \Im \Theta|_\Gamma=0, \quad \Im \Theta(\tau, y) 
\text{ is bi-Lipschitz equivalent to } 
|y|^2,
\end{equation}
and an amplitude $b\in C^\infty(\R \times \R \times \Omega, \C)$ such that $\supp(b(t, x_1, \cdot)) \subset \{|y|<\delta'/2\}$. 
In particular, the name ``Gaussian beam" comes from \eqref{eq:property_of_phi}.


\subsubsection{Eikonal equation and phase function.}
\label{ssec:phase_function}
Our goal in this step is to find a phase function $\Theta$ by solving the eikonal equation $\langle \nabla_{g_0}\Theta, \nabla_{g_0}\Theta\rangle_{g_0}=1-\beta^2$ up to order $|y|^3$ on $\Gamma$ by arguing similarly as in \cite{Ferreira_Kur_Las_Salo,KKL_book,Krupchyk_Uhlmann_magschr,Ralston_1977,Ralston_1982}. That is, we find a function $\Theta(\tau, y)\in C^\infty(\Omega, \C)$ that satisfies


\begin{equation}
\label{eq:eikonel_eq}
\langle \nabla_{g_0}\Theta, \nabla_{g_0}\Theta\rangle_{g_0}-(1-\beta^2)=\mathcal{O}(|y|^3), \quad y\to 0,
\end{equation}
and 
\begin{equation}
\label{eq:imaginarypart_phi}
\Im \Theta(\tau,y)\ge d|y|^2
\end{equation}
for some constant $d>0$ depending on $\beta$. Equation \eqref{eq:eikonel_eq}, combined with a scaling in the semiclassical parameter $h$, will be used in Subsection \ref{Ssec:local_estimate} to prove estimate \eqref{eq:estimate_v}.  

In order to utilize the Taylor expansion \eqref{eq:g0_near_geo} for the metric $g_0$, we look for a function $\Theta$ of the form $\Theta=\Theta_0+\Theta_1+\Theta_2$, where
\[
\Theta_j(\tau,y)=\sum_{|\alpha|=j}\frac{\Theta_{j,\alpha}(\tau)}{\alpha!}y^\alpha, \quad j=0,1,2,
\]
are the homogeneous polynomials in the $y$-variable.
We also write $g^{jk}_0=g_{0,0}^{jk}+g_{0,1}^{jk}+g_{0,2}^{jk}+r^{jk}_3$, where 
\[
g_{0,l}^{jk}(\tau, y)=\sum_{|\mu|=l} \frac{g_{0,l,\mu}^{jk}(\tau)}{\mu!}y^\mu, \quad l=0,1,2,
\]
and $r_3=\cO(|y|^3)$ is the remainder in Taylor's theorem. By the properties of Fermi coordinates \eqref{eq:g0_in_Fermi_coordinates}, we have $g_{0,0}^{jk}=\delta^{jk}$ and $g_{0,1}^{jk}=0$. We then choose accordingly that 
\begin{equation}
\label{eq:def_Theta0_Theta1}
\Theta_0(\tau, y)=\sqrt{1-\beta^2}\tau \quad \text{ and } \quad \Theta_1(\tau, y)=0.
\end{equation}

Let us next find $\Theta_2$. To this end, we write the metric $g^{jk}_0= \delta^{jk}+g_{0,2}^{jk}+\cO(|y|^3)$. With the understanding that $j,k$ run from 1 to $n-1$ and $\alpha, \mu$ run from 2 to $n-1$, we see that 
\begin{align*}
\langle \nabla_{g_0}\Theta, \nabla_{g_0}\Theta\rangle_{g_0}-(1-\beta^2) = 
[2\sqrt{1-\beta^2}\p_\tau\Theta_2+\nabla_y\Theta_2\cdot \nabla_y\Theta_2+(1-\beta^2)g_{0,2}^{11}]+\cO(|y|^3).
\end{align*}


Similar to \cite{Cekic, Ferreira_Kur_Las_Salo, Kenig_Salo}, our aim is to find a function $\Theta_2$ such that
\begin{equation}
\label{eq:Theta_2_eq}
2\sqrt{1-\beta^2}\p_\tau\Theta_2+\nabla_y\Theta_2\cdot \nabla_y\Theta_2+(1-\beta^2)g_{0,2}^{11}=0.
\end{equation}
Due to our previous choice for the form of $\Theta_2$, we have
\[
\Theta_2(\tau, y)=\frac{1}{2}\sqrt{1-\beta^2}H(\tau)y\cdot y, 
\]
and in order to satisfy \eqref{eq:property_of_phi}, we seek to obtain a smooth complex-valued symmetric matrix $H(\tau)$ with a positive definite imaginary part. Since each term of \eqref{eq:Theta_2_eq} is quadratic in $y$, it is sufficient for $H(\tau)$ to satisfy the following matrix initial value problem, called the Ricatti equation
\begin{equation}
\label{eq:Riccati_eq}
\dot{H}(\tau)+H(\tau)^2=F(\tau), \quad H(\tau_0)=H_0, \quad \text{ for } \tau \in \R,
\end{equation}
where $F(\tau)$ is a symmetric matrix such that $g_{0,2}^{11}(\tau, y)=-F(\tau)y\cdot y$ in $(\tau_0-\delta,\tau_0+\delta)$. If we choose $H_0$ to be a symmetric complex-valued matrix such that $\Im(H_0)$ is positive definite, then \eqref{eq:Riccati_eq} has a unique smooth symmetric complex-valued solution $H(\tau)$ with $\Im H(\tau)$ positive definite, see \cite[Lemma 2.56]{KKL_book} for details. 

Therefore, the phase function $\Theta$ reads
\begin{equation}
\label{eq:phi}
\Theta(\tau, y)=\sqrt{1-\beta^2}(\tau+\frac{1}{2}H(\tau)y\cdot y).
\end{equation}
Due to the compactness and the positive definiteness of $\Im(H(\tau))$, the function $\Theta$ satisfies the properties in \eqref{eq:property_of_phi}.

\subsubsection{Transport equation and the amplitude.}
\label{ssec:amplitude}

We next seek an amplitude $b$ of the form 
\begin{equation}
\label{eq:amp_form_regu}
b(t, x_1, \tau, y; h, \zeta)=h^{-\frac{n-2}{4}}b_0(t, x_1, \tau; \zeta)\chi(y/\delta'),
\end{equation}
where $b_0\in C^\infty (\R_t\times \R_{x_1} \times [\tau_0-\delta, \tau_0+\delta])$ is independent of $y$ and satisfies the approximate transport equation
\begin{equation}
\label{eq:condition_b0_regu}
2\p_{x_1}b_0-2\beta\p_tb_0 - 2i\langle \nabla_{g_0}\Theta, \nabla_{g_0}b_0(x_1, \cdot)\rangle_{g_0} - i(\Delta_{g_0}\Theta)b_0 + \beta \overline{a_\zeta} b_0 =\mathcal{O}(|y|\zeta^{-1})
\end{equation}
as $y,\zeta \to 0$. The cut-off function $\chi \in C^\infty_0(\R^{n-2})$ in \eqref{eq:amp_form_regu} is chosen such that $\chi=1$ for $|y|\le 1/4$ and $\chi=0$ for $|y|\ge 1/2$.
Here $\zeta$ is the regularization parameter and $a_\zeta$ is the regularized damping coefficient given in Proposition \ref{prop:regularization}. Instead of the transport equation appearing in the coefficient of $s$ in \eqref{eq:conjugated_L1}, we solve \eqref{eq:condition_b0_regu} and obtain a smooth Gaussian beam quasimode. Eventually, we want the Gaussian beam to only depend on the semiclassical parameter $h$. Therefore, at the end of this subsection we will give an explicit relation between $h$ and $\zeta$.  We shall make the expression $|y|\zeta^{-1}$ rigorous in the next subsection when we prove estimate \eqref{eq:estimate_v}. Also, we will address the effect of the change from $a$ to $a_\zeta$ on the proof of \eqref{eq:estimate_v}. 
Equation \eqref{eq:condition_b0_regu}, combined with a scaling in the semiclassical parameter $h$, will be used in Subsection \ref{Ssec:local_estimate} to prove the third estimate in \eqref{eq:estimate_v}.  

In order to find a function $b_0$ such that \eqref{eq:condition_b0_regu} holds, we first compute $\langle \nabla_{g_0}\Theta, \nabla_{g_0}b_0(x_1, \cdot)\rangle_{g_0}$. It follows from \eqref{eq:phi} that
\begin{equation}
\label{eq:dtau_phi}
\p_\tau\Theta(\tau, y)=\sqrt{1-\beta^2}+\mathcal{O}(|y|^2).
\end{equation} 
Therefore, we get from \eqref{eq:g0_near_geo} that
\begin{equation}
\label{eq:dphidb0}
\langle \nabla_{g_0}\Theta, \nabla_{g_0}b_0(x_1, \cdot)\rangle_{g_0}=\sqrt{1-\beta^2}(\p_\tau b_0+H(\tau)y\cdot \p_yb_0)+\mathcal{O}(|y|^2)\p_\tau b_0+
\mathcal{O}(|y|^2)\p_y b_0.
\end{equation}

We next compute $\Delta_{g_0}\Theta$ near the geodesic $\gamma$. Using \eqref{eq:g0_near_geo} and \eqref{eq:phi}, we have
\[
(\Delta_{g_0}\Theta)(\tau, 0) = \sqrt{1-\beta^2}\delta^{jk}H_{jk} =  \sqrt{1-\beta^2} \tr H(\tau),
\]
which implies
\begin{equation}
\label{eq:lap_phi_geo}
(\Delta_{g_0}\Theta)(\tau, y)=\sqrt{1-\beta^2} \tr H(\tau)+\mathcal{O}(|y|).
\end{equation}

Finally, we Taylor expand the coefficients appearing on the left-hand side of \eqref{eq:condition_b0_regu}. Writing
\begin{align*}
a_\zeta(t, x_1, \tau, y)
=a_\zeta(t, x_1, \tau, 0)+\int_0^1 (\nabla_ya_\zeta(t, x_1, \tau, ys))yds,
\end{align*}
and utilizing \eqref{eq:est_diff_regu}, we get
\begin{equation}
\label{eq:Taylor_damping}
a_\zeta(t, x_1, \tau, y)=a_\zeta(t, x_1, \tau, 0)+\cO(|y|\zeta^{-1}).
\end{equation}

To achieve \eqref{eq:condition_b0_regu}, we require that $b_0(t, x_1, \tau; \zeta)$ satisfies
\begin{equation}
\label{eq:transport_a0_regu}
(\beta \p_t-\p_{x_1}+i\sqrt{1-\beta^2}\p_\tau)b_0=\frac{1}{2}[-i\sqrt{1-\beta^2}\tr H(\tau)+\beta \overline{a_\zeta}(t, x_1, \tau, 0)]b_0.
\end{equation}
To solve this equation we perform a change of variables and write the left-hand side of \eqref{eq:transport_a0_regu} as a $\overline{\p}$-equation. To that end, let $S: \R^3\to \R^3$ be an invertible linear function such that for a fixed $\beta \in (\frac{1}{\sqrt{3}}, 1)$, its inverse function is
\begin{equation}
\label{eq:change_variable}
S^{-1}(t,x_1,\tau)
=
\left(\frac{1}{\beta}t, x_1+\frac{1}{\beta}t, \frac{1}{\sqrt{1-\beta^2}}\tau\right):=(\tT, p, r).
\end{equation}
Then we get from \eqref{eq:change_variable} that
\begin{equation}
\label{eq:chain_rule}
\p_t=\frac{\p}{\p\tT}\frac{\p\tT}{\p t}+ \frac{\p}{\p p}\frac{\p p}{\p t}+ \frac{\p}{\p \tau}\frac{\p \tau}{\p t}=\frac{1}{\beta}(\p_{\tT}+\p_p) \quad \text{ and } \quad \p_{x_1}= \frac{\p}{\p\tT}\frac{\p\tT}{\p x_1}+\frac{\p}{\p p}\frac{\p p}{ \p x_1} + \frac{\p}{\p \tau}\frac{\p \tau}{\p x_1}=\p_p.
\end{equation}
By substituting \eqref{eq:change_variable} and \eqref{eq:chain_rule} into \eqref{eq:transport_a0_regu}, we obtain the equation
\begin{equation}
\label{eq:transport_a0_change_variable_regu}
(\p_{\tT}+i\p_r)b_0'=\frac{1}{2}[-i\sqrt{1-\beta^2}\tr H(\sqrt{1-\beta^2}r)+ \beta \overline{a_\zeta}'(\tT, p, r, 0)]b_0',
\end{equation}
where $a_\zeta'=a_\zeta\circ S$ and $b_0'=b_0\circ S$.


Writing
$
\overline{\p} = \frac{1}{2}(\p_{\tT}+i\p_r),
$
we look for a solution to \eqref{eq:transport_a0_change_variable_regu} of the form $b_0'(\tT, p, r; \zeta)=e^{\Phi_{1, \zeta}(\tT, p, r)+f_1(r)}$. By a direct computation, we see that in order for such a function $b_0'$ to solve \eqref{eq:transport_a0_change_variable_regu}, the functions $\Phi_{1, \zeta}$ and $f_1$ need to satisfy
\begin{equation}
\label{eq:dPhi_regu}
\overline{\p} \Phi_{1, \zeta}(\tT, p, r) = \frac{\beta}{4}  \overline{a_\zeta}'(\tT, p, r, 0)
=
\frac{\beta}{4}  \overline{a_\zeta}'(\tT, p, \gamma(r))
\end{equation}
and
\begin{equation} 
\label{eq:dtau_f}
\p_r f_1= -\frac{\sqrt{1-\beta^2}}{2}\tr H(\sqrt{1-\beta^2}r).
\end{equation}
Note that $f_1$ can be obtained by integrating the right-hand side of \eqref{eq:dtau_f} with respect to $r$. 

In order to solve the $\overline{\p}$-equation \eqref{eq:dPhi_regu}, we use the fundamental solution $E(\tT,r)=\frac{1}{\pi(\tT+ir)}$ of the $\overline{\p}$-operator \cite[Section 5.4]{Friedlander_Joshi} to take
\begin{equation}
\label{eq:fund_solution}
\Phi_{1, \zeta}(\tT, p, r)=\frac{\beta}{4} (E\ast  \overline{a_\zeta}')(\tT, p, \gamma(r)).
\end{equation}
While forming the convolution over the complex variable $\tT+ir$, we note that by Proposition \ref{prop:regularization}, the function $a_\zeta'$ is compactly supported in $\R^2 \times M_0^{\mathrm{int}}$. Since $\gamma$ is a nontangential geodesic in $(M_0,g_0)$, we may assume without loss of generality that the map $(\tT,p, r) \mapsto \overline{a_\zeta}'(\tT, p, \gamma(r))$ is smooth and compactly supported in the entire $(\tT,p, r)$-space so that estimate \eqref{eq:est_diff_regu} still holds. 
Therefore, we have obtained a $C^\infty$-smooth solution $b_0'(\tT, p, r; \zeta)=e^{\Phi_{1, \zeta}(\tT, p, r)+f_1(r)}$ of \eqref{eq:transport_a0_change_variable_regu} defined in the whole $(\tT, p, r)$-space. 

To verify that $b_0$ satisfies \eqref{eq:condition_b0_regu}, we need to estimate $b_0(\cdot; \zeta)$, as well as its first and second order derivatives over the set $[0,\frac{T}{\beta}] \times \overline{J}_p \times [r_0-\delta, r_0+\delta]$, where $J_p\subset \R$ is an open and bounded interval such that the respective $p$-coordinate of each point in $Q$ is in $J_p$. Since the function $\overline{a}_\zeta$ is supported in some open and bounded set of $\R^2\times M_0^{\mathrm{int}}$, as given in Proposition \ref{prop:regularization}, there exists some compact set $K \subset \R^2$ such that the following inequality holds for every $(\tT, p, r) \in [0,\frac{T}{\beta}] \times \overline{J}_p \times [0, \frac{L}{\sqrt{1-\beta^2}}]$,
\begin{equation}
\label{eq:est_for_Phi}
|\Phi_{1,\zeta}(\tT, p, r)|
\leq 
\int_{K}|E(\tT-t,r-s)||\overline{a}'_\zeta(t, p, s)|dtds
\leq 
\|E_{(\tT,r)}\|_{L^1(K)}\|\overline{a}'_\zeta\|_{L^\infty},
\end{equation}
where $E_{\tT,r}(t,s)=E(\tT-t,r-s)$. Due to the local integrability of $E$, the term $\|E_{(\tT,r)}\|_{L^1(K)}$ has a uniform bound for all $(\tT,r)\in [0,\frac{T}{\beta}] \times [0, \frac{L}{\sqrt{1-\beta^2}}]$. Then it follows from estimate \eqref{eq:est_diff_regu} that $\|\Phi_{1,\zeta}\|_{L^\infty}=\cO(1)$. Furthermore, by replacing the function $\overline{a}_\zeta$ with $\p_{\tT} \overline{a}_\zeta$ in \eqref{eq:est_for_Phi} and utilizing \eqref{eq:est_diff_regu} again, we get $\|\p_{\tT} \Phi_{1,\zeta}\|_{L^\infty}=o(\zeta^{-1})$. We also obtain the following estimates by using similar arguments:
\begin{align*}
\|\nabla_{g_0} \Phi_{1,\zeta}\|_{L^\infty}, \|\p_{p} \Phi_{1,\zeta}\|_{L^\infty}=o(\zeta^{-1}), \quad 
\|\Delta_{g_0} \Phi_{1,\zeta}\|_{L^\infty},\|\p_{\tilde t}^2 \Phi_{1,\zeta}\|_{L^\infty}=o(\zeta^{-2}), \quad  \zeta\to 0.
\end{align*}

To complete the verification of \eqref{eq:condition_b0_regu}, we connect the semiclassical parameter $h$ and the regularization parameter $\zeta$ by setting $\zeta=h^\alpha$, $0<\alpha<\frac{1}{2}$. Note that the change of coordinates function $S$, given by \eqref{eq:change_variable}, is independent of $\zeta$. Hence, this choice of $\zeta$, in conjunction with estimates
\eqref{eq:est_diff_regu}, \eqref{eq:fund_solution}, and $\|f_1\|_{L^\infty}=\cO(1)$, yields
\begin{equation}
\label{eq:bound_deri_b0_h_old}
\begin{aligned}
&\|b_0(\cdot; h)\|_{L^\infty}=\cO(1),
\quad \|\nabla_{g_0} b_0(\cdot; h)\|_{L^\infty},\|\p_t b_0(\cdot; h)\|_{L^\infty}, \|\p_{x_1} b_0(\cdot; h)\|_{L^\infty}=o(h^{-\alpha}), 
\\
&\|\Delta_{g_0} b_0(\cdot; h)\|_{L^\infty},\|\p_t^2 b_0(\cdot; h)\|_{L^\infty}=o(h^{-2\alpha}), \quad  h\to 0.
\end{aligned}
\end{equation}

By substituting \eqref{eq:dtau_phi}--\eqref{eq:transport_a0_regu} into the left-hand side of \eqref{eq:condition_b0_regu}, we get from \eqref{eq:bound_deri_b0_h_old} that
\begin{align*}
&2\p_{x_1}b_0-2\beta\p_tb_0 - 2i\langle \nabla_{g_0}\Theta, \nabla_{g_0}b_0(x_1, \cdot)\rangle_{g_0} - i(\Delta_{g_0}\Theta)b_0 + \beta \overline{a_\zeta} b_0
\\
&= -2i\sqrt{1-\beta^2}H(\tau)y\cdot \p_yb_0+\cO(|y|^2)\p_\tau b_0+\mathcal{O}(|y|^2)\p_y b_0+\cO(|y|)+\cO(|y|h^{-\alpha}).
\\
&= \cO(|y|h^{-\alpha}).
\end{align*}
Thus, equation \eqref{eq:condition_b0_regu} is verified.


\subsubsection{Local verification of estimate  $\eqref{eq:estimate_v}$.}
\label{Ssec:local_estimate}

Let us now verify that the estimates in \eqref{eq:estimate_v} hold for the quasimode
\begin{equation}
\label{eq:beam_form_regu}
v_s(\tT, p, r, y;h)=e^{is\Theta(\sqrt{1-\beta^2}r, y)}b'(\tT, p, r,y; h)=e^{is\Theta(\sqrt{1-\beta^2}r, y)}h^{-\frac{n-2}{4}}b_0'(\tT, p, r; h)\chi(y/\delta')
\end{equation}
in the open set $(0,\frac{T}{\beta}) \times J_p \times \Omega$ of $Q$, where $\Omega\subset M_0$ is the domain of Fermi coordinates near the point $z_0=\gamma(\tau_0)$. 
To establish this, we shall need the following estimate for any 
$k\geq 0$:
\begin{equation}
\label{eq:L2_power_y}
\begin{split}
\|h^{-\frac{n-2}{4}} |y|^k e^{-\frac{\Im \Theta}{h}}\|_{L^2(|y| \le \delta'/2)}
&\leq 
\|h^{-\frac{n-2}{4}} |y|^k e^{-\frac{d}{h}|y|^2}\|_{L^2(|y| \le \delta'/2)}
\\
& \leq 
\bigg(\int_{\R^{n-2}} h^{k}|z|^{2k} e^{-2d|z|^2}dz\bigg)^{1/2} 
\\
&=\cO(h^{k/2}), \quad h \to 0.
\end{split}
\end{equation}
Here we applied estimate \eqref{eq:imaginarypart_phi} and the change of variables $z=h^{-1/2}y$.

We are now ready to start verifying \eqref{eq:estimate_v} locally. To that end, we use \eqref{eq:imaginarypart_phi}, \eqref{eq:bound_deri_b0_h_old}, and \eqref{eq:L2_power_y} with $k=0$ to get
\begin{equation}
\label{eq:estimate_v_int}
\begin{aligned}
\|v_s\|_{L^2([0,\frac{T}{\beta}]\times \overline{J}_p \times \Omega)} &\le \|b_0'\|_{L^\infty([0, \frac{T}{\beta}]\times \overline{J}_p\times [r_0-\delta, r_0+\delta])} \|e^{is\Theta}h^{-\frac{n-2}{4}}\chi(y/\delta')\|_{L^2([0, \frac{T}{\beta}]\times \overline{J}_p \times \Omega)}
\\
&\le \mathcal{O}(1)\|h^{-\frac{n-2}{4}}e^{-\frac{d}{h} |y|^2}\|_{L^2(|y|\le \delta'/2)}=\mathcal{O}(1), \quad h\to 0.
\end{aligned}
\end{equation}

Let us next estimate $\|\p_tv_s\|_{L^2([0, T]\times \overline{J}_p \times \Omega)}$. By utilizing \eqref{eq:bound_deri_b0_h_old}, \eqref{eq:beam_form_regu}, and 
$\zeta=h^{\alpha}$, $0<\alpha <\frac{1}{2}$, we obtain
\begin{equation}
\label{eq:est_dt}
\|\p_tv_s\|_{L^2([0, \frac{T}{\beta}]\times \overline{J}_p\times \Omega)}=o(h^{-1/2}), \quad h\to 0.
\end{equation}

We now proceed to estimate $\|e^{s(\beta t+x_1)}h^2 \mathcal{L}_{g,-\overline{a},\overline{q}-\p_t\overline{a}}e^{-s(\beta t+x_1)}v_s\|_{L^2([0, \frac{T}{\beta}]\times J_p\times \Omega)}$ by estimating each term on the right-hand side of \eqref{eq:conjugated_L1} independently. Let us start with the first term. By applying
\eqref{eq:eikonel_eq},  \eqref{eq:imaginarypart_phi}, \eqref{eq:bound_deri_b0_h_old}, and \eqref{eq:L2_power_y} with $k=3$, we get
\begin{equation}
\label{eq:est_first_term}
\begin{aligned}
&h^2\|e^{is\Theta}s^2 (\langle \nabla_{g_0}\Theta, \nabla_{g_0}\Theta\rangle_{g_0} - (1 - \beta^2))b\|_{L^2([0, \frac{T}{\beta}]\times  \overline{J}_p \times \Omega)}
\\
&=h^2\|e^{is\Theta}s^2 h^{-\frac{n-2}{4}}(\langle \nabla_{g_0}\Theta, \nabla_{g_0}\Theta\rangle_{g_0} - (1 - \beta^2))b_0'\chi(y/\delta')\|_{L^2([0, \frac{T}{\beta}]\times  \overline{J}_p \times \Omega)}
\\
&\le \cO(1) \|h^{-\frac{n-2}{4}} |y|^3 e^{-\frac{d}{h} |y|^2}\|_{L^2(|y| \le \delta'/2)} =\cO(h^{3/2}), \quad  h \to 0.
\end{aligned}
\end{equation}

We next consider the second term on the right-hand side of \eqref{eq:conjugated_L1}. From a direct computation, we see that
\[
|e^{is\Theta}|=e^{-\frac{1}{h}\Im \Theta}e^{-\lambda \Re \Theta}=e^{-\frac{\sqrt{1-\beta^2}}{2h} \Im H(\tau)y\cdot y}e^{-\lambda\sqrt{1-\beta^2}\tau}e^{-\lambda \cO(|y|^2)}.
\]
We observe that $e^{-\frac{1}{h}}=\cO(h^\infty)$. Therefore, on the support of $ \nabla_{g_0}\chi(y/\delta')$ we deduce from \eqref{eq:imaginarypart_phi} that
\[
|e^{is\Theta}|
\leq  
e^{-\frac{\tilde d}{h}} \quad \text{for some } \tilde d>0.
\]
Thus, using estimates \eqref{eq:condition_b0_regu} and  \eqref{eq:L2_power_y} with $k=1$, $\alpha\in (0,\frac{1}{2})$, and the triangle inequality, we have
\begin{equation}
\label{eq:est_second_term}
\begin{aligned}
&h^2 \|e^{is\Theta}s(2\p_{x_1}b-2\beta\p_tb - 2i\langle \nabla_{g_0}\Theta, \nabla_{g_0}b(x_1, \cdot)\rangle_{g_0} - i(\Delta_{g_0}\Theta)b+ \beta \overline{a_\zeta} b)\|_{L^2([0, \frac{T}{\beta}]\times  \overline{J}_p \times \Omega)}
\\
&\le \cO(h) \|e^{is\Theta}h^{-\frac{n-2}{4}} [|y|h^{-\alpha}\chi(y/\delta')-2i\n{\nabla_{g_0}\Theta, \nabla_{g_0}\chi(y/\delta')}_{g_0}]\|_{L^2([0, \frac{T}{\beta}]\times  \overline{J}_p \times \Omega)}
\\
& \le \cO(h) \|h^{-\frac{n-2}{4}} |y| h^{-\alpha} e^{-\frac{d}{h}|y|^2}\|_{L^2(|y| \le \delta'/2)} +\cO(e^{-\frac{\tilde{d}}{h} })
\\
&=\cO(h^{3/2-\alpha})=o(h), \quad  h \to 0.
\end{aligned}
\end{equation}
We want to emphasize that in the second term on the right-hand side of \eqref{eq:conjugated_L1} we have the damping coefficient $\bar a$ instead of its smooth approximation $\overline{a_\zeta}$, which appeared in \eqref{eq:est_second_term}. To medicate this discrepancy, we use estimates \eqref{eq:est_diff_regu} and \eqref{eq:L2_power_y} with $k=0$ to get
\begin{equation}
\label{eq:est_second_term_regu}
\begin{aligned}
h^2 \|e^{is\Theta}s \beta (\overline{a}-\overline{a_\zeta}) b\|_{L^2([0, \frac{T}{\beta}]\times  \overline{J}_p \times \Omega)} &= \cO(h) \|e^{is\Theta} (\overline{a}-\overline{a_\zeta}) h^{-\frac{n-2}{4}}b_0'\chi(y/\delta')\|_{L^2([0, \frac{T}{\beta}]\times  \overline{J}_p \times \Omega)}
\\
&\le \cO(h)\|\overline{a}-\overline{a_\zeta}\|_{L^\infty([0, \frac{T}{\beta}]\times  \overline{J}_p \times \Omega)} \|h^{-\frac{n-2}{4}} e^{-\frac{d}{h} |y|^2}\|_{L^2(|y| \le \delta'/2)}
\\
&=o(h),\quad  h \to 0.
\end{aligned}
\end{equation}

Finally, we estimate the third term on the right-hand side of \eqref{eq:conjugated_L1}. To that end, we utilize estimates \eqref{eq:bound_deri_b0_h_old} and \eqref{eq:L2_power_y} with $k=0$ to obtain
\begin{equation}
\label{eq:est_dtsquare}
h^2 \|e^{is\Theta}\p_t^2b\|_{L^2([0, \frac{T}{\beta}]\times  \overline{J}_p \times \Omega)}=o(h^{2(1-\alpha)})=o(h), \quad h\to 0.
\end{equation}
To estimate the term involving the $\Delta_g$, we incorporate estimates \eqref{eq:amp_form_regu}, \eqref{eq:bound_deri_b0_h_old}, and \eqref{eq:L2_power_y} with $k=0$, as well as the triangle inequality, to get
\begin{equation}
\label{eq:est_laplacian}
\begin{aligned}
h^2 &\|e^{is\Theta}(-\Delta_gb)\|_{L^2([0, \frac{T}{\beta}]\times  \overline{J}_p \times \Omega)}
\\
\le &\cO(h^2)\|h^{-\frac{n-2}{4}} e^{is\Theta}  \chi(y/\delta')\Delta_g b_0'\|_{L^2([0, \frac{T}{\beta}]\times  \overline{J}_p \times \Omega)}
\\
&+\cO(h^2)\|h^{-\frac{n-2}{4}} e^{is\Theta} [b_0' \Delta_g  \chi(y/\delta')+2\n{\nabla_g b_0', \nabla_g \chi(y/\delta')}_g]\|_{L^2([0, \frac{T}{\beta}]\times  \overline{J}_p \times \Omega)}
\\
\le &\cO(h^2)\left(\|h^{-\frac{n-2}{4}} e^{-\frac{d}{h} |y|^2}h^{-2\alpha}\|_{L^2(|y| \le \delta'/2)} + \cO(e^{-\frac{\tilde{d}}{h}})\right)
\\
= &\cO(h^{2(1-\alpha)})+\cO(e^{-\frac{\tilde{d}}{h}})=o(h), \quad h\to 0. 
\end{aligned}
\end{equation}
For the lower order terms, it follows from \eqref{eq:bound_deri_b0_h_old} that
\begin{equation}
\label{eq:est_first_and_zero}
h^2 \|e^{is\Theta}(-\overline{a}\p_tb+(\overline{q}-\p_t\overline{a})b)\|_{L^2([0, \frac{T}{\beta}]\times  \overline{J}_p \times \Omega)}=o(h^{2-\alpha})=o(h^{3/2}), \quad h\to 0.
\end{equation}
Therefore, by combining estimates \eqref{eq:est_first_term}--\eqref{eq:est_first_and_zero}, we conclude from  \eqref{eq:conjugated_L1} that 
\begin{equation}
\label{eq:est_op_regu}
\|e^{s(\beta t+x_1)}h^2\mathcal{L}_{g, -\overline{a}, \overline{q}-\p_t\overline{a}}e^{-s(\beta t+x_1)}v_s\|_{L^2([0, \frac{T}{\beta}]\times  \overline{J}_p \times \Omega)}=o(h), \quad h\to 0.
\end{equation}
This completes the verification of estimate \eqref{eq:estimate_v} locally in the set $(0, \frac{T}{\beta})\times J_p\times \Omega$.

Before proceeding to the global construction, we need an estimate for $\|v(t, x_1, \cdot)\|_{L^2(\pM_0)}$ for later purposes. If $\Omega$ contains a boundary point $x_0=(\tau_0, 0) \in \p M_0$, then $\dot \gamma(\tau_0)$ is transversal to $\p M_0$. Let $\rho$ be a boundary defining function for $M_0$ so that $\p M_0$ is given by the level set $\rho(r, y)=0$ near $x_0$, and $\nabla_{g_0} \rho$ is normal to $\p M_0$ at $x_0$. These imply $\p_\tau \rho(x_0)\ne 0$. By the implicit function theorem, there exists a smooth function $y\mapsto r(y)$ near 0 such that $\p M_0$ near $x_0$ is given by $\{(r(y), y):|y|<r_0\}$ for some $r_0>0$ small, see the proof of \cite[Proposition 7.5]{Kenig_Salo}. 

Using \eqref{eq:property_of_phi}, \eqref{eq:imaginarypart_phi}, and \eqref{eq:bound_deri_b0_h_old}, we see that there exists a constant $C$ such that
\begin{equation}
\label{eq:bound_of_v}
|v_s(t, x_1, \tau, y;h)|\le Ch^{-\frac{n-2}{4}}e^{-\frac{d}{h} |y|^2}\chi(y/\delta').
\end{equation}
Thus, after shrinking the set $\Omega$ if necessary and using \eqref{eq:L2_power_y} with $k=0$ along with \eqref{eq:bound_of_v}, we get
\begin{equation}
\label{eq:est_v_bdyM}
\begin{aligned}
\|v(t, x_1, \cdot)\|^2_{L^2(\pM_0 \cap \Omega)} &= \int_{|y|_e<r_0} |v(t, x_1, r(y), y)|^2dS_g(y)
\\
&\le \cO(1) \int_{\R^{n-2}} h^{-\frac{n-2}{2}}e^{-\frac{2d}{h} |y|^2} dy =\cO(1), \quad h\to 0.
\end{aligned}
\end{equation}

\subsubsection{Global construction of the Gaussian beam quasimodes via gluing.}
\label{sssec:global_const}
Finally, we glue together the quasimodes defined along small pieces of the geodesic $\gamma$ to obtain the quasimode $v_s$ in 
$\R^2\times M_0^{\mathrm{int}}$. Since $\hat{M}_0$ is compact and $\gamma(r):(-2\varepsilon, \frac{L}{\sqrt{1-\beta^2}}+2\varepsilon)\to \hat{M}_0$ is a nontangential geodesic that is not a loop, 
it follows from \cite[Lemma 7.2]{Kenig_Salo} and the choice of $\varepsilon>0$ that the curve $\gamma|_{(-2\varepsilon, \frac{L}{\sqrt{1-\beta^2}}+2\varepsilon)}$ 
has finitely many self intersection times $r_\ell\geq 0$ with $\ell \in \{1,\ldots,R\}$ and 
\[
-\varepsilon =r_0< r_1< \dots<r_R<r_{R+1} =\frac{L}{\sqrt{1-\beta^2}}+\varepsilon.
\]
Due to 
\cite[Lemma 7.4]{Kenig_Salo},
there exists an open cover $\{(\Omega_\ell, \kappa_\ell)_{\ell=0}^{R+1}\}$ of $\gamma([-\varepsilon, \frac{L}{\sqrt{1-\beta^2}}+\varepsilon])$ consisting of Fermi coordinate neighborhoods that have the following properties:
\begin{enumerate}
\item[(1)] $ \kappa_\ell(\Omega_\ell)=I_\ell\times B$, where $I_\ell$ are open intervals and $B=B(0, \delta')$ is an open ball in $\R^{n-2}$. Here $\delta'>0$ can be taken arbitrarily small and the same for each $\Omega_\ell$.
\item[(2)] $ \kappa_\ell(\gamma(r))=(r, 0)$ for $r \in I_\ell$.
\item[(3)] $r_\ell$ only belongs to $I_\ell$ and $\overline{I_\ell}\cap \overline{I_k}=\emptyset$ unless $|\ell-k|\le 1$.
\item[(4)] $ \kappa_\ell= \kappa_k$ on $\kappa_\ell^{-1}((I_\ell\cap I_k)\times B)$.
\end{enumerate}
In particular, 
the intervals $I_0$ and $I_{R+1}$ are chosen in such a way that they do not contain any self-intersection times.
In the case when $\gamma$ does not self-intersect, there is a single coordinate neighborhood of $\gamma|_{[-\varepsilon, \frac{L}{\sqrt{1-\beta^2}}+\varepsilon]}$ such that (1) and (2) are satisfied.

We proceed as follows to construct the quasimode $v_s$. Suppose first that $\gamma$ does not self-intersect at $r=0$. Using the procedure from the earlier part of this proof, we find a quasimode
\[
v_s^{(0)}(\tT, p, r, y;h)=h^{-\frac{n-2}{4}}e^{is\Theta^{(0)}(\sqrt{1-\beta^2}r, y)}e^{\Phi_{1, h}(\tT, p, r)+f_1(r)}\chi(y/\delta')
\]
in $\Omega_0$ with some fixed initial conditions at $r=-\varepsilon$ for the Riccati equation \eqref{eq:Riccati_eq} determining $\Theta^{(0)}$. We now choose some $r_0'$ such that $\gamma(r_0')\in \Omega_0\cap \Omega_1$ and let
\[
v_s^{(1)}(\tT, p, r, y;h)=h^{-\frac{n-2}{4}}e^{is\Theta^{(1)}(\sqrt{1-\beta^2}r, y)}e^{\Phi_{1,h}(\tT, p, r)+f_1(r)}\chi(y/\delta')
\] 
be the quasimode in $\Omega_1$ by choosing the initial conditions for 
$\Theta^{(1)}$ such that $\Theta^{(1)}(r_0')=	\Theta^{(0)}(r_0')$. 
Here we have used the same functions $\Phi_{1,h}$ and $f_1$ in $v_s^{(0)}$ and $v_s^{(1)}$ since $\Phi_{1,h}$ and $f_1$ are both globally defined for all $r\in (-2\varepsilon, \frac{L}{\sqrt{1-\beta^2}}+2\varepsilon)$, and neither of the functions depends on $y$. On the other hand, since the equations determining the phase functions $\Theta^{(0)}$ and $\Theta^{(1)}$ have the same initial data in $\Omega_0$ and in $\Omega_1$, and the local coordinates $ \kappa_0$ and $ \kappa_1$ coincide on $\kappa_0^{-1}((I_0\cap I_1)\times B)$, we have
$\Theta^{(1)}=\Theta^{(0)}$ in $\Omega_0\cap \Omega_1$.
Therefore, we conclude that $v_s^{(0)}=v_s^{(1)}$ in the overlapped region $\Omega_0\cap \Omega_1$. Continuing in this way, we obtain quasimodes $v_s^{(2)}, \dots, v_s^{(R+1)}$ such that
\begin{equation}
\label{eq:quasi_corres}
v_s^{(\ell)}(\tilde{t}, p, \cdot)=v_s^{(\ell+1)}(\tilde{t}, p, \cdot) \quad \text{in} \quad \Omega_\ell\cap \Omega_{\ell+1}
\end{equation}
for all $\tT$ and $p$. If $\gamma$ self-intersects at $r=0$, we start the construction from $v^{(1)}$ by fixing initial conditions for \eqref{eq:Riccati_eq}\label{key} at $r=0$ and find $v^{(0)}$ by going backwards.

Let $\chi_j(r)$ be a partition of unity subordinate to the intervals $(I_\ell)_{\ell=0}^{R+1}$. We denote $\tilde \chi_\ell(\tilde t, p,r, y)=\chi_\ell(r)$ and define a smooth function
\[
v_s=\sum_{\ell=0}^{R+1} \tilde{\chi}_\ell v_s^{(\ell)} \quad \text{in $\R^2\times \hat M_0$.}
\]


Let $z_1, \dots, z_{R'}\in M_0$, $R'< R$, 
be the distinct self-intersection points of $\gamma$, corresponding to the self intersection times $0\le r_1<\cdots<r_R$. Let $V_j$ be a small neighborhood in $\hat{M}_0$ centered at $z_j$ for $j\in \{1, \dots, R'\}$. We proceed as in \cite[Proposition 7.5]{Kenig_Salo}, and use the definition of the intervals $I_0,\ldots, I_{R+1}$ and \eqref{eq:quasi_corres} to pick a finite cover $W_0,\ldots W_S$ of the remaining points on the geodesic $\gamma$ such that for each $k\in \{1,\ldots,S\}$ we have $W_k\subset \Omega_{\ell(k)}$ for some $\ell(k)\in \{0,\ldots,R+1\}$. This gives us an open cover for $\supp(v_s(\tT, p, \cdot))\cap M_0$
\[
\supp(v_s(\tT, p, \cdot))\cap M_0 \subset \left(\bigcup_{j=1}^{R'}V_j\right) \cup \left(\bigcup_{k=0}^{S}W_{k}\right),
\]
and the quasimode restricted to $V_j$ and $W_{k}$ is of the form
\begin{equation}
\label{eq:finite_sum_v}
v_s(\tT, p,\cdot)|_{V_j}=\sum_{\ell: \gamma(r_\ell)=z_j}v_s^{(\ell)}(\tT, p,\cdot) \quad \text{ and } \quad v_s(\tT, p, \cdot)|_{W_{k}}=v_s^{(\ell(k))}(\tT, p, \cdot),
\end{equation}
respectively. Here \eqref{eq:finite_sum_v} follows from the construction of the intervals  $(I_\ell)_{\ell=0}^{R+1}$, the partition of unity subordinate to these intervals, and choosing the set $V_j$ small enough. 

Since in both cases of \eqref{eq:finite_sum_v} the function $v_s$ is a finite sum of $v^{(\ell)}$, the estimate
\[
\|v_s(\tT, p,\cdot)\|_{L^2(\p M_0)}=\cO(1)
\]
and those in \eqref{eq:estimate_v} follow from the corresponding local considerations \eqref{eq:est_dt}, \eqref{eq:est_op_regu}, and \eqref{eq:est_v_bdyM} for each of $v_s^{(\ell)}$, respectively. 
This completes the construction of the Gaussian beam quasimode $v_s$.

\subsubsection{Construction of a Gaussian beam quasimode for the operator $e^{-s(\beta t+x_1)}\mathcal{L}_{g, a,q}e^{s(\beta t+x_1)}$.}
We next seek a Gaussian beam quasimode for the operator 
$e^{-s(\beta t+x_1)}\mathcal{L}_{g, a,q}e^{s(\beta t+x_1)}$ of the form
\[
w_s(t, x_1, \tau, y;h, \zeta)=e^{is\Theta(\tau,y)} B(t, x_1, \tau, y;h, \zeta)
\]
with the phase function $\Theta \in C^\infty(\Omega, \C)$ satisfying \eqref{eq:property_of_phi} and $B (t, x_1, \tau, y)\in C^\infty(\R \times \R \times \Omega)$ supported near $\Gamma$. By replacing $s$ in \eqref{eq:conjugated_dt} and \eqref{eq:conjugated_Laplacian} by $-s$, and recalling that $\Theta$ is independent of $x_1$, we obtain
\begin{equation}
\label{eq:conjugated_L2}
\begin{aligned}
&e^{-s(\beta t+x_1)}\mathcal{L}_{g, a,q}e^{s(\beta t+x_1)}w_s
\\
&= e^{is\Theta}[s^2(\langle \nabla_{g_0}\Theta, \nabla_{g_0}\Theta\rangle_{g_0}-(1-\beta^2))B
 \\
&+s(-2\p_{x_1}B+2\beta\p_tB-2i\langle \nabla_{g_0}\Theta, \nabla_{g_0}B(x_1, \cdot)\rangle_{g_0} -i(\Delta_{g_0}\Theta)B + \beta a B)
\\
&+\mathcal{L}_{g,a,q}B].
\end{aligned}
\end{equation}

Since the eikonal equation above is identical to the one in \eqref{eq:conjugated_L1}, we see from Subsection \ref{ssec:phase_function} that the phase function $\Theta$ is given by \eqref{eq:phi}.
We next find the amplitude $B$ in the form of 
\begin{equation}
\label{eq:form_B}
B(t, x_1, \tau, y;h)=h^{-\frac{n-2}{4}}B_0(t, x_1, \tau;h) \chi(y/\delta'),
\end{equation}
where $B_0\in C^\infty([\R \times\R\times \{\tau:|\tau-\tau_0|< \delta\})$. To that end, by proceeding similarly as in the construction of $b_0$ in Subsection \ref{ssec:amplitude}, we require that $B_0$ solves
\begin{equation}
\label{eq:transport_B0}
(\beta\p_t-\p_{x_1}-i\sqrt{1-\beta^2}\p_\tau) B_0 =\frac{1}{2}[(i\sqrt{1-\beta^2}\tr H(\tau)-\beta a(t, x_1, \tau, 0)] B_0.
\end{equation}
Using change of coordinates \eqref{eq:change_variable} again, we get
\begin{equation}
\label{eq:transport_B0_change}
(\p_{\tT}-i\p_r) B_0'=\frac{1}{2}[i\sqrt{1-\beta^2}\tr H(\sqrt{1-\beta^2}r)-\beta a_\zeta' (\tT, p, r, 0)] B_0',
\end{equation}
where $B_0'=B_0\circ S$ and $a_\zeta'=a_\zeta\circ S$.

By writing $\p=\frac{1}{2}(\p_{\tT}-i\p_r)$
and looking for a solution of the form $B_0=e^{\Phi_2(\tT, p, r)+f_2(r)}\eta(\tT, p, r)$ with $\p\eta=0$, we see that the functions $\Phi_{2,\zeta}$ and $f_2$ must satisfy
\begin{equation}
\label{eq:dPhi2}
\p\Phi_{2,\zeta}=-\frac{1}{4}\beta a_\zeta'(\tT, p,\gamma(r))
\end{equation}
and
\begin{equation}
\label{eq:dtau_f2}
\p_rf_2 = -\frac{\sqrt{1-\beta^2}}{2}\tr H(\sqrt{1-\beta^2}r).
\end{equation}
Using similar arguments as in the construction of $v_s$, we obtain a Guassian beam quasimode $w_s\in C^\infty(Q)$ such that the estimates in \eqref{eq:estimate_w} hold. 

This completes the proof of Theorem \ref{prop:Gaussian_beam}.

\subsection{Concentration property of the Gaussian beam quasimodes}
\label{Ssec:concetration}

By the proof of Theorem \ref{prop:Gaussian_beam}, for each nontangential geodesic $ \gamma\colon[0, \frac{L}{\sqrt{1-\beta^2}}]\to M_0$ and $h>0$ there exist smooth functions $\Phi_{1,h},\Phi_{2,h}$ in $[0,\frac{T}{\beta}] \times \overline{J}_p \times [0, \frac{L}{\sqrt{1-\beta^2}}]$ satisfying
\[
(\p_{\tT}+i\p_r) \Phi_{1,h}(\tT, p, r) = \frac{1}{2} \beta \overline{a}_{h}'(\tT, p, \gamma(r)) \text{ and } 
(\p_{\tT}-i\p_r) \Phi_{2,h}(\tT, p, r) = -\frac{1}{2} \beta a'_{h}(\tT, p, \gamma(r)),
\]
where $a_h'=a_h\circ S$, and the change of coordinates $S$ is given by \eqref{eq:change_variable}. Here $J_p\subset \R$ is an open and bounded interval such that for each point in $Q$ the respective $p$-coordinate is in $J_p$. In the next lemma we study the behavior for these functions as $h \to 0$.
\begin{lem}
\label{lem:Phi_nonregularized}
Let $\beta \in (\frac{1}{\sqrt{3}},1)$, and let $\gamma\colon[0, \frac{L}{\sqrt{1-\beta^2}}]\to M_0$ be a nontangential geodesic in $(M_0,g_0)$ as in Proposition \ref{prop:Gaussian_beam}. Then there exist continuous functions $\Phi_1$ and $\Phi_2$ in $[0,\frac{T}{\beta}] \times \overline{J}_p \times [0, \frac{L}{\sqrt{1-\beta^2}}]$ that satisfy
\begin{equation}
\label{eq:transport_eqns}
(\p_{\tT}+i\p_r) \Phi_1(\tT, p, r) = \frac{1}{2} \beta \overline{a}'(\tT, p, \gamma(r)) \text{ and } 
(\p_{\tT}-i\p_r) \Phi_2(\tT, p, r) = -\frac{1}{2} \beta a'(\tT, p, \gamma(r)),
\end{equation}
respectively. Furthermore, the following estimate holds:
\begin{equation}
\label{eq:bound_diff_phi}
\|\Phi_{j, h}-\Phi_j\|_{L^\infty([0,\frac{T}{\beta}] \times \overline{J}_p \times [0, \frac{L}{\sqrt{1-\beta^2}}])}=o(1), \quad  j=1,2,\quad h \to 0.
\end{equation}
\end{lem}

\begin{proof}
With a slight abuse of notation, we consider the compactly supported function 
\\
$\bar a '(\tT,p,r) =\overline{a}'(\tT, p, \gamma(r))$  in $\R^3$ and define a continuous function 
$
\Phi_1(\tT, p, r)=\frac{\beta}{4} (E\ast  \overline{a}')(\tT, p, r).
$
Here the convolution is taken over the complex variable $\tT+ir$.
Since $E:=\frac{1}{\pi(\tT+ir)}$ is the fundamental solution for the $\bar \p$-operator, we see that
$
\overline{\p} \Phi_1 = \frac{1}{4} \beta \overline{a}'.
$
Last, estimate \eqref{eq:bound_diff_phi} follows from the local integrability
of $E$, estimate \eqref{eq:est_diff_regu}, and an inequality analogous to \eqref{eq:est_for_Phi}. The analogous claims for $j=2$ follow by the same arguments. This completes the proof of Lemma \ref{lem:Phi_nonregularized}.
\end{proof}

In the following theorem we show that a Gaussian beam quasimode concentrates along the geodesic in the limit $h \to 0$. 

\begin{thm}
\label{prop:limit_behavior}
Let $s=\frac{1}{h}+i\lambda$, $0<h\ll 1$, $\lambda \in \R$ fixed, and $\beta \in (\frac{1}{\sqrt{3}},1)$. Let  $\gamma\colon[0, \frac{L}{\sqrt{1-\beta^2}}]\to M_0$ be a nontangential geodesic in $(M_0,g_0)$ as in Proposition \ref{prop:Gaussian_beam}. Let $J_p$ be as above.
Let $v_s$ and $w_s$ be the quasimodes from Proposition \ref{prop:Gaussian_beam}. Then for each $\psi \in C(M_0)$ and
$(\tT', p')\in  [0,\frac{T}{\beta}]\times \overline{J}_p$ 
we have
\begin{equation}
\label{eq:limit_prod_vw}
\begin{aligned}
\lim_{h\to 0} \int_{M_0} \overline{v_s}(\tT', p', \cdot)w_s(\tT', p', \cdot) \psi dV_{g_0} 
= 
(1-\beta^2)^{-\frac{n-6}{4}}&\int_{0}^{\frac{L}{\sqrt{1-\beta^2}}} e^{-2(1-\beta^2)\lambda r}e^{\overline{\Phi_1}(\tT', p', r)+\Phi_2(\tT', p', r)}
\\
&\quad \quad \quad \times \eta(\tT', p',r)\psi(\gamma(r))dr.
\end{aligned}
\end{equation}
Here the functions $\Phi_1, \Phi_2\in C( [0,\frac{T}{\beta}] \times \overline{J}_p \times [0, \frac{L}{\sqrt{1-\beta^2}}])$ are as in Lemma \ref{lem:Phi_nonregularized},
and $\eta \in C^\infty( [0,\frac{T}{\beta}] \times \overline{J}_p \times [0, \frac{L}{\sqrt{1-\beta^2}}])$ with $(\p_{\tT}-i\p_r)\eta=0$.
\end{thm}

\begin{proof}
By a partition of unity, it suffices to verify \eqref{eq:limit_prod_vw} for $\psi\in C_0(V_j\cap M_0)$ and $\psi\in C_0(W_k\cap M_0)$, where $V_j$ and $W_k$ are the same as in Subsection \ref{sssec:global_const} in the proof of Theorem \ref{prop:Gaussian_beam}.

\textit{Case 1: $\psi\in C_0(W_k\cap M_0)$}. We first consider the easier case that $\psi\in C_0(W_k\cap M_0)$ for some $k$. Here $\supp \psi$ may extend to $\p M_0$, and we extend $\psi$ by zero outside $W_k\cap M_0$. 
Using the quasimodes $v_s$ and $w_k$  given in \eqref{eq:finite_sum_v}, by arguing similarly as in the proof of Theorem \ref{prop:Gaussian_beam}, we obtain Gaussian beam quasimodes
\begin{equation}
\label{eq:quasimode_supp_psi}
\begin{aligned}
&v_s(\tT, p, r, y)=e^{is\Theta(\sqrt{1-\beta^2}r, y)}h^{-\frac{n-2}{4}}e^{\Phi_{1,h}(\tT, p, r)+f_1(r)}\chi(y/\delta'),
\\
&w_s(\tT, p, r, y)=e^{is\Theta(\sqrt{1-\beta^2}r, y)}h^{-\frac{n-2}{4}} e^{\Phi_{2,h}(\tT, p, r)+f_2(r)}\eta(\tilde{t},p, r)\chi(y/\delta').
\end{aligned}
\end{equation}

In order to establish \eqref{eq:limit_prod_vw}, we substitute the quasimodes given in \eqref{eq:quasimode_supp_psi} directly into the left-hand side of \eqref{eq:limit_prod_vw}. After that, we apply the dominated convergence theorem to simplify our computation since the term $e^{-\frac{1}{h} \sqrt{1-\beta^2}\Im H(\sqrt{1-\beta^2}r)y\cdot y}$ dominates the other exponentials in the phase function. Then we utilize properties of the solution $H(r)$ to the Riccati equation \eqref{eq:Riccati_eq}, as well as the definitions of functions $f_1, f_2$ given by \eqref{eq:dtau_f} and \eqref{eq:dtau_f2}, respectively.

Let us now provide the detailed proof. Using \eqref{eq:g0_near_geo}, we see that the determinant satisfies
\begin{equation}
\label{eq:g0_near_geo_r}
|g_0(r, y)|^{1/2}=\sqrt{1-\beta^2}+\cO(|y|^2).
\end{equation}
It then follows from \eqref{eq:phi}, 
\eqref{eq:change_variable}, \eqref{eq:quasimode_supp_psi}, and  \eqref{eq:g0_near_geo_r} that
\begin{equation}
\label{eq:prod_computation}
\begin{aligned}
&\int_{M_0} \overline{v_s}(\tT', p', \cdot)w_s(\tT', p', \cdot) \psi dV_{g_0} 
\\
=& \sqrt{1-\beta^2}\int_{0}^{\frac{L}{\sqrt{1-\beta^2}}}\int_{\R^{n-2}}e^{-\frac{2}{h} \Im \Theta} e^{-2\lambda \Re \Theta} h^{-\frac{n-2}{2}} e^{\overline{\Phi_{1,h}}(\tT', p', r)+\Phi_{2,h}(\tT', p', r)+\overline{f_1}(r)+f_2(r)}
\\
&\quad \quad \quad \quad \quad \quad \quad \quad \quad\quad  \times\chi^2(y/\delta')\eta(\tT, p, r)\psi(r, y)|g_0|^{1/2}dydr
\\
= &\sqrt{1-\beta^2} \int_{0}^{\frac{L}{\sqrt{1-\beta^2}}}\int_{\R^{n-2}} e^{-\frac{1}{h} \sqrt{1-\beta^2}\Im H(\sqrt{1-\beta^2}r)y\cdot y} e^{-2\lambda (1-\beta^2)r} e^{\lambda \mathcal{O}(|y|^2)}h^{-\frac{n-2}{2}}\chi^2(y/\delta')\eta(\tT', p', r)
\\
& \quad \quad \quad \quad \quad \quad \quad \quad \quad \times e^{\overline{\Phi_{1,h}}(\tT', p', r)+\Phi_{2,h}(\tT', p', r)+\overline{f_1}(r)+f_2(r)}\psi(r,y)(\sqrt{1-\beta^2}+\mathcal{O}(|y|^2))dydr
\\
=&(1-\beta^2)\int_{0}^{\frac{L}{\sqrt{1-\beta^2}}}\int_{\R^{n-2}}e^{-\sqrt{1-\beta^2}\Im H(\sqrt{1-\beta^2}r)y\cdot y} e^{-2(1-\beta^2)\lambda r} e^{h\lambda \mathcal{O}(|y|^2)}\chi^2(h^{1/2}y/\delta')
\\
&\quad \quad \quad \quad \quad \quad \quad \quad \times e^{\overline{\Phi_{1,h}}(\tT', p', r)+\Phi_{2,h}(\tT', p', r)+\overline{f_1}(r)+f_2(r)}\psi(r,h^{\frac{1}{2}}y)(1+h\mathcal{O}(|y|^2))\eta(\tT', p', r)dydr,
\end{aligned}
\end{equation}
where we have performed a change of variables $y\mapsto h^{\frac{1}{2}}y$ in the last step.

Passing to the limit $h\to 0$ in \eqref{eq:prod_computation}, we get the following pointwise limits,
\[
e^{h\lambda\mathcal{O}(|y|)^2}\to 1, \quad \chi^2(h^{1/2}y/\delta')\to 1, \quad \psi(r, h^{\frac{1}{2}}y)\to \psi(r, 0)=\psi(\gamma(r)), \quad \Phi_{i,h} \to \Phi_i,
\]
where we used \eqref{eq:bound_diff_phi} in the verification of the last limit. We recall that
\[
\Im \Theta= \frac{1}{2}\sqrt{1-\beta^2} \Im H(\sqrt{1-\beta^2}r)y\cdot y\ge d|y|^2 \quad \text{ and } \quad \int_{\R^{n-2}}e^{-d|y|^2}dy<\infty.
\]
Hence, by the dominated convergence theorem, we get
\begin{equation}
\label{eq:lim_prod}
\begin{aligned}
&\lim_{h\to 0} \int_{M_0} \overline{v_s}(\tT', p', \cdot)w_s(\tT', p', \cdot) \psi dV_{g_0} 
\\
&= (1-\beta^2)\int_{0}^{\frac{L}{\sqrt{1-\beta^2}}}  e^{\overline{f_1}(r)+f_2(r)}\bigg(\int_{\R^{n-2}} e^{-\sqrt{1-\beta^2}\Im H(\sqrt{1-\beta^2}r)y\cdot y}dy\bigg)
\\
&\quad \quad \quad \quad \quad \quad \times \eta(\tT', p', r)e^{-2(1-\beta^2)\lambda r}e^{\overline{\Phi_1}(\tT', p', r)+\Phi_2(\tT', p', r)}\psi(\gamma(r)) dr.
\end{aligned}
\end{equation} 

To simplify the expression on the right-hand side of \eqref{eq:lim_prod}, we perform a change of variable $y\mapsto (1-\beta^2)^{-\frac{1}{4}}y$ to obtain
\begin{equation}
\label{eq:ImH}
\begin{aligned}
\int_{\R^{n-2}} e^{-\sqrt{1-\beta^2}\Im H(\sqrt{1-\beta^2}r)y\cdot y}dy&= \int_{\R^{n-2}} (1-\beta^2)^{-\frac{n-2}{4}}e^{-\Im H( \sqrt{1-\beta^2}r)y\cdot y} dy
\\
&= \frac{\pi^{\frac{n-2}{2}}(1-\beta^2)^{-\frac{n-2}{4}}}{\sqrt{\det(\Im H( \sqrt{1-\beta^2}r))}}.
\end{aligned}
\end{equation}
We set $r_0=\sqrt{1-\beta^2}\tau_0$ and recall from \cite[Lemma 2.58]{KKL_book} that 
\begin{equation}
\label{eq:det_ImH}
\det(\Im H( \sqrt{1-\beta^2}r)) = \det(\Im H( \sqrt{1-\beta^2}r_0)) e^{-2 \int_{r_0}^{r} \sqrt{1-\beta^2}\mathrm{tr} \Re H( \sqrt{1-\beta^2}w)dw}. 
\end{equation}
This implies 
\begin{equation}
\label{eq:Gaussian_int}
\int_{\R^{n-2}} e^{-\sqrt{1-\beta^2}\Im H(\sqrt{1-\beta^2}r)y\cdot y}dy=\frac{\pi^{\frac{n-2}{2}}(1-\beta^2)^{-\frac{n-2}{4}}e^{\int_{r_0}^{r} \sqrt{1-\beta^2}\mathrm{tr} \Re(H( \sqrt{1-\beta^2}w))dw}}{\sqrt{\det(\Im H(\sqrt{1-\beta^2}r_0))}}.
\end{equation}
Furthermore, since
\begin{equation}
\label{eq:deri_phase}
\p_r f_j(r)= -\frac{1}{2} \sqrt{1-\beta^2}\tr H(\sqrt{1-\beta^2}r), \quad j=1, 2,
\end{equation}
we get
\[
\p_r (\overline{f_1}(r)+f_2(r)) =- \sqrt{1-\beta^2} \tr \Re H(\sqrt{1-\beta^2}r).
\]
Thus, by the fundamental theorem of calculus, we have
\begin{equation}
\label{eq:f1_plus_f2}
\overline{f_1}(r)+f_2(r)=\overline{f_1}(r_0)+f_2(r_0)-\int_{r_0}^{r} \sqrt{1-\beta^2} \mathrm{tr} \Re(H(\sqrt{1-\beta^2}w))dw.
\end{equation}
We next choose $f_1(r_0)$ and $f_2(r_0)$ so that
\begin{equation}
\label{eq:choice_r0}
\frac{e^{\overline{f_1}(r_0)+f_2(r_0)}\pi^{\frac{n-2}{2}}}{\sqrt{\det(\Im H(\sqrt{1-\beta^2}r_0))}}=1.
\end{equation}
Thus, it follows from \eqref{eq:Gaussian_int}, \eqref{eq:f1_plus_f2}, and \eqref{eq:choice_r0} that
\begin{equation}
\label{eq:prod_exp_int}
\begin{aligned}
&e^{\overline{f_1}(r)+f_2(r)}\bigg(\int_{\R^{n-2}} e^{-\Im H(\sqrt{1-\beta^2}r)x\cdot x}dx\bigg)=(1-\beta^2)^{-\frac{n-2}{4}}.
\end{aligned}
\end{equation} 
Finally, we obtain \eqref{eq:limit_prod_vw} for $\psi\in C_0(W_k\cap M_0)$ by substituting \eqref{eq:prod_exp_int} into \eqref{eq:lim_prod}. 

\textit{Case 2: $\psi\in C_0(V_j\cap M_0)$}. Let us now verify \eqref{eq:limit_prod_vw} when $\psi\in C_0(V_j\cap M_0)$ for some $j \in \{1,\ldots,R'\}$, where $R'$ is the number of self-intersection points of $\gamma$. In this case, we obtain from \eqref{eq:finite_sum_v} that the quasimodes are of the form
\[
v_s=\sum_{l:\gamma(r_l)=z_j} v_s^{(l)}, \quad w_s=\sum_{l:\gamma(r_l)=z_j} w_s^{(l)}
\]
on $\supp(\psi)$, thus 
\begin{equation}
\label{eq:product_vsws}
\overline{v_s}w_s= \sum_{l:\gamma(r_l)=z_j} \overline{v_s^{(l)}}w_s^{(l)}+\sum_{l\ne l':\gamma(r_l)=\gamma(r_{l'})=z_j} \overline{v_s^{(l)}}w_s^{(l')}.
\end{equation}

We aim to prove that the contribution of the cross terms vanish in the limit $h\to 0$. More precisely, 
\begin{equation}
\label{eq:cross_term_vanish}
\lim_{h\to 0} \int_{M_0} \overline{v_s^{(l)}}(\tT', p', \cdot)w_s^{(l')}(\tT', p', \cdot)\psi dV_{g_0}=0, \quad l\ne l'.
\end{equation}
If so, then the limit \eqref{eq:limit_prod_vw} follows from the first part of this proof. 

In order to verify  \eqref{eq:cross_term_vanish}, we split $\psi$ into a smooth and a sufficiently small part. In the latter case, the limit \eqref{eq:cross_term_vanish} follows from Proposition \ref{prop:regularization}. For the smooth part, we integrate by parts, which requires us to differentiate $\psi$. In addition, we need the estimate \eqref{eq:est_v_bdyM} to show that the boundary term vanishes. This estimate was not needed to prove Theorem \ref{prop:Gaussian_beam}.

To start the proof, following similar arguments as in the proof of \cite[Proposition 3.1]{Ferreira_Kur_Las_Salo}, we write
\[
v_s^{(l)}=e^{\frac{i}{h} \Re \Theta^{(l)}}\phi^{(l)}, \quad \phi^{(l)}=e^{-\lambda \Re \Theta^{(l)}}e^{-s\Im \Theta^{(l)}}b^{(l)},
\]
and
\[
w_s^{(l')}=e^{\frac{i}{h} \Re \Theta^{(l')}}\omega^{(l')}, \quad \omega^{(l')}=e^{-\lambda \Re \Theta^{(l')}}e^{-s\Im \Theta^{(l')}}B^{(l')},
\]
which imply that
\begin{equation}
\label{eq:prod_vw_rewrite}
\overline{v_s^{(l)}}w_s^{(l')}= e^{\frac{i}{h}\sigma} \overline{\phi^{(l)}}\omega^{(l')}, \quad \sigma= \Re \Theta^{(l')} - \Re \Theta^{(l)}.
\end{equation}
Hence, in view of \eqref{eq:prod_vw_rewrite}, we need to show that
\begin{equation}
\label{eq:prod_vanish_rewrite}
\lim_{h\to 0}\int_{M_0} e^{\frac{i}{h}\sigma}\overline{\phi^{(l)}}(\tT', p', \cdot)\omega^{(l')}(\tT', p', \cdot)\psi dV_{g_0}=0, \quad l\ne l'.
\end{equation}

To prove \eqref{eq:prod_vanish_rewrite}, let us write $\psi=\psi_h+(\psi-\psi_h)$, where the regularization $\psi_h \in C^\infty_0(V_j\cap M_0)$, but its support can meet $\p M_0$, and $\psi-\psi_h$ is continuous. We recall that, due to the estimates in \eqref{eq:estimate_v} and \eqref{eq:estimate_w}, we have $\|\phi^{(l)}\|_{L^2(V_j\cap M_0)},\|\omega^{(l')}\|_{L^2(V_j\cap M_0)}=\cO(1)$. Thus, by H\"older's inequality and estimate \eqref{eq:est_diff_regu}, we obtain
\begin{equation}
\label{eq:est_psi2}
\bigg|\int_{M_0}e^{\frac{i}{h}\sigma}\overline{\phi^{(l)}}(\tT', p', \cdot)\omega^{(l')}(\tT', p', \cdot)(\psi-\psi_h) dV_{g_0}\bigg|\le \|v_s^{(l)}\|_{L^2} \|w_s^{(l')}\|_{L^2}\|\psi-\psi_h\|_{L^\infty} =o(1), \quad h \to 0,
\end{equation}
where the $L^p$-norms are taken over the set $V_j\cap M_0$. 

To analyze the term involving the smooth part $\psi_h$, we note that by \eqref{eq:g0_near_geo} and \eqref{eq:phi}, the gradients of $\Re\Theta^{(l)}$ and $\Re\Theta^{(l')}$ at $z_j$ are parallel to $\dot{\gamma}(r_l)$ and $\dot \gamma(r_{l'})$, respectively. Since the geodesic $\gamma$ intersects itself transversally at $z_j$, we have $\nabla_{g_0}\sigma(z_j)\ne 0$. Therefore, by shrinking the set $V_j$ if necessary, we may assume that $\sigma$ has no critical points in $V_j$.
Thus, the vector field $L=\frac{\overline{\phi^{(l)}}\omega^{(l')}\psi_h}{|\nabla_{g_0}\sigma|^2}\nabla_{g_0}\sigma$ is well-defined and satisfies 
$
e^{\frac{i}{h}\sigma}\overline{\phi^{(l)}}\omega^{(l')}\psi_h
=
-ihL(e^{\frac{i}{h}\sigma}).
$
Therefore, 
we integrate by parts in \eqref{eq:prod_vanish_rewrite} to obtain
\begin{equation}
\label{eq:prod_psi1}
\begin{aligned}
\int_{M_0}e^{\frac{i}{h}\sigma}\overline{\phi^{(l)}}(\tT', p', \cdot)\omega^{(l')}(\tT', p', \cdot)\psi_hdV_{g_0}
= &
-ih\int_{V_j\cap \p M_0} \frac{\p_\nu \sigma}{ |\nabla_{g_0}\sigma|^2} e^{\frac{i}{h}\sigma} \overline{\phi^{(l)}}(\tT', p', \cdot)\omega^{(l')}(\tT', p', \cdot)\psi_hdS_{g_0}
\\
&+ih \int_{M_0} e^{\frac{i}{h}\sigma} 
\div(L)(\tT', p', \cdot) dV_{g_0}.
\end{aligned}
\end{equation}

To show that the boundary term on the right-hand side of \eqref{eq:prod_psi1} vanishes as $h\to 0$, we use estimate \eqref{eq:est_v_bdyM} to observe that $\|\phi^{(l)}\|_{L^2(\p M_0)}$, $\|\omega^{(l')}\|_{L^2(\p M_0)}=\cO(1)$. Furthermore, $\sigma$ is real-valued and independent of $h$. Then 
by estimate \eqref{eq:est_diff_regu} and H\"older's inequality, we have
\begin{align*}
-ih\int_{V_j\cap \p M_0} &\frac{\p_\nu \sigma}{|\nabla_{g_0}\sigma|^2} e^{\frac{i}{h}\sigma} \overline{\phi^{(l)}}(\tT', p', \cdot)\omega^{(l')}(\tT', p', \cdot)\psi_hdS_{g_0}
=
\cO(h), \quad h \to 0.
\end{align*}

To prove that the second term on the right-hand side of \eqref{eq:prod_psi1} vanishes as $h\to 0$, we first compute that
\[
\div L
=\div \left(\overline{\phi^{(l)}}\omega^{(l')}\psi_h \frac{\nabla_{g_0}\sigma}{|\nabla_{g_0}\sigma|^2}\right)
=\left\langle \nabla_{g_0}(\overline{\phi^{(l)}}\omega^{(l')}\psi_h), \frac{\nabla_{g_0}\sigma}{|\nabla_{g_0}\sigma|^2}\right\rangle_{g_0}
+
\overline{\phi^{(l)}}\omega^{(l')}\psi_h \div \left(\frac{\nabla_{g_0}\sigma}{|\nabla_{g_0}\sigma|^2}\right).
\]
Using estimates \eqref{eq:est_diff_regu}, $\|\phi^{(l)}\|_{L^2(M_0)},\|\omega^{(l')}\|_{L^2(M_0)}=\cO(1)$, and H\"older's inequality, we get
\begin{align*}
&ih \int_{M_0} e^{\frac{i}{h}\sigma} \overline{\phi^{(l)}}(\tT', p', \cdot)\omega^{(l')}(\tT', p', \cdot)\psi_h \div \left(\frac{\nabla_{g_0}\sigma}{|\nabla_{g_0}\sigma|^2}\right) dV_{g_0}
=\cO(h), \quad h \to 0.
\end{align*}

Next, we write
\begin{align*}
\overline{\phi^{(l)}}\omega^{(l')}\psi_h
&=[e^{-\lambda(\Re\Theta^{(l)}+\Re \Theta^{(l')})}e^{-i\lambda(\Im\Theta^{(l')}-\Im \Theta^{(l)})}][e^{-\frac{1}{h}(\Im\Theta^{(l)}+\Im \Theta^{(l')})}h^{-\frac{n-2}{2}}][b_0^{(l)}\overline{B_0^{(l')}}\chi^2(y/\delta')]\psi_h
\\
&=f_1f_2f_3\psi_h.
\end{align*}
To streamline the proof, we only provide the estimate for the worst case scenario, which occurs when $\nabla_{g_0}$ acts on $f_2=e^{-\frac{1}{h}(\Im\Theta^{(l)}+\Im \Theta^{(l')})}h^{-\frac{n-2}{2}}$. To that end, due to \eqref{eq:phi}, there exists $C>0$ such that by the Cauchy-Schwarz inequality, as well as estimates \eqref{eq:est_diff_regu} and \eqref{eq:L2_power_y} with $k=\frac{1}{2}$, we have
\begin{align*}
&h\int_{M_0} 
|e^{\frac{i}{h}\sigma} 
f_1f_2\psi_h \langle \nabla_{g_0}f_2, \frac{\nabla_{g_0}\sigma}{|\nabla_{g_0}\sigma|^2}\rangle|
dV_{g_0}
\\
&\leq
C\|\psi_h\|_{L^\infty}\|\phi^{(l)}(\tT', p', \cdot)\|_{L^2}\|\omega^{(l')}(\tT', p', \cdot)\|_{L^2}\int_{M_0} 
h^{-\frac{n-2}{2}}|y|e^{-\frac{d}{h} |y|^2} 
dV_{g_0}
=
\cO(h^{1/2}), \quad  h \to 0.
\end{align*}
Hence, we conclude that the second term on the right-hand side of \eqref{eq:prod_psi1} is of order $\cO(h^{1/2})$, thus the limit \eqref{eq:prod_vanish_rewrite} is verified. This completes the proof of Proposition \ref{prop:limit_behavior}. 
\end{proof}

\subsection{Construction of CGO solutions}
\label{subsec:CGO}
We now proceed to construct CGO solutions of the forms \eqref{eq:form_CGO} and \eqref{eq:CGO_exp_grow}. Thanks to the interior Carleman estimate established in Proposition \ref{prop:solvability}, we can put the ingredients together in a simple way. Let $(M, g)$ be a CTA manifold given by Definition \ref{def:CTA_manifolds}. We already computed in Section \ref{sec:Car_est} that 
\[
c^{\frac{n+2}{4}}\circ \mathcal{L}_{c,g,a,q} \circ \conf = \mathcal{L}_{\tilde{g}, \tilde{a}, \tilde{q}},
\]
where $\tilde g=e\oplus g_0$, $\tilde{a}=ca$, and $\tilde{q}=c(q-c^{\frac{n-2}{4}}\Delta_g(\conf))$. This implies that $u=\conf \tilde u$ satisfies $\mathcal{L}_{c,g,a,q} u=0$ if $\tilde{u}$ solves $ \mathcal{L}_{\tilde{g}, \tilde{a}, \tilde{q}} \tilde u=0$ in $Q$.

Let us write $(t,x)=(t,x_1, x')$ for local coordinates in $Q$ and recall that
$s=\frac{1}{h}+i\lambda$, $0<h\ll 1$, where $\lambda\in \R$ is fixed. We are interested in finding CGO solutions to the equation
\begin{equation}
\label{eq:op_no_conformal}
\mathcal{L}_{\tilde{g}, \tilde{a}, \tilde{q}}\tilde{u}=0 \quad \text{in} \quad Q
\end{equation}
of the form
\[
\tilde{u}=e^{-s(\beta t+x_1)}(v_s+r),
\]
where $v_s$ is the Gaussian beam quasimode given in Proposition \ref{prop:Gaussian_beam}, and $r=r_s$ is a correction term that vanishes in the limit $h\to 0$. Indeed, $\tilde{u}$ is a solution to \eqref{eq:op_no_conformal} if 
\begin{equation}
\label{eq:op_no_conf_conjugate}
e^{s(\beta t+x_1)}h^2\mathcal{L}_{\tilde{g}, \tilde{a}, \tilde{q}}e^{-s(\beta t+x_1)}r=-e^{s(\beta t+x_1)}h^2\mathcal{L}_{\tilde{g}, \tilde{a}, \tilde{q}}e^{-s(\beta t+x_1)}v_s.
\end{equation}
Then we apply Proposition \ref{prop:solvability} with $v=-e^{s(\beta t+x_1)}h^2\mathcal{L}_{\tilde{g}, \tilde{a}, \tilde{q}}e^{-s(\beta t+x_1)}v_s$ and estimate \eqref{eq:estimate_v} to conclude that there exists $r\in H^1(Q^\mathrm{int})$ such that \eqref{eq:op_no_conf_conjugate} holds and $	\|r\|_{H^1_{\scl}(Q)}= o(1)$ as $h\to 0$. 

We summarize our discussion above in the following theorem. In particular, we have general conformal factor $c$ in this result instead of $c=1$, which we assumed in all of the earlier results in this paper. 
\begin{thm}
\label{prop:CGO_solution}
Let $a \in W^{1,\infty}(Q)$ and $q \in C(\overline{Q})$. Let  $s=\frac{1}{h}+i\lambda$ with $\lambda \in \R$ fixed. For all $h>0$ small enough, there exists a solution $u_1\in H^1(Q)$ to $\mathcal{L}_{c,g, a, q}^*u_1=0$ of the form
\begin{equation}
\label{eq:CGO_v_old}
u_1=e^{-s(\beta t+x_1)} \conf (v_s+r_1),
\end{equation}
where $v_s\in C^\infty(Q)$ is the Gaussian beam quasimode given in Proposition \ref{prop:Gaussian_beam}, and $r_1\in H^1_{\scl}(Q^\mathrm{int})$ is such that $\|r_1\|_{H^1_\mathrm{\scl}(Q^\mathrm{int})}= o(1)$ as $h\to 0$. 

There also exists a solution $u_2\in H^1(Q)$ to $\mathcal{L}_{c,g, a, q}u_2=0$ that has the form
\begin{equation}
\label{eq:CGO_u2_old}
u_2=e^{s(\beta t+x_1)}\conf (w_s+r_2),
\end{equation}
where $w_s\in C^\infty(Q)$ is the Gaussian beam quasimode given in Proposition \ref{prop:Gaussian_beam}, and $r_2\in H^1_{\scl}(Q^\mathrm{int})$ is such that $\|r_2\|_{H^1_\mathrm{scl}(Q^\mathrm{int})}= o(1)$ as $h\to 0$. 
\end{thm}

\section{Proof of Theorem \ref{thm:main_result_damping}}
\label{sec:proof_of_theorem}

Let $u_1\in H^1(Q)$ be an exponentially decaying CGO solution given by \eqref{eq:CGO_v_old} to the equation
$\mathcal{L}_{c,g,a_1,q_1}^* u_1 = 0$ in $Q$, and let $u_2\in H^1(Q)$ be an exponentially growing CGO solution given by \eqref{eq:CGO_u2_old} satisfying $\mathcal{L}_{c,g,a_2,q_2} u_2 = 0$ in $Q$. Due to the main assumption $\mathcal{C}_{g,a_1,q_1} = \mathcal{C}_{g,a_2,q_2}$, it follows from Proposition \ref{prop:well_posedness}
that there exists a function $v\in H_{\Box_{c,g}}(Q)$ such that $\mathcal{L}_{c,g,a_1,q_1} v = 0$ and 
\[
(u_2 - v)|_\Sigma=(u_2 - v)|_{t=0}= (u_2 - v)|_{t=T}= \p_t(u_2 - v)|_{t=0}=\p_\nu (u_2 - v)|_V = 0.
\]
Then $u : = u_2 - v \in H_{\Box_{c,g}}(Q)$ solves the equation
\begin{equation}
\label{eq:ibvp_difference}
\mathcal{L}_{c,g, a_1, q_1}u=a\p_tu_2+qu_2 \quad \text{in} \quad Q, 
\quad 
u|_\Sigma=u|_{t=0}= u|_{t=T}= \p_tu|_{t=0}=\p_\nu u|_V = 0.
\end{equation}
Here and in what follows we denote $a:=a_1-a_2$ and $q:=q_1-q_2$. Since 
$a\p_tu_2+qu_2\in L^2(Q)$, it follows from \cite[Theorem 2.1]{Lasiecka_Lions_Triggiani} that 
\[
u\in C^1([0,T];L^2(M))\cap C([0,T];H^1_0(M))\subset H^1(Q) \text{ with } \p_\nu u\in L^2(\Sigma).
\]

Since $u_1\in H^1(Q)$ and $\Box_{c,g}u_1\in L^2(Q)$, we see that 
\[
(c^{-1}\p_t u_1, -\nabla_g u_1)\in H_{\div}(Q):=\{F\in L^2(Q, TQ): \div_{(t,x)}F\in L^2(Q)\}.
\]
We view $\overline{Q}$ as a compact Riemannian manifold of dimension $n+1$  with the metric $\overline{g} =dt^2\oplus g$,
and let $\bar\nu$ be the outward unit normal vector to $\p Q$. 
In view of \cite[Lemma 2.2]{kavian2003lectures}, for $F\in H_{\div}(Q)$, $F\cdot \bar \nu|_{\p Q}$ can be defined as an element of $H^{-1/2}(\p Q)$, and for $\psi \in H^1(Q)$ we have
\[
\langle \psi, {F}\cdot \bar \nu \rangle_{H^{1/2}(\p Q),H^{-1/2}(\p Q)}  =  \langle \psi, \div_{\overline{g}}({F})\rangle_{L^2(Q)} + \langle \nabla_{\overline{g}} \psi,{F} \rangle_{L^2(Q)}.
\]
By taking $\psi = u$ and $F=\overline{(c^{-1}\p_t u_1, -\nabla_g u_1)}$, we deduce that
\begin{equation}
\label{eq:int_comp_1}
\begin{aligned}
&\langle u,  \overline{(c^{-1}\p_t u_1, -\nabla_g u_1)}\cdot \bar{\nu} \rangle_{H^{1/2}(\p Q), H^{-1/2}(\p Q)}
\\
&= \langle u, \Box_{c,g}\overline{u_1}\rangle_{L^2(Q)}
+\langle (\p_tu,\nabla_g u), (c^{-1}\p_t\overline{u_1},-\nabla_g\overline{u_1}) \rangle_{L^2(Q)}.
\\
\end{aligned}
\end{equation}
Arguing similarly and using $\p_tu|_{t=0}, \: \p_tu|_{t=T}\in L^2(M)$, $\p_\nu u\in L^2(\Sigma)$, and $u_1\in H^1(Q)$, we get
\begin{equation}
\label{eq:int_comp_2}
\begin{aligned}
&\langle (c^{-1}\p_tu,-\nabla_gu)\cdot \bar\nu , \overline{u_1}\rangle_{L^2(\p Q)}
=\langle \Box_{c,g}u,\overline{u_1}\rangle_{L^2(Q)} + \langle (c^{-1}\p_tu, -\nabla_g u),(\p_t\overline{u_1}, \nabla_g\overline{u_1}) \rangle_{L^2(Q)}.
\end{aligned}
\end{equation}
Since  $a \in W^{1,\infty}(Q)$ and $u, \overline{u_1}\in H^1(Q)$, it follows from the proof \cite[Proposition 9.4]{brezis2011functional} that $a\overline{u_1} \in H^1(Q)$ and $\p_t(a\overline{u_1})=\p_ta \overline{u_1}+a\p_t\overline{u_1}$ in the weak sense. Therefore, the following integration by parts is justified: 
\begin{equation}
\label{eq:int_comp_3}
\int_Q (a_1\overline{u_1})\p_t udV_gdt 
= -\int_Q  \p_t(a_1\overline{u_1})udV_gdt
=-\int_Q ( \p_ta_1\overline{u_1}+ a_1\p_t\overline{u_1})udV_gdt.
\end{equation}
In particular, there are no boundary terms since $u|_{t=0}=u|_{t=T}=0$.


We then multiply \eqref{eq:ibvp_difference} by $\overline{u_1}$ and integrate over $Q$. Therefore, we deduce from \eqref{eq:ibvp_difference}--\eqref{eq:int_comp_3} that
\begin{equation}
\label{eq:int_comp_comb}
\begin{aligned}
&\int_Q (a\p_tu_2+qu_2)\overline{u_1}dV_gdt
\\
&=\langle\mathcal{L}_{c,g, a_1,q_1}u,\overline{u_1}\rangle_{L^2(Q)} - \langle u,\overline{\mathcal{L}_{c,g,a_1,q_1}^*u_1}\rangle_{L^2(Q)}
\\
&=\langle (c^{-1}\p_tu,-\nabla_gu)\cdot \bar\nu , \overline{u_1}\rangle_{L^2(\p Q)} - \langle u,(c^{-1}\p_t\overline{u_1},-\nabla_g\overline{u_1})\cdot \bar\nu \rangle_{H^{1/2}(\p Q), H^{-1/2}(\p Q)}.
\end{aligned}
\end{equation}
Since $u|_\Sigma=u|_{t=0}= u|_{t=T}= 0$, the second term on the right-hand side of \eqref{eq:int_comp_comb} vanishes. Furthermore, since $\p_\nu u\in L^2(\Sigma)$ with  $\p_tu|_{t=0}=\p_\nu u|_V = 0$, we obtain the integral identity
\begin{equation}
\label{eq:int_id_2}
\begin{aligned}
\int_Q (a\p_tu_2+qu_2)\overline{u_1}dV_gdt=-\int_{\Sigma\setminus V}\p_\nu u\overline{u_1}dS_g dt 
+
\int_M c^{-1}\p_t u(T, x)\overline{u_1(T, x)}dV_g.
\end{aligned}
\end{equation}



We shall next substitute the CGO solutions \eqref{eq:CGO_v_old} and \eqref{eq:CGO_u2_old} into \eqref{eq:int_id_2}, multiply the equation by $h$, and pass to the limit $h\to 0$. In order to analyze the limit of the terms on the left-hand side of \eqref{eq:int_id_2}, we use estimates \eqref{eq:solvability_bound} and \eqref{eq:estimate_v} to obtain the following estimates for the remainder terms:
\begin{equation}
\label{eq:est_remainder}
\|r_j\|_{L^2(Q)}\le \|r_j\|_{H^1_{\scl}(Q^\mathrm{int})}=o(1), \quad j = 1,2,
\end{equation}
and
\begin{equation}
\label{eq:est_dremainder}
\|\p_tr_j\|_{L^2(Q)} \le \frac{1}{h} \|r_j\|_{H^1_{\scl}(Q^\mathrm{int})}=o(h^{-1}), \quad j= 1,2.
\end{equation}
On the other hand, the following lemma explains the behaviors of the two terms on the right-hand side of \eqref{eq:int_id_2} as $h\to 0$.
\begin{lem}
\label{lem:rhs_est}
Let $u_1$ and $u$ be the functions described above. 
Then the following estimates hold as $h\to 0$:
\begin{equation}
\label{eq:est_int_id_rhs_1}
\int_M c^{-1}\p_t u(T, x)\overline{u_1(T, x)}dV_g =
\left\{
\begin{array}{cc}
\mathcal{O}(h^{-1/2}) & \text{if } a\neq 0 ,
\\
\\
\mathcal{O}(h^{1/2}) & \text{ if } a =0.
\end{array}
\right. 
\end{equation}
\begin{equation}
\label{eq:est_int_id_rhs_2}
\int_{\Sigma\setminus V}\p_\nu u\overline{u_1}dS_g dt = \left\{
\begin{array}{cc}
o(h^{-1}) & \text{if } a\neq 0 ,
\\
\\
o(1) & \text{ if } a =0.
\end{array}
\right.
\end{equation}
\end{lem}
We will postpone the proof of this result and use it to prove the uniqueness of the damping coefficient first.

\subsection{Uniqueness of the damping coefficient}
From the respective CGO solutions \eqref{eq:CGO_v_old} and \eqref{eq:CGO_u2_old} for $u_1$ and $u_2$, we compute that
\[
u_2\overline{u_1}=e^{2i\lambda (\beta t+x_1)}\confm(\overline{v_s}w_s+\overline{v_s}r_2+\overline{r_1}w_s+\overline{r_1}r_2).
\]
Therefore, by estimates \eqref{eq:estimate_v}, \eqref{eq:estimate_w}, \eqref{eq:est_remainder}, and the Cauchy-Schwartz inequality, we have
\begin{align*}
\|u_2\overline{u_1}\|_{L^1(Q)}\le \mathcal{O}(1).
\end{align*}
Hence, the following estimate holds:
\begin{equation}
	\label{eq:est_hqu2u1}
h\bigg|\int_Q qu_2\overline{u_1}dV_gdt\bigg|\le h\|q\|_{L^\infty(Q)}\|u_2\overline{u_1}\|_{L^1(Q)} =\mathcal{O}(h), \quad h\to 0.
\end{equation}

We next consider the term $h \int_Q a(\p_tu_2)\overline{u_1}dV_gdt$ on the left-hand side of \eqref{eq:int_id_2}. To that end, direct computations yield
\begin{align*}
(\p_tu_2)\overline{u_1}= &e^{2i\lambda (\beta t+x_1)}\confm[(\overline{v_s}\p_t w_s+\overline{v_s}\p_t r_2+\overline{r_1}\p_tw_s+\overline{r_1}\p_t r_2)
\\
&+s\beta(\overline{v_s}w_s+\overline{v_s}r_2+\overline{r_1}w_s+\overline{r_1}r_2)].
\end{align*}
Using estimates \eqref{eq:estimate_v}, \eqref{eq:estimate_w}, \eqref{eq:est_remainder}, \eqref{eq:est_dremainder}, as well as the Cauchy-Schwartz inequality, we obtain
\[
h\bigg|\int_Qe^{2i\lambda(\beta t+x_1)}\confm a(\overline{v_s}\p_t w_s+\overline{v_s}\p_t r_2 + \overline{r_1}\p_tw_s +\overline{r_1}\p_t r_2)dV_gdt\bigg|=o(1), \quad h\to 0,
\]
and
\[
h\bigg|\int_Q e^{2i\lambda (\beta t+x_1)} \confm as \beta (\overline{v_s}r_2 + \overline{r_1}w_s+\overline{r_1}r_2)dV_gdt\bigg|=o(1), \quad h\to 0.
\]
Therefore, we have
\begin{equation}
\label{eq:quasimode_zero}
h\int_Q a\p_tu_2\overline{u_1}dV_gdt\to  \int_Q e^{2i\lambda (\beta t+x_1)} \confm \beta a\overline{v_s}w_sdV_gdt,\quad h\to 0.
\end{equation}
Using \eqref{eq:est_hqu2u1}, \eqref{eq:quasimode_zero}, and Lemma \ref{lem:rhs_est}, we deduce from \eqref{eq:int_id_2} that 
\[
\int_Q e^{2i\lambda (\beta t+x_1)} \confm \beta a\overline{v_s}w_sdV_gdt\to 0,\quad h\to 0.
\]

On the other hand, since $a_1,a_2 \in W^{1,\infty}(Q)$ and $a_1=a_2$ on the boundary $\p Q$, we can continuously extend $a$ on $(\R^2\times M_0)\setminus Q$ by $0$ and denote the extension by the same letter. Using $dV_g = c^{\frac{n}{2}}dV_{g_0}dx_1$,  the change of coordinates \eqref{eq:change_variable}, Fubini's theorem, the dominated convergence theorem, and the concentration property \eqref{eq:limit_prod_vw} of the quasimodes $v_s$ and $w_s$, we obtain
\begin{align*}
&\int_Q e^{2i\lambda (\beta t+x_1)} \confm \beta a\overline{v_s}w_sdV_gdt
\\
= & \int_\R \int_\R \int_{M_0} e^{2i\lambda (\beta t+x_1)} c \beta a\overline{v_s}w_s dV_{g_0}dx_1dt
\\
= & \int_\R \int_\R \int_{M_0} \beta^2e^{2i\lambda ((\beta^2-1) \tT+p)} (c a)\overline{v_s}w_s dV_{g_0}dp d\tT
\\
\to &\beta^2 (1-\beta^2)^{-\frac{n-6}{4}}  \int_\R \int_\R \int_{0}^{\frac{L}{\sqrt{1-\beta^2}}} e^{2i\lambda ((\beta^2-1) \tT+p)-2(1-\beta^2)\lambda r} (ca)(\beta \tT, p-\tT, \gamma(\sqrt{1-\beta^2}r))
\\
&\quad \quad \quad \quad \quad \quad \quad \quad \quad \quad \quad\quad \quad \times e^{\overline{\Phi_1(\tT, p, r)}+\Phi_2(\tT, p, r)}\eta(\tT,p, r)drdp d\tT, \quad h\to 0.
\end{align*}
As $\beta \in [\frac{1}{\sqrt{3}}, 1)$, we conclude that
\begin{equation}
\label{eq:cgo_sub}
\begin{aligned}
\int_{0}^{\frac{L}{\sqrt{1-\beta^2}}}\int_\R \int_\R   &e^{2i\lambda ((\beta^2-1) \tT+p)-2(1-\beta^2)\lambda r} (ca)(\beta \tT, p-\tT, \gamma(\sqrt{1-\beta^2}r))
\\
&\times e^{\overline{\Phi_1(\tT, p, r)}+\Phi_2(\tT, p, r)} \eta(\tT,p, r)dp d\tT dr=0.
\end{aligned}
\end{equation}

We next follow the arguments in \cite[Section 4]{Krupchyk_Uhlmann_magschr} closely to prove that \eqref{eq:cgo_sub} holds when
\\
$e^{\overline{\Phi_1(\tT, p, r)}+\Phi_2(\tT, p, r)}\eta(\tT,p, r)$ is removed from the integral. To this end, let us write $\eta(\tilde{t},p, r) := \eta_1(\Tilde{t},r)\eta_2(p)$ in \eqref{eq:quasimode_supp_psi} with $\p \eta_1 = 0$, where $\p = \frac{1}{2}(\p_{\tT}+i\p_r)$, and denote
\begin{equation}
\label{eq:Psi_def}
\Psi(\tT, p, r)=  e^{2i\lambda[(\beta^2-1)\tT+p-i(\beta^2-1)r]}\eta_1(\tT, r).
\end{equation}
It follows from direct computations that $\p\Psi = \frac{1}{2}(\p_{\tilde t} - i\p_r)\Psi=0$.

Since $a$ is supported in $Q$, we get from \eqref{eq:cgo_sub} that 
\[
\int_\R \int_\R \int_\R \Psi(\tT, p, r) (ca)(\beta \tT, p-\tT, \gamma(\sqrt{1-\beta^2}r)) e^{\overline{\Phi_1(\tT, p, r)}+\Phi_2(\tT, p, r)}\eta_2(p)dp d\tT dr=0,
\]
and there exists a constant $R>0$ such that $\supp a\subset \subset B_{\tT,p, r}(0, R)$.
Also, since $\eta_2\in C^\infty(\R)$ is arbitrary, for almost every $p\in \R$ we have
\begin{equation}
\label{eq:int_aap}
\begin{aligned}
&\int_{\Omega_p} \Psi(\tT, p, r) (ca)(\beta \tT, p-\tT, \gamma(\sqrt{1-\beta^2}r)) e^{\overline{\Phi_1(\tT, p, r)}+\Phi_2(\tT, p, r)}d\tT dr=0,
\end{aligned}
\end{equation}
where $\Omega_p=\{(\tilde{t},r): (\tilde{t}, p, r, y)\in Q\}$. We shall view $\Omega_p$ as a domain in the complex plane with the complex variable $z = \tilde t + ir$.

Recall that it was explained in the formulas \eqref{eq:conformal_equivalence}--\eqref{eq:equiv_coeff} how to transform the hyperbolic operator $\mathcal{L}_{c,g,a,q}$ into another operator $\mathcal{L}_{\tilde g, \tilde a, \tilde q}$ of the same type, where the contribution of the conformal factor $c$ was moved from the highest order term to the lower order ones. In particular, we have $\tilde a=ca$, and the construction of the Gaussian beams in Section \ref{sec:CGO_solution} is carried over for this damping coefficient. Hence,   \eqref{eq:transport_eqns} yields
\begin{equation}
\label{eq:dP1plusP2}
\p(\overline{\Phi_1}+\Phi_2)= \frac{1}{4}\beta (ca)(\beta \tT, p-\tT, \gamma(\sqrt{1-\beta^2}r)).
\end{equation}
Thus, it follows from \eqref{eq:int_aap} and \eqref{eq:dP1plusP2} that 
\begin{equation}
\int_{\Omega_p} \p \big(\Psi(z, p)e^{\overline{\Phi_1(z, p)} +\Phi_2(z, p)}\big)dz \wedge d\overline{z}=0 \quad \text{for almost every }p.
\end{equation}

We now discuss the regularity of $\Phi_i$. To that end, using \eqref{eq:dP1plusP2} along with the fact that $a(\cdot, p)\in L^\infty(\C)$ is compactly supported for almost every $p$, we see that $\p\Phi_i\in L^p(\C)$ for $1\le p\le \infty$. By the boundedness of the Beurling-Ahlfors operator $\overline{\p}\p^{-1}$ on $L^p(\C)$, $1<p<\infty$, we get $\overline{\p}\Phi_i=\overline{\p}\p^{-1}(\p \Phi_i)$, which implies that $\nabla_{g}\Phi_i\in L^p(\C)$. Furthermore, since $\Phi_i\in L^\infty(\C)$, we have $\Phi_i(\cdot, p)\in W^{1,p}_{\mathrm{loc}}(\C)$, $1<p<\infty$. Hence, we conclude that 
$\Phi_i(\cdot, p) \in H^1(\Omega_p)$, $i=1,2$. Thus, an application of Stokes' theorem \cite[Theorem 18A]{Whitney} yields
\begin{equation}
\label{eq:Stoke_1}
\int_{\p \Omega_p} 	\Psi(z, p)e^{\overline{\Phi_1(z, p)} +\Phi_2(z, p)}d\overline{z} =0.
\end{equation}
By \cite[Lemma 5.1]{DDS_Kenig_Sjo_Uhl}, see also \cite[Lemma 3, Section 2.3, Chapter 4]{Ahlfors}, there exists a non-vanishing function $F\in C(\overline{\Omega_p})$, anti-holomorphic in $\Omega_p$, such that
\[
F|_{\p \Omega_p}=e^{\overline{\Phi_1} +\Phi_2}|_{\p \Omega_p}.
\]
Furthermore, the arguments in the proof of \cite[Lemma 7.3]{Knudsen_Salo} show that there exists an anti-holomorphic function $G\in C(\overline{\Omega_p})$ such that $F=e^G$ in $\Omega_p$, and we may assume that $G= \overline{\Phi_1} +\Phi_2$ on $\p \Omega_p$. Choosing $\eta_1 = Ge^{-G}$ in \eqref{eq:Psi_def}, we get from \eqref{eq:Stoke_1} that
\[
\int_{\p \Omega_p} \big(\overline{\Phi_1(\tT, p, r)} +\Phi_2(\tT, p, r)\big)e^{2i\lambda[(\beta^2-1)\tT+p-i(\beta^2-1)r]}d\overline{z}=0.
\]
Applying Stokes' theorem again, using \eqref{eq:dP1plusP2}, and integrating over the $p$ variable, we obtain
\begin{equation}
\label{eq:asymp_behavior}
\int_{0}^{\frac{L}{\sqrt{1-\beta^2}}} \int_\R \int_\R e^{2i\lambda ((\beta^2-1) \tT+p)-2(1-\beta^2)\lambda r} (ca)(\beta \tT, p-\tT, \gamma(\sqrt{1-\beta^2}r))dp d\tT dr=0.
\end{equation}
Finally, we use \eqref{eq:change_variable} to return to $(t,x_1,\tau)$ coordinates from $(\tilde t, p, r)$ coordinates and replace $2\lambda$ with $\lambda$. After these changes, \eqref{eq:asymp_behavior} becomes
\begin{equation}
\label{eq:asy_beh_tau}
\int_{0}^{L} \int_\R \int_\R  e^{i\lambda (\beta t+x_1)-\sqrt{1-\beta^2}\lambda \tau} (ca)(t, x_1, \gamma(\tau)) dx_1dt d\tau=0.
\end{equation}

We are now ready to utilize Assumption \ref{asu:inj}, the invertibility of the attenuated geodesic ray transform on $(M_0,g_0)$. To that end, we let $\mathcal{F}_{(t, x_1)\to (\xi_1, \xi_2)}$ be the Fourier transform in the Euclidean variables $(t,x_1)$ and define
\begin{align*}
f(x', \beta, \lambda) &:= \int_\R\int_\R e^{i\lambda (\beta t+x_1)} (ca)(t, x_1, x')dx_1dt
\\
&=\mathcal{F}_{(t, x_1)\to (\xi_1, \xi_2)}(ca)|_{(\xi_1, \xi_2)=-\lambda(\beta, 1)}
\quad
\text{for } x'\in M_0, \: \beta\in [1/2, 1), \: \lambda \in \R.
\end{align*}
Since $a\in W^{1,\infty}(Q)$, we see that the function $f(\cdot, \beta, \lambda)$ is continuous on $M_0$. Furthermore, as $\gamma$ is an arbitrarily chosen nontangential geodesic in $(M_0,g_0)$, we get from \eqref{eq:asy_beh_tau} that the following attenuated geodesic ray transform vanishes: 
\begin{equation}
\label{eq:Fou_geo_trans}
\int_{0}^{L} e^{-\sqrt{1-\beta^2}\lambda \tau} f(\gamma(\tau), \beta, \lambda)d\tau=0.
\end{equation}

By Assumption \ref{asu:inj}, there exists $\varepsilon>0$ such that $f(\gamma(\tau), \beta, \lambda)=0$ whenever $\sqrt{1-\beta^2}|\lambda|<\varepsilon$. Hence, there exist $\beta_0\in (\frac{1}{\sqrt 3}, 1)$, $\lambda_0>0$, and $\delta>0$ such that for every $(\lambda, \beta)\in \R^2$ that satisfies $|\beta-\beta_0|$, $|\lambda-\lambda_0|<\delta$, and $\lambda\ne 0$, we have $\sqrt{1-\beta^2}|\lambda|<\varepsilon$. In particular, the mapping $(\lambda, \beta)\mapsto -\lambda(\beta, 1)$ is a diffeomorphism when $\lambda\neq 0$, implying that $\mathcal{F}_{(t, x_1)\to (\xi_1, \xi_2)}(ca)=0$ in an open set of $\R^2$. Last, the compact support of $a$ and the Paley-Wiener theorem yield that $\cF(ca)$ is real analytic. Therefore, we get $ca=0$ in $Q$. Since $c$ is a positive function, we must have $a=a_1-a_2=0$.

To complete the proof of uniqueness for the damping coefficient in Theorem \ref{thm:main_result_damping}, we still need to verify Lemma \ref{lem:rhs_est}. 

\begin{proof}[Proof of Lemma \ref{lem:rhs_est}]
Let us prove estimate \eqref{eq:est_int_id_rhs_1} first. To this end, using estimates \eqref{eq:estimate_v} and \eqref{eq:est_remainder}, the CGO solution  \eqref{eq:CGO_v_old}, and the Cauchy-Schwartz inequality, we get
\begin{equation}
\label{eq:simplification}
\begin{aligned}
\bigg|\int_M c^{-1}\p_t u(T, x)\overline{u_1(T, x)}dV_g\bigg| 
&\le \|c^{-1}\|_{L^\infty(Q)} \int_M |\p_t u(T, x)\overline{u_1(T, x)}|dV_g
\\
&\le \mathcal{O}(1) \int_M |\conf e^{-s(\beta T+x_1)} \p_tu(T,x)|(|v_s(T,x)|+|r_1(T,x)|)dV_g
\\
&\le \mathcal{O}(1) \|e^{-s(\beta T+x_1)}\p_tu(T,\cdot)\|_{L^2(M)}.
\end{aligned}
\end{equation}
Utilizing the boundary Carleman estimate \eqref{eq:Car_est_conjugated} and   \eqref{eq:ibvp_difference}, we obtain
\begin{align}
\label{eq:application_of_Car_est}
& 
\|e^{-s(\beta T+x_1)}\p_tu(T,\cdot)\|_{L^2(M)}
\le \mathcal{O}(h^{1/2})\|e^{-s(\beta t +x_1)}\mathcal{L}_{c,g,a_1,q_1}u\|_{L^2(Q)}.
\end{align} 
We then 
substitute the CGO solution \eqref{eq:CGO_u2_old} for $u_2$ into \eqref{eq:ibvp_difference} to get
\begin{align*}
e^{-s(\beta t +x_1)}\mathcal{L}_{c,g,a_1,q_1}u
&= e^{-s(\beta t+x_1)}(a\p_tu_2+qu_2)
\\
&=\confm[as\beta(w_s+r_2)+a(\p_tw_s+\p_t r_2)+q(w_s+r_2)].
\end{align*}

We recall that $s=h^{-1}+i\lambda$ and use estimates \eqref{eq:estimate_w}, \eqref{eq:est_remainder}, and \eqref{eq:est_dremainder} to obtain
\begin{equation}
\label{eq:bifurgation}
\|e^{-s(\beta t +x_1)}\mathcal{L}_{c,g,a_1,q_1}u\|_{L^2(Q)}
=
\left\{
\begin{array}{cc}
\mathcal{O}(h^{-1}), & \text{if } a\neq 0 ,
\\
\\
\mathcal{O}(1), & \text{ if } a =0.
\end{array}
\right. 
\end{equation}
Therefore, estimates \eqref{eq:simplification}--\eqref{eq:bifurgation} imply \eqref{eq:est_int_id_rhs_1}.

We next prove estimate \eqref{eq:est_int_id_rhs_2}.  To that end, for all $\varepsilon>0$ we set
\[
\p M_{+, \varepsilon}=\{x\in \p M: \p_\nu \varphi(x)>\varepsilon\} 
\quad \text{and} \quad \Sigma_{+, \varepsilon}=(0, T)\times \p M_{+, \varepsilon}.
\]
We recall that in Section \ref{sec:intro} we defined the open sets $U',V' \subset \p M$ such that they contain the back and front faces $\p M_+, \p M_-$ of the manifold $M$, respectively. By the compactness of $\{x\in \p M: \p_\nu \varphi(x)=0\}$, there exists $ \varepsilon>0$ such that  $\Sigma \setminus V \subset \Sigma_{+, \varepsilon}$, where we had set $V=(0,T)\times V'$.

We utilize estimates \eqref{eq:estimate_v}, \eqref{eq:CGO_v_old}, \eqref{eq:est_remainder}, as well as the Cauchy-Schwartz inequality to get
\begin{align*}
\bigg|\int_{\Sigma\setminus V}\p_\nu u\overline{u_1}dS_g dt\bigg| 
&
\le \int_{\Sigma_{+, \varepsilon}} e^{-s(\beta t +x_1)} |\p_\nu u|  (|v_s|+|r_1|)dS_gdt
\\
&\le C \bigg(\int_{\Sigma_{+, \varepsilon}} |e^{-s(\beta t +x_1)} \p_\nu u |^2dS_gdt\bigg)^{1/2}(\|v_s\|_{L^2(\Sigma_{+,\varepsilon})}+\|r_1\|_{L^2(\Sigma)}).
\end{align*}

Next we estimate the terms in the inequality above. By Proposition \ref{prop:CGO_solution} and estimate \eqref{eq:est_remainder}, in conjunction with the inequalities
\[
\|r_1\|_{L^2(\Sigma)}\le \|r_1\|_{L^2(Q)}^{1/2}\|r_1\|_{H^1(Q)}^{1/2}  \quad \text{and} \quad  \|r_1\|_{H^1(Q)}\le Ch^{-1}\|r_1\|_{H^1_\scl(Q)},
\] 
we obtain
\begin{equation}
\label{eq:trace}
\|r_1\|_{L^2(\Sigma)} =o(h^{-1/2}), \quad h \to 0.
\end{equation}

We now follow the steps in the proof of \cite[Theorem 6.2]{Cekic} to verify
\begin{equation}
\label{eq:beam_est_boundary}
\|v_s\|_{L^2(\Sigma_{+,\varepsilon})}=\cO(1), \quad h \to 0.
\end{equation}
Due to the product structure of $\Sigma_{+,\varepsilon}$ and the fact that $T>0$ is finite, it suffices to prove that $\|v_s\|_{L^2(\p M_{+,\varepsilon})}=\cO(1)$. 

Let $\pi_1:\p M\to \R$ be a projection defined by
\[
\pi_1(x)=x_1 \quad \text{for any } x=(x_1, x')\in \p M.
\]
Without loss of generality, we assume that $\p M_\varepsilon:=\p M\cap \pi^{-1}_1(\varepsilon)$ is a manifold for any $\varepsilon>0$. 
Hence, $\p M_{+,\varepsilon}$ is a compact $(n-1)$-dimensional manifold with boundary.

First, we observe that $\p M_{+,\varepsilon} \subset \p M$ is an open and precompact manifold of dimension $n-1$, the same as $M_0$. Clearly, the projections $\pi_1 \colon \p M \to \R, \: \pi_1(x_1,x')=x_1$, and $\pi_2: \p M\to M_0, \: \pi_2(x_1,x')=x'=(x_2,\ldots,x_{n})$, are smooth. From here, our aim is to show that $\pi_2$ is a local diffeomorphism in $\p M_{+,\varepsilon}$. To accomplish this, we note that by definition, the vector field $\p x_1$ is transversal to $\p M$ on $\p M_{+,\varepsilon}$. Thus, if $z_2,\ldots, z_{n}$ are some local coordinates in $\p M_{+,\varepsilon}$, the functions $x_1,z_2,\ldots z_n$ form local coordinates in $\R \times M_0$ near $z_0$. Moreover, the map $x \mapsto (x_1,x')$ is a diffeomorphim.
Thus, the $(n-1)\times (n-1)$ matrix $\frac{\p x_\alpha}{\p z_\beta}$ for $\alpha,\beta= 2,\ldots, n$, which is also the differential of $\pi_2$, is invertible. By the inverse function theorem, the map $\pi_2$ is a local diffeomorphism in $\p M_{+,\varepsilon}$.

Let $x \in \p M_{+,\varepsilon}$ be an arbitrary point, and let $\mathcal{U}\subset \p M_{+,\varepsilon}$ be a neighborhood of $x$ such that $\pi_2|_{\mathcal{U}}$ is a diffeomorphism. Then it follows from the change of variables formula that the pullback of the surface elements satisfy $(\pi_2)^\ast(dS_g)=J_{\pi_2}dV_{g_0}$, where $J_{\pi_2}$ is the Jacobian of $\pi^{-1}_2$. Therefore, we have
\[
\int_{\mathcal{U}}|v_s|^2dS_g=\int_{\pi_2(\mathcal{U})}|v_s\circ \pi_2^{-1}|^2J_{\pi_2}dV_{g_0}.
\]
Furthermore, after possibly choosing a smaller set $\mathcal{U}$, we see that the Jacobian $J_{\pi_2}$ is bounded on $\pi_2(\mathcal{U}) \subset M_0$. Therefore, we deduce from \eqref{eq:L2_power_y} and \eqref{eq:bound_of_v} that
\[
\int_{\mathcal{U}}|v_s|^2dS_g=\int_{\pi_2(\mathcal{U})}|v_s\circ \pi_2^{-1}|^2J_{\pi_2}dV_{g_0}=\cO(1).
\]
Since $x \in \p M_{+,\varepsilon}$ was arbitrarily chosen, we can choose a larger $\varepsilon$ and obtain a finite cover for $\p M_{+,\varepsilon}$ consisting of the sets $\mathcal{U}$ as above by shrinking $\p M_{+,\varepsilon}$. This leads to estimate \eqref{eq:beam_est_boundary}.

Whence, estimates \eqref{eq:trace} and \eqref{eq:beam_est_boundary} yield
\begin{equation}
\label{eq:est_bdy_vpr}
\|v_s\|_{L^2(\Sigma_{+,\varepsilon})}+\|r_1\|_{L^2(\Sigma)}=o(h^{-1/2}), \quad h\to 0.
\end{equation}

On the other hand, we have
\begin{align*}
\bigg(\int_{\Sigma_{+, \varepsilon}}|\p_\nu u e^{-s(\beta t +x_1)}|^2dS_gdt\bigg)^{1/2} 
&= \frac{1}{\sqrt \varepsilon}\bigg(\int_{\Sigma_{+, \varepsilon}} \varepsilon |\p_\nu u e^{-s(\beta t +x_1)} |^2dS_gdt\bigg)^{1/2}
\\
&\le \frac{1}{\sqrt \varepsilon} \bigg(\int_{\Sigma_{+, \varepsilon}} \p_\nu \varphi |\p_\nu u e^{-s(\beta t +x_1)} |^2dS_gdt\bigg)^{1/2}
\\
&\le \frac{1}{\sqrt \varepsilon}\bigg(\int_{\Sigma_{+}} \p_\nu \varphi|\p_\nu u e^{-s(\beta t +x_1)} |^2dS_gdt \bigg)^{1/2},
\end{align*}
where  we used $\Sigma_+=(0,T)\times \p M_+^{\text{int}}$.

Using the boundary Carleman estimate \eqref{eq:Car_est_conjugated} and   \eqref{eq:ibvp_difference}, we get
\[
\bigg(\int_{\Sigma_{+}} \p_\nu \varphi|\p_\nu u e^{-s(\beta t +x_1)} |^2dS_gdt\bigg)^{1/2}
\le\mathcal{O}(h^{1/2})\frac{1}{\sqrt \varepsilon}\|e^{-s(\beta t +x_1)}\mathcal{L}_{c,g,a_1,q_1}u \|_{L^2(Q)}.
\]
Therefore, estimates \eqref{eq:bifurgation} and \eqref{eq:est_bdy_vpr} yield  \eqref{eq:est_int_id_rhs_2}. This completes the proof of Lemma \ref{lem:rhs_est}.
\end{proof}

\subsection{Uniqueness of the potential}

In this subsection we assume $a_1=a_2$ and prove that $\mathcal{C}_{g,a_1, q_1}=\mathcal{C}_{g,a_2,q_2}$ implies $q_1=q_2$. Our starting point is again the integral identity \eqref{eq:int_id_2}. When $a_1-a_2=a=0$, this reads 
\begin{equation}
\label{eq:int_id_q}
\int_Q qu_2\overline{u_1}dV_gdt= \int_M c^{-1}\p_t u(T, x)\overline{u_1(T, x)}dV_g-  \int_{\Sigma\setminus V}\p_\nu u\overline{u_1}dS_g dt.
\end{equation}

Since $a=0$, Lemma \ref{lem:rhs_est} implies that both terms on the right-hand side of \eqref{eq:int_id_q} vanish in the limit $h\to 0$. Therefore, we have
\[
\int_Q qu_2\overline{u_1}dV_gdt\to0,\quad h\to 0.
\]
On the other hand, by substituting the CGO solutions \eqref{eq:CGO_v_old} and \eqref{eq:CGO_u2_old} into the left-hand side of \eqref{eq:int_id_q}, we get
\[
\int_Q qu_2\overline{u_1}dV_gdt=\int_Q qe^{2i\lambda (\beta t+x_1)} \confm (\overline{v_s}w_s +\overline{v_s}r_2+ w_s\overline{r_1}+\overline{r_1}r_2)dV_gdt.
\]
It follows from estimates \eqref{eq:estimate_v}, \eqref{eq:estimate_w}, and \eqref{eq:est_remainder} that
\[
\int_Q q  e^{2i\lambda (\beta t+x_1)} \confm (\overline{v_s}r_2+w_s\overline{r_1}+\overline{r_1}r_2)dV_gdt=o(1),\quad h\to 0.
\]
Therefore, we obtain
\[
\int_Q q e^{2i\lambda (\beta t+x_1)} \confm \overline{v_s}w_s dV_gdt\to0, \quad h\to 0.
\]
By repeating the arguments leading from \eqref{eq:quasimode_zero} to \eqref{eq:cgo_sub}, with the assumptions that $q_1,q_2 \in C(\overline{Q})$ and $q_1=q_2$ on $\p Q$, we get
\begin{align*}
&\int_{0}^{\frac{L}{\sqrt{1-\beta^2}}}\int_\R \int_\R  e^{2i\lambda ((\beta^2-1) \tT+p)-2(1-\beta^2)\lambda r} (cq)(\beta \tT, p-\tT, \gamma(\sqrt{1-\beta^2}r))
\\
& \quad \quad \quad \quad \quad \quad \times e^{\overline{\Phi_1(\tT, p, r)}+\Phi_2(\tT, p, r)}\eta(\tT,p, r)dp d\tT dr=0.
\end{align*}
Then we follow the same arguments from \eqref{eq:cgo_sub} onward in the proof for the uniqueness of the damping coefficient to obtain $q_1=q_2$. This completes the proof of Theorem \ref{thm:main_result_damping}.

\bibliographystyle{abbrv}
\bibliography{bibliography_damping}

\end{document}